\documentclass[a4paper,10pt]{amsart}

\usepackage[english]{babel}
\usepackage{bm}

\let\WT=\widetilde

\usepackage{stmaryrd}
\usepackage{upgreek}

\usepackage[lmargin=1.2in,rmargin=1.2in,tmargin=1.3in,bmargin=1.1in]{geometry}

\usepackage{color}

\usepackage[utf8]{inputenc}
\usepackage{hyperref}
\usepackage{relsize}
\usepackage{amssymb,amsmath}
\usepackage{dsfont}
\usepackage[all]{xy}
\usepackage{mathtools}
\usepackage{mathrsfs}

\renewcommand{\theta}{\uptheta}
\renewcommand{\iota}{\upiota}
\renewcommand{\alpha}{\upalpha}
\renewcommand{\beta}{\upbeta}
\renewcommand{\gamma}{\upgamma}
\renewcommand{\delta}{\updelta}
\renewcommand{\zeta}{\upzeta}
\renewcommand{\pi}{\uppi\hspace{0.05em}}
\renewcommand{\xi}{\upxi}
\renewcommand{\chi}{\upchi}
\renewcommand{\sigma}{\upsigma}
\renewcommand{\Lambda}{\Uplambda}
\renewcommand{\Gamma}{\Upgamma}
\renewcommand{\phi}{\upphi}
\renewcommand{\nu}{\upnu}
\renewcommand{\tau}{\uptau}
\renewcommand{\mu}{\upmu}
\renewcommand{\eta}{\upeta}

\newtheorem{theorem}{Theorem}[section]
\newtheorem{thmx}{Theorem}

\newtheorem{thmdef}[theorem]{Theorem/Definition}
\newtheorem{proposition}[theorem]{Proposition}
\newtheorem{lemma}[theorem]{Lemma}
\newtheorem{conjecture}[theorem]{Conjecture}
\newtheorem{corollary}[theorem]{Corollary}

\theoremstyle{definition}
\newtheorem{definition}[theorem]{Definition}
\newtheorem{assumption}[theorem]{Assumption}

\theoremstyle{remark}
\newtheorem{example}[theorem]{Example}
\newtheorem{remark}[theorem]{Remark}

\newcommand{\dvs}{\mathbb{N}^{Q_0}}
\newcommand{\dvst}{\Lambda_{\theta}^{\zeta}}

\newcommand{\CC}{\mathbb{C}}
\newcommand{\SSN}{\mathcal{SSN}}
\DeclareMathOperator{\B}{B\!}

\newcommand{\SN}{\mathcal{SN}}
\DeclareMathOperator{\Yang}{\mathbf{Y}}

\newcommand{\LLL}{\mathscr{L}}

\newcommand{\prim}{\mathfrak{pr}}
\newcommand{\Prim}{\mathcal{P}r}
\newcommand{\kac}{\mathtt{a}}
\newcommand{\dd}{\mathbf{d}}

\newcommand{\ee}{\mathbf{e}}

\newcommand{\JH}{\mathtt{JH}}
\DeclareMathOperator{\CStab}{\mathcal{S}ta}
\DeclareMathOperator{\fstab}{\mathfrak{st}}
\DeclareMathOperator{\fisot}{\mathfrak{is}}
\DeclareMathOperator{\CIsot}{\mathcal{I}so}
\DeclareMathOperator{\lb}{lb}
\DeclareMathOperator{\bdd}{b}

\newcommand{\sssct}{\subsubsection{}}

\let \ol=\overline
\let \ul=\underline
\newcommand{\DbMMHM}{\mathcal{D}^{\lb}\!\MMHM}
\newcommand{\DfbMMHM}{\mathcal{D}^{\bdd}\!\MMHM}
\newcommand{\DfbMHM}{\mathcal{D}^{\bdd}\!\MHM}

\newcommand{\DbMHM}{\mathcal{D}^{\lb}\!\MHM}

\newcommand{\FX}{\mathfrak{X}}
\DeclareMathOperator{\sph}{\scriptscriptstyle{sph}}

\DeclareMathOperator{\DlMMHM}{\mathcal{D}\!\MMHM^{\shortrightarrow}}
\DeclareMathOperator{\DlMHM}{\mathcal{D}\!\MHM^{\shortrightarrow}}

\DeclareMathOperator{\DcMMHM}{\mathcal{D}\!\MMHM^{\shortleftarrow}}

\newcommand{\HA}{\mathcal{DT}}
\newcommand{\exte}{\mathtt{e}}
\newcommand{\ff}{\mathbf{f}}
\newcommand{\GG}{\mathbb{G}}

\newcommand{\fc}{\mathfrak{cu}}

\DeclareMathOperator{\udim}{\underline{dim}}

\DeclareMathOperator{\lP}{\mathfrak{L}}

\DeclareMathOperator{\mult}{\mathbf{m}}

\DeclareMathOperator{\TTr}{\mathcal{T}r}

\DeclareMathOperator{\opp}{op}

\newcommand{\HG}{\HO_G}

\newcommand{\nIC}{\underline{\ICst}}
\newcommand{\nnIC}{\WT{\ICst}}

\newcommand{\g}{\mathfrak{gl}}
\newcommand{\NN}{\mathbb{N}}

\newcommand{\QQ}{\mathbb{Q}}

\DeclareMathOperator{\TTTr}{\mathfrak{T}r}

\newcommand{\WW}{\mathcal{T}r(W)}
\newcommand{\WWW}{\mathfrak{Tr}(W)}

\newcommand{\XX}{\mathcal{X}}
\newcommand{\YY}{\mathcal{Y}}
\newcommand{\ZZ}{\mathbb{Z}}

\newcommand{\Mst}{\mathfrak{M}}

\newcommand{\Msp}{\mathcal{M}}

\newcommand{\ICS}{\mathcal{IC}}

\newcommand{\phim}[1]{\phi^{\mon}_{#1}}

\newcommand{\AAA}[1]{\mathbb{A}^{#1}}

\DeclareMathOperator{\mathG}{G}
\newcommand{\GQ}[1]{\mathG_{#1}}

\DeclareMathOperator{\image}{Image}
\DeclareMathOperator{\Imag}{im}
\DeclareMathOperator{\edge}{edge}
\DeclareMathOperator{\sst}{-ss}
\DeclareMathOperator{\ssst}{-ss^{\circ}}
\DeclareMathOperator{\stab}{-st}
\DeclareMathOperator{\sstab}{sst}

\DeclareMathOperator{\nilp}{nilp}
\DeclareMathOperator{\hype}{hyp}
\DeclareMathOperator{\isot}{iso}

\DeclareMathOperator{\crit}{crit}

\DeclareMathOperator{\cyc}{cyc}

\DeclareMathOperator{\Sp}{\mathcal{S}}
\DeclareMathOperator{\SP}{\mathcal{S}}

\DeclareMathOperator{\Hom}{Hom}

\DeclareMathOperator{\mon}{mon}
\DeclareMathOperator{\End}{End}
\DeclareMathOperator{\weight}{wt}
\DeclareMathOperator{\wt}{\chi_{\weight}}

\DeclareMathOperator{\Exp}{Exp}
\DeclareMathOperator{\simp}{simp}

\DeclareMathOperator{\AS}{\mathbb{A}}

\DeclareMathOperator{\MMHM}{MMHM}
\DeclareMathOperator{\MMHS}{MMHS}

\DeclareMathOperator{\ICst}{\mathcal{IC}}
\DeclareMathOperator{\Gr}{\mathbf{Gr}}

\DeclareMathOperator{\Vect}{Vect}

\DeclareMathOperator{\smooth}{\scriptscriptstyle{sm}}

\DeclareMathOperator{\Id}{Id}
\DeclareMathOperator{\rat}{\mathbf{rat}}

\DeclareMathOperator{\lmod}{-mod}

\DeclareMathOperator{\BM}{BM}
\DeclareMathOperator{\mathc}{c}
\newcommand{\HOc}{\HO_{\mathc}}

\newcommand{\HOBM}{\HO^{\BM}}
\newcommand{\fg}{\mathfrak{g}}
\newcommand{\fn}{\mathfrak{n}}

\DeclareMathOperator{\Nak}{\mathbf{M}}
\DeclareMathOperator{\Lus}{\mathbf{L}}

\DeclareMathOperator{\supp}{supp}

\DeclareMathOperator{\Perv}{Perv}
\DeclareMathOperator{\MHM}{MHM}

\DeclareMathOperator{\Sym}{Sym}

\DeclareMathOperator{\red}{\scriptscriptstyle{red}}

\DeclareMathOperator{\Gl}{GL}

\DeclareMathOperator{\gl}{\mathfrak{gl}}
\DeclareMathOperator{\Sl}{SL}

\DeclareMathOperator{\UEA}{\mathbf{U}}

\DeclareMathOperator{\id}{id}
\DeclareMathOperator{\Jac}{Jac}
\DeclareMathOperator{\Symp}{Sp}

\DeclareMathOperator{\Tr}{Tr}

\DeclareMathOperator{\pt}{pt}
\DeclareMathOperator{\Tot}{Tot}

\DeclareMathOperator{\cone}{cone}

\DeclareMathOperator{\vir}{vir}
\DeclareMathOperator{\Ob}{Ob}
\DeclareMathOperator{\Ho}{\mathcal{H}}
\DeclareMathOperator{\HO}{\mathbf{H}}

\DeclareMathOperator{\Coha}{\mathcal{A}}
\DeclareMathOperator{\HCoha}{\HO\!\mathcal{A}}
\DeclareMathOperator{\rCoha}{\mathcal{RA}}

\DeclareMathOperator{\CCCoh}{\mathcal{C}oh}
\DeclareMathOperator{\NS}{NS}
\DeclareMathOperator{\HHiggs}{\mathfrak{H}iggs}
\DeclareMathOperator{\CHiggs}{\mathcal{H}iggs}
\DeclareMathOperator{\TS}{\mathtt{TS}}

\DeclareMathOperator{\Wt}{W}
\DeclareMathOperator{\Lie}{Lie}
\DeclareMathOperator{\Betti}{\scriptscriptstyle{Betti}}
\DeclareMathOperator{\Borch}{\mathbf{Bor}}
\DeclareMathOperator{\Cusp}{\mathcal{C}\!\mathit{u}}

\DeclareMathOperator{\Mall}{\mathcal{N}}
\DeclareMathOperator{\Coh}{Coh}
\newcommand{\HMall}{\HO\Mall}
\DeclareMathOperator{\BPS}{BPS}
\DeclareMathOperator{\CCoh}{\mathfrak{C}oh}
\newcommand{\cdotsh}{\!\cdot\!}
\newcommand{\BPSh}{\mathcal{BPS}}

\newcommand{\Db}{\mathcal{D}^{b}}

\newcommand{\VD}{\mathbb{D}}
\newcommand{\Dmon}{\mathbb{D}^{\mon}}

\DeclareMathOperator{\reel}{real}
\title[BPS Lie algebras and less perverse filtrations]{BPS Lie algebras and the less perverse filtration on the preprojective CoHA}
\author{Ben Davison}
\begin{document}

\begin{abstract}
The affinization morphism for the stack $\Mst(\Pi_Q)$ of representations of a preprojective algebra $\Pi_Q$ is a local model for the morphism from the stack of objects in a general 2-Calabi--Yau category to the good moduli space.  We show that the derived direct image of the dualizing complex along this morphism is pure, and admits a decomposition in the sense of the Beilinson--Bernstein--Deligne--Gabber decomposition theorem.  

We introduce a new perverse filtration on the Borel--Moore homology of $\Mst(\Pi_Q)$, using this decomposition.  We show that the zeroth piece of the resulting filtration on the cohomological Hall algebra built out of the Borel--Moore homology of $\Mst(\Pi_Q)$ is isomorphic to the universal enveloping algebra of an associated BPS Lie algebra $\mathfrak{g}_{\Pi_Q}$.  This Lie algebra is defined via the Kontsevich--Soibelman theory of critical cohomological Hall algebras for 3-Calabi--Yau categories.  We then lift this Lie algebra to a Lie algebra object in the category of perverse sheaves on the coarse moduli space of $\Pi_Q$-modules, and use this algebra structure to prove results about the summands appearing in the above decomposition theorem.  In particular, we prove that the intersection cohomology of singular spaces of semistable $\Pi_Q$-modules provide ``cuspidal cohomology'' -- a conjecturally complete subspace of canonical generators for $\mathfrak{g}_{\Pi_Q}$.
\end{abstract}
\maketitle
\setcounter{tocdepth}{1}

%\tableofcontents
\section{Introduction}
\subsection{Background and motivation}
Let $p \colon X\rightarrow Y$ be a projective morphism of complex varieties, and assume that $X$ is smooth.  The decomposition theorem of Beilinson, Bernstein, Deligne and Gabber \cite{BBD} states that there is an isomorphism 
\begin{equation}
\label{decompp1}
p_*\QQ_X\cong \bigoplus_{n\in\ZZ}{}^{\mathfrak{p}}\!\Ho^n(p_*\QQ_X)[-n]
\end{equation}
where $p_*\QQ_X$ denotes the derived direct image of the constant sheaf with $\QQ$ coefficients.  The theorem moreover states that each perverse sheaf ${}^{\mathfrak{p}}\!\Ho^n(p_*\QQ_X)$ appearing on the right hand side of \eqref{decompp1} admits a canonical decomposition
\begin{equation}
\label{decompp2}
{}^{\mathfrak{p}}\!\Ho^n(p_*\QQ_X)\cong \bigoplus_{i\in S_n}\ICS_{\ol{Z_i}}(\mathcal{F}_i)
\end{equation}
where $Z_i\subset Y$ are locally closed smooth connected subvarieties, $\mathcal{F}_i$ are semisimple local systems on them, and $\ICS_{\ol{Z_i}}(\mathcal{F}_i)$ is the intermediate extension of $\mathcal{F}_i$ to $Y$.  While the first decomposition is not in general canonical, the resulting inclusions of graded vector spaces 
\begin{equation}
\label{Pfiltgd}
\mathfrak{P}_{\leq m}=\HO\left(Y,\bigoplus_{n\leq m}{}^{\mathfrak{p}}\!\Ho^n(p_*\QQ_X)[-n]\right)\hookrightarrow \HO\left(Y,\bigoplus_{n\in \ZZ}{}^{\mathfrak{p}}\!\Ho^n(p_*\QQ_X)[-n]\right)\cong \HO(X,\QQ)
\end{equation}
\textit{are} canonical, and define the \textit{perverse filtration} of $\HO(X,\QQ)$ with respect to $p$.  Perverse filtrations constructed this way have played an extremely important role in nonabelian Hodge theory (giving the ``P'' in the famous P=W conjecture\footnote{Since the first version of this paper appeared, this has become a theorem: see \cite{MaSh22,HMMS22,MSY23b} for different proofs.}) \cite{ngo2006fibration,ngo2010lemme, deCat12,deC19, deC20, CHS20, felisetti2022p} and algebraic geometry \cite{HLSY21,ShenYin22}.  Additionally, perverse filtrations and decompositions of direct images as in \eqref{decompp2} have played a central role in various studies in geometric representation theory: see \cite{chriss2009representation,hotta2007d,achar2021perverse} for representative background.  Under favourable circumstances, the second decomposition \eqref{decompp2} also plays a crucial role in understanding the topology of $X$, since it offers the means to study the complicated direct image $p_*\QQ_X$ via the restriction to the possibly nicer locus $Z_i$.  For instance, this is precisely the strategy offered by Ng\^{o}'s support theorem \cite{ngo2010lemme}, which states that for morphisms given by weak abelian fibrations, each of the support varieties $Z_i$ appearing in \eqref{decompp2} is dense.  Such fibrations play a very prominent role in nonabelian Hodge theory and the study of moduli spaces of sheaves on surfaces, see e.g. \cite{MS23} for applications and further developments.

Let $Q$ be an arbitrary quiver, let $\overline{Q}$ be its double, and let $\Pi_Q\coloneqq\CC \ol{Q}/\langle \sum_{a\in Q_1}[a,a^*]\rangle$ be the associated preprojective algebra.  The morphism $\JH\colon \Mst(\Pi_Q)\rightarrow \Msp(\Pi_Q)$ from the stack of $\Pi_Q$-modules to the coarse moduli space provides the \'etale local model for the morphism from the stack of objects in a 2-Calabi--Yau category to its good moduli space; we come to our motivating examples of such categories shortly.    For now we remark that the category of semistable Higgs bundles of fixed slope provides one such category, and also that the morphism $\JH$ is closely related to the study of Nakajima quiver varieties, a central object in geometric representation theory.  Despite these relations to areas of mathematics referenced above, $\JH$ is a long way from satisfying the assumptions of the BBDG decomposition theorem: most importantly, $\Mst(\Pi_Q)$ is in general highly singular.   We nonetheless extend the decomposition theorem to $\JH$, providing versions of the decompositions \eqref{decompp1} and \eqref{decompp2} for the direct image of the dualising complex\footnote{Or, equivalently, the direct image with compact supports of $\QQ_{\Mst(\Pi_Q)}$.} along $\JH$ (see Theorem \ref{thma} for a precise statement).

We find that, in contrast to Ng\^{o}'s result, the subvarieties $Z_i$ appearing in our analogue of \eqref{decompp2} are almost never dense in the base.  In this paper we will use a  sheaf-theoretic lift of cohomological Hall algebra structures to describe the summands that do not have full support, and write a precise conjectural description of the direct image $\JH_*\VD\ul{\QQ}_{\Mst(\Pi_Q)}$ in terms of convolution tensor products of shifted intersection complexes (Conjecture \ref{mainConj}).  The derived global sections of the summands in the decomposition \eqref{decompp1} will turn out to be tightly connected to the theory of BPS Lie algebras, certain Lie algebras that arise in the study of cohomological Hall algebras and algebras of BPS states \cite{HM98,KS2,Da13,QEAs}: see Theorem \ref{newthmB} for the main result in this direction.

The morphisms $\JH$, as we vary the quiver $Q$, provide the local model of the morphism from the stack of objects in a 2-Calabi--Yau category $\mathscr{C}$ to the good moduli space of objects in $\mathscr{C}$, enabling us to study the \textit{global} topology of stacks of objects in general 2-Calabi--Yau categories.  In this paper we mainly concentrate on the decomposition theorem and its implications for BPS Lie algebras of preprojective algebras, in essence the \textit{local} model for 2CY categories.  Further examples and applications of the global theory, which are explored in other papers building on the results here, include
\begin{enumerate}
\item
(Nonabelian Hodge theory): Categories of semistable Higgs bundles on a smooth projective curve of fixed slope, and the categories of representations of the fundamental group of the once-punctured Riemann surface, form two sides of the nonabelian Hodge correspondence.  The decomposition theorem we prove is an essential tool in extending the nonabelian Hodge isomorphism to the Borel--Moore homology of singular stacks appearing on these two sides of nonabelian Hodge theory \cite{DHSM22, Henn23}, and formulating a version of the P=W conjecture in this generality \cite{Dav21c}.
\item
(Algebraic geometry): The present paper was motivated in large part by the study of mixed Hodge structures on Borel--Moore homology of stacks of coherent sheaves on K3 surfaces, and especially the paper \cite{HL08}, in which it was first suggested that results in cohomological Donaldson--Thomas theory could be used to prove new purity results for moduli stacks of coherent sheaves on K3 surfaces.  We discuss this application in greater depth in \S \ref{HLC_sec}.
\item
(Geometric representation theory):  As mentioned above, the most natural way to explain and control the failure of density of the supports $Z_i$ in the decomposition theorem is via Hall algebras.  Moreover, the ``local'' geometry we consider, that of the morphism $\JH$, is very close to the affinization morphism for Nakajima quiver varieties, an extremely fertile ground for studying geometric representation theory.  In a subsequent paper with Hennecart, Schlegel Mejia \cite{DHSM23}, we completely describe the decompositions appearing in this paper, and prove conjectures of Bozec and Schiffmann \cite{BoSch19} regarding cuspidal polynomials, as well as Okounkov's conjecture relating Kac polynomials to Maulik--Okounkov Yangians \cite{botta2023okounkov} constructed from the cohomology of Nakajima quiver varieties and stable envelopes \cite{MO19}.
\end{enumerate}

\subsection{Main results}
\label{mrs}
Let $\ol{Q}$ be the double of a quiver $Q$, obtained by adding an arrow $a^*$ to $Q$ for every $a$ an arrow of $Q$,  with the opposite orientation to that of $a$.  Let $\Pi_Q\coloneqq\CC \ol{Q}/\langle \sum_{a\in Q_1}[a,a^*]\rangle$ be the preprojective algebra associated to $Q$.  Let $\SP$ be a Serre subcategory of the category of $\CC\ol{Q}$-modules.  We set 
\[
\HCoha^{\SP}_{\Pi_Q}\coloneqq \bigoplus_{\dd\in\dvs}\HOBM\!\left( \Mst^{\SP}_{\dd}(\Pi_Q),\QQ\right)\otimes \LLL^{-\chi_{Q}(\dd,\dd)},
\]
the (shifted) Borel--Moore homology of the stack $\Mst^{\SP}(\Pi_Q)$ of finite-dimensional $\Pi_Q$-modules which are objects of $\SP$.  Here $\LLL=\HOc(\AAA{1},\QQ)$ is a Tate twist, which is introduced so that the object $\HCoha^{\SP}_{\Pi_Q}$ carries an associative multiplication, (see \S \ref{CCSec} for the definition).  The resulting algebra plays a key role in geometric representation theory; it is the algebra of all conceivable raising operators on the cohomology of Nakajima's quiver varieties, and so via several decades of work \cite{Nak98, Nak94, Groj95, Var00, ScVa13, MO19} contains half of various quantum groups associated to $Q$.

Let $\JH\colon \Mst(\Pi_Q)\rightarrow \Msp(\Pi_Q)$ be, as above, the semisimplification (equivalently, affinization) morphism to the coarse moduli space of $\Pi_Q$-modules.  We study $\HCoha^{\SP}_{\Pi_Q}$ via the richer object
\[
\rCoha_{\Pi_Q}\coloneqq \bigoplus_{\dd\in\dvs} \JH_*\VD\ul{\QQ}_{\Mst(\Pi_Q)}\otimes \LLL^{-\chi_Q(\dd,\dd)},
\]
the derived direct image of the dualizing mixed Hodge module.  The derived category of mixed Hodge modules on $\Msp(\Pi_Q)$ is a tensor category via convolution along the direct sum map, and we may consider $\rCoha_{\Pi_Q}$ as an algebra object in this category, from which we recover $\HCoha^{\SP}_{\Pi_Q}$ by restricting to $\Msp^{\SP}(\Pi_Q)$ and taking hypercohomology.
\begin{thmx}[Corollary \ref{relPurity}\footnote{In the interests of digestibility, in the introduction we state all results without reference to extra gauge groups $G$, stability conditions or slopes.  The results in the main body incorporate these generalisations.}]
\label{thma}
There is an isomorphism of complexes of mixed Hodge modules
\begin{equation}
\label{mainDecomp}
\rCoha_{\Pi_Q}\cong\bigoplus_{n\in2\cdot\ZZ_{\geq 0}}\Ho^n(\rCoha_{\Pi_Q})[-n]
\end{equation}
and each $\Ho^n(\rCoha_{\Pi_Q})$ is pure of weight $n$, i.e. $\rCoha_{\Pi_Q}$ is \textbf{pure}.  As a consequence, we may write
\begin{equation}
\label{decomppp2}
\Ho^n(\rCoha_{\Pi_Q})\cong \bigoplus_{i\in S_n}\ICS_{\ol{Z_i}}(\underline{\mathcal{F}}_i)
\end{equation}
where $Z_i\subset \Msp(\Pi_Q)$ are locally closed smooth connected subvarieties, and $\ICS_{\ol{Z_i}}(\underline{\mathcal{F}}_i)$ is the intermediate extension of a simple pure weight $n$ variation of Hodge structure on $Z_i$.
\end{thmx}
The theorem contains the statement that the derived direct image with compact support $\JH_!\ul{\QQ}_{\Mst(\Pi_Q)}$ is pure, i.e. this complex satisfies the statement of the decomposition theorem of Beilinson, Bernstein, Deligne and Gabber \cite{BBD}, or more precisely, Saito's version in the language of mixed Hodge modules \cite{Saito88,Saito90}.  Note that the functor taking a complex of mixed Hodge modules to its underlying complex of perverse sheaves is exact, so that the cohomology functors $\Ho^n$ are the lifts of the cohomology functors ${}^{\mathfrak{p}}\!\Ho^n$ appearing in \eqref{decompp1}.%  This is rather surprising, since the preconditions of that theorem are not met; $\Mst(\Pi_Q)$ is a highly singular stack, and $p$ is not projective.  

\sssct
As a result of the decomposition \eqref{mainDecomp}, the mixed Hodge structure $\HCoha^{\SP}_{\Pi_Q}$ carries an ascending perverse filtration $\lP_{\bullet}\!\HCoha^{\SP}_{\Pi_Q}$ defined as in \eqref{Pfiltgd}, starting in degree zero.  This filtration is defined by setting
\[
\lP_{i}\!\HCoha^{\SP}_{\Pi_Q}\coloneqq \HO\left(\Msp^{\Sp}(\Pi_Q),\varpi'^*\bm{\tau}^{\leq i}\rCoha_{\Pi_Q}\right),
\]
where $\varpi'\colon \Msp^{\Sp}(\Pi_Q)\hookrightarrow \Msp(\Pi_Q)$ is the inclusion, and $\bm{\tau}^{\leq i}$ is the $i$th truncation functor for the natural t structure on the category of complexes of mixed Hodge modules on $\Msp(\Pi_Q)$ (see \S \ref{MMHMs} for definitions).  As mentioned, the mixed Hodge structure $\HCoha^{\SP}_{\Pi_Q}$  carries an associative algebra structure, and it is our strategy to use the compatibility of this algebra structure with the perverse filtration, obtained by lifting the algebra structure to an algebra structure at the level of mixed Hodge modules, to investigate the summands appearing in the decomposition \eqref{decomppp2}.  This filtration also offers insight into the algebra $\HCoha^{\SP}_{\Pi_Q}$ itself, reflected in our second main theorem:
\begin{thmx}[Proposition \ref{respProp}, Theorem \ref{UEAthm}]
\label{newthmB}
The filtration $\lP_{\bullet}\!\HCoha^{\SP}_{\Pi_Q}$ is respected by the algebra structure on $\HCoha^{\SP}_{\Pi_Q}$.  Moreover, there is an isomorphism\footnote{Note that the algebra structure on the domain of this ismorphism involves a sign twist, which we recall in \S \ref{psitwists}.} of algebras $\lP_{0}\!\HCoha^{\SP}_{\Pi_Q}\cong \UEA(\fg^{\SP}_{\Pi_Q})$ where $\fg^{\SP}_{\Pi_Q}$ is isomorphic to the \textbf{BPS Lie algebra} \cite{QEAs} $\fg_{\tilde{Q},\tilde{W}}^{\tilde{\Sp}}$ determined by the tripled quiver $\WT{Q}$, cubic potential $\WT{W}$ and Serre subcategory $\WT{\SP}$ of the category of $\CC\WT{Q}$-modules defined in \S \ref{3desc}.
\end{thmx}
The construction of the BPS Lie algebra for arbitrary symmetric quivers with potential is recalled in \S \ref{PBWsec}.  These were introduced in joint work with Sven Meinhardt in \cite{QEAs}, as part of a project to realise the cohomological Hall algebras defined by Kontsevich and Soibelman \cite{KS2} as positive halves of generalised Yangians.  Note that the BPS Lie algebra is defined by a quite different perverse filtration, on vanishing cycle cohomology of a different Calabi--Yau category.  

Outside of cases with zero Lie bracket, no examples of BPS Lie algebras have been calculated before this paper; we start to remedy this situation here.  We identify the zeroth cohomologically graded piece of $\fg_{\Pi_Q}^{\SP}$ (for various choices of $\Sp$) with various variants of the Kac--Moody Lie algebras associated to the underlying graph of $Q$ (see \S \ref{zpf}).  In particular, we provide the first examples of nonabelian BPS Lie algebras.

\sssct One motivation for proving the purity statement in Theorem \ref{thma} is the study of generalised Yangians in the work of Maulik and Okounkov \cite{MO19}.  Conjecturally\footnote{After the first version of this paper appeared this conjecture was proved; see \cite{botta2023okounkov,SV23O}.}, at least after extending scalars by some ring $\mathbb{K}$, there is an isomorphism between the algebra\footnote{We will adopt the convention throughout that where an expected $\SP$ superscript is missing, we assume that $\SP$ is the whole category $\CC\ol{Q}\lmod$.} $\HCoha_{\Pi_Q}$ and the positive half of their Yangian $\Yang_{\mathtt{MO},Q}$.  On the other hand, the Lie algebra $\mathfrak{g}_{\mathtt{MO},Q}$ that (along with the action of tautological classes) generates $\Yang_{\mathtt{MO},Q}$ is given by a subspace of the cohomology of certain Nakajima quiver varieties.  By Saito's version of the decomposition theorem, the direct image of the constant sheaf on Nakajima quiver varieties to the coarse moduli space is pure, hence the expected purity of Theorem \ref{thma}.  One may interpret the theorem as evidence for the conjecture that $\HCoha_{\Pi_Q}\otimes\mathbb{K}\cong\Yang_{\mathtt{MO},Q}^+$, and as evidence for Okounkov's conjecture\footnote{Similarly, this is no longer a conjecture, see again \cite{botta2023okounkov,SV23O}.} that the graded dimensions of $\mathfrak{g}_{\mathtt{MO},Q}$ are given by the coefficients of Kac polynomials.

\subsection{Cuspidal cohomology}
For general $Q$ and $\dd\in\mathbb{N}^{Q_0}$, the $\dd$th graded piece of the BPS Lie algebra $\fg_{\Pi_Q}$ satisfies the condition on the dimensions of the cohomologically graded pieces
\[
\sum_{n\in\mathbb{Z}}\dim(\HO^n(\fg_{\Pi_Q,\dd}))q^{n/2}=\kac_{Q,\dd}(q^{-1})
\]
where the polynomials on the right hand side are the polynomials introduced by Victor Kac in \cite{Kac83}, counting $\dd$-dimensional absolutely indecomposable $Q$-representations over a finite field of order q.  A conjecture of Bozec and Schiffmann \cite[Conj.1.3]{BoSch19} states that the Kac polynomials $\kac_{Q,\dd}(q^{-1})$ are the characteristic functions of the $\dvs$-graded pieces of a cohomologically graded Borcherds algebra, and so it is natural to suspect that $\fg_{\Pi_Q}$ itself is the positive half of a cohomologically graded Borcherds algebra.  In particular, $\fg_{\Pi_Q}$ should be given by some cohomologically graded Cartan datum, including the data of (usually infinitely many) imaginary simple roots.  This paper initiates the search for these imaginary simple roots; it is the shadow, under the derived global sections functor, of the search for new summands in the decomposition \eqref{decomppp2}.

\sssct One of the motivations for pursuing a lift of the BPS Lie algebra to the category of mixed Hodge modules is a question of Schiffmann \cite{SchICM}, itself a quiver analogue of an older question of Deligne \cite{De15}: 
{\begin{center}
Is there any geometric description of the Cartan datum, for example some algebraic variety $\Msp_{\mathrm{cusp},\dd}(\Pi_Q)$ along with a natural embedding $\Psi\colon\HO(\Msp_{\mathrm{cusp},\dd}(\Pi_Q),\QQ)\hookrightarrow \fg_{\Pi_Q,\dd}$ as the space of imaginary simple roots of weight $\dd$? \end{center}
}
\smallbreak
Such a construction would answer in the affirmative the complex geometric analogue of Conjecture 3.5 of \cite{SchICM}.  We can make progress on the above question even without proving that $\fg_{\Pi_Q}$ is the positive part of a Borcherds algebra.  Precisely, we do so via (the special case $\SP=\CC\ol{Q}\lmod$ of) our more general theorem on primitive generators:
\begin{thmx}[Theorem \ref{thmBdone}]
\label{mainThmB}
Let $\dd$ be such that there exists a simple $\dd$-dimensional $\Pi_Q$-module, let $\varpi'\colon \Msp^{\SP}_{\dd}(\Pi_Q)\hookrightarrow \Msp_{\dd}(\Pi_Q)$ be the inclusion, and set
\[
\fc_{\Pi_Q,\dd}^{\SP}\coloneqq \HO\!\left(\Msp^{\SP}_{\dd}(\Pi_Q),\varpi'^!\ICS_{\Msp_{\dd}(\Pi_Q)}(\QQ)\right)\otimes\LLL^{1+\chi_Q(\dd,\dd)}.
\]
There is a canonical decomposition $\fg^{\SP}_{\Pi_Q,\dd}\cong \fc^{\Sp}_{\Pi_Q,\dd}\oplus \mathfrak{l}$ of mixed Hodge structures, such that the Lie bracket $\fg^{\SP}_{\Pi_Q,\dd'}\otimes\fg^{\SP}_{\Pi_Q,\dd''}\xrightarrow{[\cdot,\cdot]}\fg^{\SP}_{\Pi_Q,\dd}$ for $\dd'+\dd''=\dd$ factors through the inclusion of $\mathfrak{l}$.  In particular, the mixed Hodge structures $\fc^{\Sp}_{\Pi_Q,\dd}$ give a collection of \textbf{canonical} subspaces of generators for $\fg^{\SP}_{\Pi_Q}$.
\end{thmx}
The proof of the above theorem uses the construction of the new ``less''\footnote{See \S \ref{morePerverse} for an explanation of the name.} perverse filtration on $\HCoha_{\Pi_Q}^{\SP}$ arising from Theorem \ref{thma}, and the resulting lift of the Lie algebra $\fg_{\Pi_Q}$ to a Lie algebra object in the category of pure Hodge modules on $\Msp(\Pi_Q)$.  In particular, the decomposition into generators and non-generators in the BPS Lie algebra arises from the decomposition theorem for perverse sheaves/mixed Hodge modules.  We conjecture\footnote{Since this paper first appeared, this conjecture has become a theorem, see \cite{DHSM23}.} that aside from the known simple roots of $\dvs$-degree $1_i$ for $i$ a vertex of $Q$, or $\dd$ such that $\chi_Q(\dd,\dd)=0$, these are \textit{all} of the generators; see Conjecture \ref{mainConj} for the precise statement.  

\sssct Since by the decomposition theorem there is a canonical embedding 
\[
\HO\!\left(\Msp_{\dd}(\Pi_Q),\ICS_{\Msp_{\dd}(\Pi_Q)}(\QQ)\right)\subset \HO\!\left(X,\QQ\right)
\]
where $X\rightarrow \Msp_{\dd}(\Pi_Q)$ is a semi-small resolution, Theorem \ref{mainThmB} suggests that the answer to the question above is ``yes'', and the above embedding provides a route towards \cite[Conj.3.5]{SchICM}.  For example, any symplectic\footnote{See \cite{BeSch} for a comprehensive treatment of when we may expect to find such a resolution.} resolution is a semi-small \cite{Kal09} resolution of singularities, and thus its cohomology contains ``cuspidal'' cohomology as a canonical summand.

\subsection{Halpern--Leistner's conjecture}
\label{HLC_sec}
Our purity theorem is independent from the statement (proved in \cite{preproj}) that the mixed Hodge structure on $\HCoha_{\Pi_Q}$ is pure.  We explain the particular utility of the purity statement of the current paper, with reference to a particular application: the proof of a conjecture of Halpern-Leistner \cite{HL08}.

Let $X$ be a K3 surface, fix a generic ample class $H\in\NS(X)_{\QQ}$, and fix a Hilbert polynomial $P(t)$.  Then there is a moduli stack $\CCoh_{P(t)}^H(X)$ of $H$-semistable coherent sheaves with Hilbert polynomial $P(t)$, and Halpern-Leistner conjectures that the mixed Hodge structure on $\HOBM(\CCoh_{P(t)}^H(X),\QQ)$ is pure.  The above-mentioned purity result of \cite{preproj} encouraged this statement, while the purity result of the current paper provides the means to prove it.  The idea of the proof is that locally, the morphism $p\colon\CCoh_{P(t)}^H(X)\rightarrow \CCCoh_{P(t)}^H(X)$ to the coarse moduli space is modelled as the morphism $\Mst_{\dd}(\Pi_Q)\rightarrow \Msp_{\dd}(\Pi_Q)$ for some quiver $Q$, and so Theorem \ref{thma} tells us that the direct image $p_!\ul{\QQ}_{\Mst_{\dd}(\Pi_Q)}$ is locally, and hence globally, pure.  The result then follows from the fact that the direct image of a pure complex of mixed Hodge modules along a projective morphism is pure: Saito's version \cite{Saito90} of the decomposition theorem.  Full details of this proof, and other global applications of the results of the current paper, are worked out in \cite{Da21a}.

\subsection{The algebras $\UEA(\fg_{C})$ and $\UEA(\fg_{\Sigma_g})$}
The construction and results of the present paper can be applied in nonabelian Hodge theory, and the construction of ``higher genus'' BPS lie algebras, since they concern any category for which the moduli of objects is locally modeled by moduli stacks of modules for preprojective algebras.  Here we briefly explain one such application, referring the reader to \cite{Da21a,DHSM22} for further details.  

\sssct Let $C$ be a smooth genus $g$ complex projective curve, which for ease of exposition we assume to be defined over $\mathbb{Z}$, and let $\HHiggs^{\sstab}_{r,0}(C)$ denote the complex algebraic stack of semistable rank $r$ degree zero Higgs bundles on $C$. By \cite{MoSc20} there is an equality
\begin{align}
\label{MSid}
&\sum_{r\geq 0, i,n\in\ZZ}\dim(\Gr^W_n\!(\HOBM_{-i}(\HHiggs^{\sstab}_{r,0}(C),\QQ)))(-1)^iq^{n/2+(g-1)r^2}T^r\\&=\Exp_{q^{1/2},T}\left(\sum_{r\geq 1}\Omega_{C,r,0}(q^{1/2})(1-q)^{-1}T^r\right)\nonumber
\end{align}
where $\Omega_{C,r,0}(q^{1/2})=\kac_{C,r,0}(q^{1/2},\ldots,q^{1/2})$ is a specialisation of Schiffmann's polynomial, counting absolutely indecomposable vector bundles of rank $r$ on $C$ over $\mathbb{F}_q$.  On the right hand side we have taken the plethystic exponential, an operation which satisfies the identity
\begin{equation}
\label{plethy}
\Exp\left(\sum_{r,i\in \mathbb{Z}_{>0}\times\mathbb{Z}}(-1)^i\dim(\fg_{r,i})q^{i/2}T^r\right)=\sum_{r,i\in \mathbb{Z}_{>0}\times\mathbb{Z}}(-1)^i\dim (\UEA(\fg)_{r,i})q^{i/2}T^r
\end{equation}
for $\fg$ any $\mathbb{Z}_{> 0}\times\mathbb{Z}$-graded Lie algebra with finite-dimensional graded pieces.  We presume that the second $\ZZ$-grading agrees with the cohomological grading, so that the Koszul sign rule is in effect with respect to it, e.g.
\[
[a,b]=(-1)^{\lvert a\lvert \lvert b\lvert+1}[b,a]
\]
for $\lvert a\lvert$ and $\lvert b\lvert$ the $\ZZ$-degrees of $a$ and $b$ respectively.  This explains the introduction of the signs in \eqref{plethy}.  The Borel--Moore homology of $\HHiggs^{\sstab}_{r,0}(C)$ is pure (see \cite[Sec.7.4]{Da21a}), so that the only terms that contribute on the left hand side of \eqref{MSid} have $n=i$.  
\sssct
After staring at \eqref{MSid} and \eqref{plethy} together, it is natural to conjecture (as in \cite{SchICM}) that there is some Lie algebra $\fg_C$, and an isomorphism of bigraded Hodge structures
\[
\mathcal{H}^{\mathrm{Higgs}}_{C}\coloneqq \bigoplus_{r\geq 0}\HOBM(\HHiggs_{r,0}^{\sstab}(C),\QQ)\otimes \LLL^{(g-1)r^2}\cong \UEA_q(\fg_C[u])
\]
where the right hand side is the universal enveloping algebra of a current algebra for some Lie algebra $\fg_C$, which should be a ``curve'' cousin of the Kac--Moody Lie algebras associated to quivers.  This Lie algebra should be defined as the BPS Lie algebra associated to the non-compact Calabi--Yau threefold $Y=\Tot_C(\omega_C\oplus\mathcal{O}_C)$.  Technically, this presents some well-known complications: stacks of coherent sheaves on $Y$ do not have a global critical locus description, so that the definition of vanishing cycle sheaves on them requires a certain amount of extra machinery (see \cite{Jo15,BBBD}).  
\sssct
As a consequence of the results of this paper there is a ``less perverse'' definition of $\UEA(\fg_C)$ ready off the shelf, avoiding d critical structures, vanishing cycles etc.: we may \textit{define}
\[
\UEA(\fg_C)\coloneqq \bigoplus_{r\geq 0}\HO(\CHiggs_{r,0}^{\sstab}(C),\bm{\tau}^{\leq 0}p_*\VD\ul{\QQ}_{\HHiggs_{r,0}^{\sstab}(C)})\otimes\LLL^{(g-1)r^2}
\]
where $p\colon \HHiggs_{r,0}^{\sstab}(C)\rightarrow \CHiggs_{r,0}^{\sstab}(C)$ is the morphism to the coarse moduli space, and the multiplication is via the correspondences in the CoHA of Higgs sheaves as in \cite{SS20,Mi20}.  Similarly, we define 
\begin{align*}
\UEA(\fg^{\nilp}_C)\coloneqq &\bigoplus_{r\geq 0}\HO(\CHiggs_{r,0}^{\sstab}(C),g_*g^!\bm{\tau}^{\leq 0}p_*\VD\ul{\QQ}_{\HHiggs_{r,0}^{\sstab}(C)})\otimes\LLL^{(g-1)r^2}\\ &\subset \mathcal{H}^{\mathrm{Higgs},\nilp}_{C}\coloneqq \bigoplus_{r\geq 0}\HOBM(\HHiggs_{r,0}^{\sstab,\nilp}(C),\QQ)\otimes \LLL^{(g-1)r^2}
\end{align*}
where $g\colon \CHiggs_{r,0}^{\sstab,\nilp}(C)\rightarrow \CHiggs_{r,0}^{\sstab}(C)$ is the inclusion of the locus for which the Higgs field is nilpotent, to define the correct enveloping algebra inside the CoHA of nilpotent Higgs bundles \cite{SS20}.  

\sssct
Whenever the morphism $p$ from the stack of objects in a category $\mathscr{C}$ to the coarse moduli space is locally modeled as the semisimplification morphism from the stack of representations of a preprojective algebra, the definition of the enveloping algebra of the BPS Lie algebra for $\mathscr{C}$ is forced by Theorem \ref{thma}; we likewise define
\begin{align*}
\UEA(\fg_{\Sigma_g})\coloneqq &\bigoplus_{r\geq 0}\HO(\Msp^{\Betti}_{g,r},\bm{\tau}^{\leq 0}p_*\VD\ul{\QQ}_{\Mst^{\Betti}_{g,r}})\otimes\LLL^{(g-1)r^2}\\
\subset&\mathcal{H}_{\Sigma_g}\coloneqq \bigoplus_{r\geq 0}\HOBM(\Mst^{\Betti}_{g,r},\QQ)\otimes\LLL^{(g-1)r^2}
\end{align*}
where $p\colon \Mst^{\Betti}_{g,r}\rightarrow \Msp^{\Betti}_{g,r}$ is the semisimplification morphism from the moduli stack of $r$-dimensional $\pi_1(\Sigma_g)$-modules to the coarse moduli space, for $\Sigma_g$ a genus $g$ Riemann surface without boundary\footnote{This case is slightly different, since the moduli stack of $\pi_1(\Sigma_g)[\omega]$-modules \textit{is} written as a global critical locus; see \cite{Da16}.}.  The object $\mathcal{H}_{\Sigma_g}$ is the CoHA of representations of the stack of $\CC[\pi_q(\Sigma_g)]$-modules defined in \cite{Da16}.  We leave the detailed study of the algebras $\UEA(\fg_C)$ and $\UEA(\fg_{\Sigma_g})$ to future work (see \cite{DHSM22,Dav21c}), and refer the reader to \cite{Da21a} for a general treatment of the perverse filtration on CoHAs for 2CY categories \cite{PS19,KV19}.
\subsection{Notation and conventions}
\label{NaCs}
All schemes and stacks are defined over $\CC$, and assumed to be locally of finite type.  All quivers are finite.  All functors are derived.  

If $G$ is an algebraic group we define $\HG\coloneqq\HO(\B G,\QQ)$.

We write $\HO(X,\mathbb{Q})=\HO(X_{\mathrm{an}},\mathbb{Q})$.

If an algebraic group $G$ acts on a scheme $X$, we denote by $X/G$ the stack-theoretic quotient.

If $\XX$ is a scheme or stack, and $p\colon \XX\rightarrow \pt$ is the morphism to  point, we often write $\HO$ for the derived functor $p_*$, and $\HO^i$ for the $i$th cohomology of $\HO$, i.e. for $\mathcal{F}$ an analytically constructible complex of sheaves on $X$, we abbreviate
\[
\HO(\mathcal{F})\coloneqq\HO(\XX,\mathcal{F});\quad\quad
\HO^i(\mathcal{F})\coloneqq \HO^i(\XX,\mathcal{F}).
\]
By a variety we mean a finite type reduced scheme $X$.  In particular we do not assume that $X$ is irreducible.

For $X$ an irreducible scheme or stack, we write $\HO(X,\QQ)_{\vir}=\HO(X,\QQ)\otimes \LLL^{-\dim(X)/2}$ where the half Tate twist is as in \S \ref{ICsec}.  For example, 
\[
\HO(\B \CC^*,\QQ)_{\vir}\coloneqq \HO(\B \CC^*,\QQ)\otimes\LLL^{1/2}.
\]
We define $\HOBM(\XX,\QQ)=\HO(\VD\:\!\ul{\QQ}_{\XX})$ where $\VD$ is the Verdier duality functor.

If $\mathscr{C}$ is a triangulated category equipped with a t structure we write $\Ho(\mathcal{F})=\bigoplus_{i\in \mathbb{Z}}\Ho^i(\mathcal{F})[-i]$ when the right hand side exists in $\mathscr{C}$.

If $V$ is a cohomologically graded vector space with finite-dimensional graded pieces, we define 
\begin{equation}
\label{PPdef}
\chi_t(V)\coloneqq \sum_{i\in \mathbb{Z}}(-1)^i\dim(V^i)t^{i/2}.
\end{equation}
If $V$ also carries a weight filtration $W_{\bullet}V$, we define the weight polynomial
\begin{equation}
\label{VPPdef}
\wt(V)\coloneqq \sum_{i,n\in\ZZ}(-1)^i\dim(\Gr^{\mathrm{W}}_n(V^i))t^{n/2}.
\end{equation}
We define $\mathbb{N}=\mathbb{Z}_{\geq 0}$.
\subsection{Acknowledgements}
During the writing of the paper, I was supported by the starter grant ``Categorified Donaldson--Thomas theory'' No. 759967 of the European Research Council.  I was also supported by a Royal Society university research fellowship.  I would like to thank Andrei Okounkov and Olivier Schiffmann for helpful conversations, Tristan Bozec for patiently explaining his work on crystals to me, Lucien Hennecart and Shivang Jindal for helpful comments regarding an earlier version of the paper, and an anonymous referee for a careful reading of the paper and many helpful suggestions.  Finally, I offer my heartfelt gratitude to Paul, Sophia, Sacha, Kristin and Nina, for their help and support throughout the writing of this paper.

\section{Background on CoHAs}
\label{bocsec}
\subsection{Monodromic mixed Hodge modules}
\label{MMHMs}
\subsubsection{Mixed Hodge modules}
Let $X$ be a locally finite type reduced scheme.  We define as in \cite{Saito89,Saito90} the category $\MHM(X)$ of mixed Hodge modules on $X$.  If $X$ is not reduced we write $\MHM(X)\coloneqq \MHM(X_{\mathrm{red}})$.  There is an exact functor 
\[
\rat_X\colon\Db(\MHM(X))\rightarrow \Db(\Perv(X))
\]
and the functor $\rat_X\colon \MHM(X)\rightarrow \Perv(X)$ is faithful.  There is a six functor formalism, at the level of derived categories of mixed Hodge modules, worked out in \cite{Saito90}, lifting the analogous formalism at the level of constructible complexes, in the sense that the six functors all naturally commute with $\rat_X$.  In particular we may define $\VD\ul{\QQ}_X\coloneqq (X\rightarrow \pt)^!\ul{\QQ}_{\pt}$, the dualizing complex, which is a lift of the constant sheaf $\QQ_X$ to a complex of mixed Hodge modules.  More generally we have the Verdier duality functor $\VD_X\colon \mathcal{F}\mapsto \mathcal{H}\mathrm{om}(\mathcal{F},\VD\ul{\QQ}_X)$.

  If $q\colon X\rightarrow Y$ is smooth of dimension $d$, there is a functor (at the level of Abelian categories of mixed Hodge modules) $f^*[d]\colon \MHM(X)\rightarrow \MHM(Y)$.  Moreover via these functors, the category of mixed Hodge modules on schemes smooth over $Y$ forms a stack, so that we may define mixed Hodge modules on a locally finite type Artin stacks $\mathfrak{X}$ in the natural way, as global sections of the resulting sheaf on the (small) smooth site $\mathfrak{X}_{\mathrm{sm}}$, see \cite{emhm} for details.

\subsubsection{Monodromic mixed Hodge modules}
Let $\FX$ continue to denote a locally finite type Artin stack.  We will make light use of the larger category of monodromic mixed Hodge modules $\MMHM(\FX)$ considered in \cite{KS2,QEAs}, which is defined to be the Serre quotient $\mathscr{B}_\FX/\mathscr{C}_\FX$, of two full subcategories of $\MHM(\FX\times\AAA{1})$. Here, $\mathscr{B}_\FX$ is the full subcategory containing those objects for which the cohomology mixed Hodge modules are locally constant, away from the origin, when restricted to $\{x\}\times\AAA{1}$, for each $x\in \FX$ a closed point.  The category $\mathscr{C}_\FX$ is the full subcategory containing those $\mathcal{F}$ for which such restrictions have globally constant cohomology sheaves.  

The functor $F=(\FX\times\GG_m\hookrightarrow \FX\times\AAA{1})_!$ provides an equivalence of categories between $\MMHM(\FX)$ and the full subcategory of mixed Hodge modules on $\FX\times\GG_m$ containing those $\mathcal{F}$ satisfying the condition that the restriction to each $\{x\}\times\GG_m$ has locally constant cohomology sheaves.  Write $G$ for a quasi-inverse to $F$.  We define the inclusion 
\[
\tau\colon \FX\hookrightarrow \FX\times\GG_m;\quad\quad
x\mapsto (x,1).
\]
Then there is a faithful functor
\[
\rat^{\mon}_\FX=\rat_\FX \circ\tau^*[-1]\circ G\colon \MMHM(\FX)\rightarrow\Perv(\FX).
\]
In words, this is the forgetful functor that forgets both the monodromy action on a monodromic mixed Hodge module, and then only remembers the underlying perverse sheaf of the resulting mixed Hodge module.

Let $z_\FX\colon \FX\hookrightarrow \FX\times\AAA{1}$ be the inclusion of the zero section.  Then $z_{\FX,*}\colon \MHM(\FX)\rightarrow \MMHM(\FX)$ is an inclusion of tensor categories, where the tensor product on the target is the one described in \S \ref{sfs}.  We refer to monodromic mixed Hodge modules, or complexes of monodromic mixed Hodge modules, obtained via applying $z_{\FX,*}$ as \textit{monodromy-free}.  We write $\MMHS\coloneqq \MMHM(\pt)$.  The category of graded-polarizable mixed Hodge structures is a full subcategory of $\MMHS$ via $z_{\pt,*}$.  

\subsubsection{Six functors for MMHMs}\label{sfs}
In this section, $X,Y,Z$ denote locally finite type schemes.  Excepting the definition of tensor products, the six functor formalism for categories of monodromic mixed Hodge modules is induced in a straightforward way by that of mixed Hodge modules, e.g. for $f\colon X\rightarrow Y$ a morphism of varieties we define
\[
f_*,f_!\colon\Db(\MMHM(X))\rightarrow \Db(\MMHM(Y))
\]
to be the functors induced by
\[
(f\times\id_{\AAA{1}})_*,(f\times\id_{\AAA{1}})_!\colon \Db(\MHM(X\times\AAA{1}))\rightarrow \Db(\MHM(Y\times\AAA{1}))
\]
respectively.  The functor $\VD_X\colon \MHM(X\times \AAA{1})\rightarrow \MHM(X\times \AAA{1})^{\opp}$ sends objects of $\mathscr{C}_X$ to objects of $\mathscr{C}_X^{\opp}$, inducing the functor $\Dmon_X\colon\MMHM(X)\rightarrow \MMHM(X)^{\opp}$.  We may omit the $\mon$ superscript when doing so is unlikely to cause confusion.

If $X$ and $Y$ are varieties over a variety $S$, and $\mathcal{F}\in\Ob(\Db(\MMHM(X)))$, $\mathcal{G}\in\Ob(\Db(\MMHM(Y)))$, then taking their external tensor product (as mixed Hodge modules) we obtain 
\[
\mathcal{J}\in\Ob(\Db(\MHM(Z\times\AAA{2}))),
\]
where $Z=X\times_S Y$.  Write $+\colon \AAA{2}\rightarrow \AAA{1}$ for the addition map.  We define
\[
\mathcal{F}\boxtimes_S \mathcal{G}\coloneqq (\id_Z\times +)_*\mathcal{J}\in\Ob(\Db(\MMHM(Z))).
\]
If $S=\pt$ one may show that this monoidal product is bi-exact \cite[Sec.4.2]{KS2}.  Let $X$ be a monoid over $S$.  I.e. there exist $S$-morphisms 
\[
\nu\colon X\times_S X\rightarrow X;\quad\quad
i\colon S\rightarrow X
\]
satisfying the standard axioms.  For $\mathcal{F},\mathcal{G}\in\Ob(\Db(\MMHM(X)))$, we define 
\[
\mathcal{F}\boxtimes_{\nu} \mathcal{G}\coloneqq \nu_*\!\left(\mathcal{F}\boxtimes_S\mathcal{G}\right)\in\Ob(\Db(\MMHM(X))).
\]
This monoidal product is symmetric if $\nu$ is commutative, and is exact if $\nu$ is finite.  If $\nu$ is commutative, we define
\[
\Sym_{\nu}(\mathcal{F})\coloneqq \bigoplus_{i\geq 0}\Sym_{\nu}^i(\mathcal{F})
\]
where $\Sym_{\nu}^i(\mathcal{F})$ is the $\mathfrak{S}_i$-invariant part of
\[
\underbrace{\mathcal{F}\boxtimes_{\nu}\ldots\boxtimes_{\nu}\mathcal{F}}_{i \textrm{ times}},
\]
and the $\mathfrak{S}_i$-action is defined via the isomorphism 
\begin{equation}
\label{toS}
\underbrace{\mathcal{F}\boxtimes_{\nu}\ldots\boxtimes_{\nu}\mathcal{F}}_{i \textrm{ times}}\cong e^*\!\left(\underbrace{\mathcal{F}\boxtimes\ldots\boxtimes\mathcal{F}}_{i \textrm{ times}}\right)
\end{equation}
where $e\colon X\times_S\cdots\times_S X\hookrightarrow X\times\cdots \times X$ is the natural embedding.  By \cite{MSS11} (see \cite[Sec.3.2]{QEAs}), the target of \eqref{toS} carries a natural $\mathfrak{S}_i$-action.  The functor $z_*\colon\MHM(X)\rightarrow \MMHM(X)$ from \S \ref{MMHMs} is a symmetric monoidal functor.
\subsubsection{MMHM complexes on stacks}
Let $\XX$ be a locally finite type Artin stack, which throughout the paper we assume may be written as a global quotient stack $\XX=X/G$, with $G$ an affine algebraic group.  We define the bounded derived category of mixed Hodge modules $\DfbMHM(\XX)$ as in \cite{emhm}, using $n$-acyclic covers as in \cite{BL94}.  

We give $\AAA{1}$ the trivial $G$-action.  We define the category $\DfbMMHM(\XX)$ to be the full subcategory of $\DfbMMHM((X\times\mathbb{G}_m)/G)$ containing complexes $\mathcal{F}$ satisfying the condition that for each inclusion $\iota\colon \pt \rightarrow X/G$ of a closed point, the cohomology sheaves of $(\iota\times\id_{\mathbb{G}_m})^*\mathcal{F}$ are locally constant.  Since, aside from some half Tate twists, almost all monodromic mixed Hodge modules in this paper will be monodromy-free, the reader may safely replace $\DfbMMHM((X\times\mathbb{G}_m)/G)$ with the category of $G$-equivariant mixed Hodge modules described in \cite{emhm} if they would prefer.  

%If $\XX$ is a connected locally finite type Artin stack we define the bounded derived category of monodromic mixed Hodge modules $\DfbMMHM(\XX)$ as in \cite{preproj1}.  Since all Artin stacks that we encounter for the rest of this paper will be global quotient stacks, and aside from some half Tate twists almost all monodromic mixed Hodge modules will be monodromy-free, the reader may think of this as the category of $G$-equivariant mixed Hodge modules described in \cite{emhm}.  

The category $\DfbMMHM(\XX)$ admits a natural t structure for which the heart is the category $\MMHM(\XX)$ of monodromic mixed Hodge modules on $\XX$, which admits the faithful functor $\rat_{\XX}^{\mon}$ to the category of perverse sheaves on $\XX$.  We write $\bm{\tau}^{\leq n}$ and $\bm{\tau}^{\geq n}$ for the truncation functors defined with respect to this t structure.  If we write $\XX\cong X/G$ as a global quotient stack, then up to a cohomological shift by $\dim(G)$ this is the category of $G$-equivariant perverse sheaves on $X$.  %For a more leisurely treatment of this part of the theory, we refer the reader to \cite{preproj1}.  

If $\XX$ is not necessarily connected we define
\[
\DbMMHM(\XX)\coloneqq \prod_{\XX'\in\pi_0(\XX)}\DfbMMHM(\XX).
\]

Let $\XX$ be a connected locally finite type Artin stack.  We define the category $\DlMMHM(\XX)$ by setting the objects to be $\mathbb{Z}$-tuples of objects $\mathcal{F}^{\leq n}\in\DfbMMHM(\XX)$ such that $\Ho^m(\mathcal{F}^{\leq n})=0$ for $m>n$, along with the data of isomorphisms $\bm{\tau}^{\leq n-1}\mathcal{F}^{\leq n}\cong\mathcal{F}^{\leq n-1}$.  We define $\DcMMHM(\XX)$ in the analogous way, by considering tuples of objects $\mathcal{F}^{\geq n}$ along with isomorphisms $\bm{\tau}^{\geq n}\mathcal{F}^{\geq n-1}\cong\mathcal{F}^{\geq n}$.  If $\XX$ is a disjoint union of locally finite type Artin stacks we define
\[
\DlMMHM(\XX)\coloneqq \prod_{\XX'\in \pi_0(\XX)}\DlMMHM(\XX')
\]
and likewise for $\DcMMHM(\XX)$.  For $f\colon \XX\rightarrow \YY$ a morphism of Artin stacks we define functors $f_*\colon \DlMMHM(\XX)\rightarrow \DlMMHM(\YY)$ and $f_!\colon \DcMMHM(\XX)\rightarrow \DcMMHM(\YY)$ in the natural way.  Alternatively, since we work throughout with global quotient stacks, these direct image functors may be defined via approximation to the Borel construction, as in \cite[Sec.2.3]{Da21a}.  %The selling point of the categories introduced in this paragraph is that they give us a setting to talk about direct images of complexes of monodromic mixed Hodge modules along non-representable morphisms of stacks without needing a full theory of unbounded derived categories of such objects.

We define $\DbMHM(\XX),\DlMHM(\XX)$ etc. the same way, and consider these categories as subcategories of their monodromic counterparts via $z_{X,*}$.  The Verdier duality functor $\Dmon_X$ preserves $\DbMMHM(\XX)$ and $\DbMHM(\XX)$ (embedded via $z_*$), and interchanges $\DcMMHM(\XX)$ with $\DlMMHM(\XX)$.

\subsubsection{Weight filtrations}
If $X$ is a locally finite scheme, an object $\mathcal{F}\in\Ob(\MMHM(X))$ inherits a weight filtration from its weight filtration in $\MHM(X\times\AAA{1})$, and is called pure of weight $n$ if $\Gr^W_i\!(\mathcal{F})=0$ for $i\neq n$.  For $\XX$ a stack, an object $\mathcal{F}\in\DlMMHM(\XX)$ is called \textbf{pure} if $\Ho^i(\mathcal{F})$ is pure of weight $i$ for every $i$.  An object of $\DlMMHM(\XX)$ or $\DcMMHM(\XX)$ is pure if its pullback along a smooth atlas is pure.  The following theorem, due to Saito, plays a crucial role in what follows, and is the lift of the BBDG decomposition theorem to mixed Hodge modules that we will use.
\begin{theorem}[\cite{Saito90}]
\label{SDC}
Let $\XX$ be a connected locally finite type quotient stack, and let $p\colon \XX\rightarrow \YY$ be a projective morphism of stacks.  Let $\mathcal{F}\in\DlMMHM(\XX)$ be pure,  and assume moreover that $\mathcal{F}\cong \mathcal{G}\otimes \LLL^{m/2}$ for some $m$, and some pure mixed Hodge module complex $\mathcal{G}\in\DlMHM(\XX)$.   Then there is an isomorphism $p_*\mathcal{F}\cong\Ho(p_*\mathcal{F})$, and each $\Ho^n(p_*\mathcal{F})$ can further be decomposed
\[
\Ho^n(p_*\mathcal{F})\cong \bigoplus_{i\in S_n}\ICS_{\ol{Z_i}}(\underline{\mathcal{F}}_i)\otimes \LLL^{m/2}[m]
\]
with $Z_i\subset \XX$ locally closed, smooth, connected substacks, and $\ICS_{\ol{Z_i}}(\underline{\mathcal{F}}_i)$ intermediate extensions of pure weight $n-2m$ variations of mixed Hodge structure on them.  
\end{theorem}
\begin{proof}
Since both the functors $\Ho$ and $p_*$ commute with $-\otimes \LLL^{m/2}$, it is sufficient to prove the proposition under the assumption that $m=0$.  Since $p_*$ has cohomological amplitude bounded below, it is sufficient to prove the proposition under the condition that $\mathcal{F}\in\DfbMMHM(\XX)$, i.e. $\mathcal{F}$ has cohomological amplitude $M\in\mathbb{N}$.  Recall that we assume that $\YY=Y/G$ and $\XX=X/G'$ are  global quotient stacks, with $G$ and $G'$ affine algebraic groups.  After possibly replacing $Y$ with $Y\times G'$ and $X$ with $X\times G$ we may assume that $G=G'$.  Embed $G\subset \Gl_m$ for some $m$.  Then pick $N\gg 0$ and let $U\subset \Hom(\CC^N,\CC^m)$ be the open subvariety of surjective morphisms; an object of $\DfbMMHM(\XX)$ with cohomological amplitude bounded above by $M$ is given by its pullback along $X\times_G U$, along with descent data, and similarly for its direct image along $p$ (see \cite{emhm} for details).  We reduce, then, to the case in which $\XX$ and $\YY$ are varieties and $\mathcal{F}\in \DfbMHM(\XX)$.  Now this is precisely Saito's version of the decomposition theorem \cite{Saito90}.
\end{proof}

%Via Saito's theory, if $p\colon \XX\rightarrow \YY$ is projective and $\mathcal{F}$ is pure, then $p_*\mathcal{F}$ is pure.
\begin{remark}
We suspect that the above decomposition theorem remains true without the strong assumption that $\mathcal{F}\cong \mathcal{G}\otimes \LLL^{m/2}$ with $\mathcal{G}$ monodromy-free.  The above version is all that we need in this paper, so we leave questions regarding this generalisation to further study.
\end{remark}
\begin{example}
For the uninitiated, we calculate the weight filtrations on a couple of objects in $\MMHM(\pt)$.  Consider the morphism 
\[
p\colon \mathbb{C}^*\rightarrow \mathbb{C}^*;\quad\quad
t\mapsto t^2
\]
and the constant weight one mixed Hodge module $\mathcal{P}=\underline{\mathbb{Q}}_{\mathbb{C}^*}[1]$.  Let $j\colon \mathbb{C}^*\rightarrow \mathbb{A}^1$ be the inclusion.  Then $j_!p_*\mathcal{P}$ is clearly locally constant away from the origin, and so defines an object of $\MMHM(\pt)$.  We may write $p_*\mathcal{P}=\mathcal{R}_{1}\oplus \mathcal{R}_{-1}$ where $\mathcal{R}_m$ is a rank one local system on $\mathbb{C}^*$ with monodromy given by multiplication by $m$.  Then $j_!\mathcal{R}_{-1}=j_*\mathcal{R}_{-1}$ is the intermediate extension of $\mathcal{R}_{-1}$, and is a simple MHM of weight one.  On the other hand, from the exact sequence in $\MHM(\mathbb{A}^1)$
\[
0\rightarrow \underline{\mathbb{Q}}_0\rightarrow j_!\mathcal{R}_1\rightarrow \underline{\mathbb{Q}}_{\mathbb{Q}^1}[1]\rightarrow 0
\]
and the identity $\underline{\mathbb{Q}}_{\mathbb{Q}^1}=0$ in $\MMHM(\pt)$ we deduce that $j_!\mathcal{R}_1$ has weight zero.
\end{example}

\subsubsection{Intersection cohomology complexes}
\label{ICsec}

Let $\XX$ be a stack.  Then $\ul{\QQ}_{\XX}\in\Ob(\DfbMMHM(\XX))$ is defined by fixing $q^*\ul{\QQ}_{\XX}=\ul{\QQ}_X$ for all smooth morphisms $q\colon X\rightarrow \XX$ with $X$ a scheme, and using naturality of the isomorphisms $r^*\ul{\QQ}_Y\cong \underline{\mathbb{Q}}_X$ for $r\colon X\rightarrow Y$ a morphism of schemes.

Likewise, if $\XX$ is irreducible we define $\ICS_{\XX}(\QQ)$ by the property that for all smooth morphisms $q\colon X\rightarrow \XX$ there is a natural isomorphism $q^*\ICS_{\XX}(\QQ)\cong \ICS_{X}(\QQ)$, between the inverse image along $q$ and the intersection mixed Hodge module complex on $X$.  Note that unless $\XX$ is zero-dimensional, $\ICS_{\XX}(\QQ)$ is not a mixed Hodge module, but rather a complex with cohomology concentrated in degree $d=\dim(\XX)$.  This complex is pure, i.e. its $d$th cohomology mixed Hodge module is pure of weight $d$.  

Consider the morphism $s\colon \AAA{1}\xrightarrow{x\mapsto x^2}\AAA{1}$.  We define
\[
\LLL^{1/2}=\cone(\ul{\QQ}_{\AAA{1}}\rightarrow s_*\ul{\QQ}_{\AAA{1}})\in\Db(\MMHM(\pt)).
\]
This complex has cohomology concentrated in degree 1, and is pure.  Moreover there is an isomorphism $(\LLL^{1/2})^{\otimes 2}\cong\LLL$, justifying the notation.  We define
\begin{equation}
\label{nICdef}
\nIC_{\XX}\coloneqq \ICS_{\XX}(\QQ)\otimes\LLL^{-\dim(\XX)/2}.
\end{equation}
Since $\LLL^{1/2}$ is pure, this is a pure weight zero monodromic mixed Hodge module.

\subsubsection{G-equivariant MMHMs}
Assume that we have fixed an algebraic group $G$, and let $\XX=X/H$ be a global quotient stack, where an embedding $G\subset H$ is understood.  Examples relevant to this paper will be $\XX=\Mst^{G,\zeta\sst}(Q)$ or $\XX=\Msp^{G,\zeta\sst}(Q)$, defined in \S \ref{SCsec}.  We define 
\begin{align}
\label{nnICdef}\nnIC_{\XX}\coloneqq &\nIC_{\XX}\otimes\LLL^{-\dim(G)/2}.
\end{align}
The motivation for introducing the extra Tate twist in \eqref{nnICdef} alongside the one in \S \ref{ICsec} comes from the case $H=G$.  Thinking of the underlying complex of perverse sheaves for $\nnIC_{\XX}$ as a $G$-equivariant complex of perverse sheaves on $X$, the extra twist of \eqref{nnICdef} means that this complex is a genuine perverse sheaf (without shifting).

Continuing in the same vein, we shift the natural t structure on $\DlMMHM(\XX)$, defining truncation functors
\[
\bm{\tau}^{G,\leq i}\coloneqq \bm{\tau}^{\leq i-\dim(G)};\quad\quad
\bm{\tau}^{G,\geq i}\coloneqq \bm{\tau}^{\geq i-\dim(G)}.
\]
So for example if $X$ is an irreducible $G$-equivariant variety then $\Ho^{G,i}\!\left(\nnIC_{X/G}\right)\neq 0$ if and only if $i=0$, where the cohomology functor is with respect to the shifted t structure.  We denote by $\MMHM^G(\XX)$ the heart of this t structure (i.e. the shift by $\dim(G)$ of the usual t structure).

\subsubsection{Vanishing cycles}
Let $\XX$ be an algebraic stack %\footnote{We state all of Saito's results for stacks, as opposed to schemes.  The details of the extension to stacks can be found in \cite{preproj1}.}, 
and let $f\in\Gamma(\XX)$ be a regular function on it.  An integral part of Saito's theory is the construction of a functor $\phi_f[-1]\colon\MHM(\XX)\rightarrow \MHM(\XX)$ lifting the usual vanishing cycle functor $\varphi_f[-1]\colon\Perv(\XX)\rightarrow \Perv(\XX)$, in the sense that there is a natural equivalence $\rat_\XX  \phi_f\cong\varphi_f  \rat_\XX$.  There is a further lift $\phim{f}$ satisfying $\rat^{\mon}_\XX \phim{f}\cong \varphi_f \rat_\XX$, defined by
\[
\phim{f}\colon \MHM(\XX)\rightarrow \MMHM(\XX);\quad\quad
\mathcal{F}\mapsto j_!\phi_{f/u}(\XX\times\GG_m\rightarrow \XX)^*\mathcal{F}
\]
where $u$ is a coordinate for $\GG_m$ and $j\colon \XX\times\GG_m\rightarrow \XX\times\AAA{1}$ is the natural inclusion.  The vanishing cycle functor commutes with Verdier duality, i.e. by \cite{Sai89duality} there is a natural isomorphism of functors
\[
\phim{f}\VD_{\XX}\cong \VD_{\XX}^{\mon}\phim{f}\colon\MHM(\XX)\rightarrow \MMHM(\XX).
\]
Let $g\in\Gamma(\YY)$ be a regular function on the stack $\YY$.  Then there is a Thom--Sebastiani natural isomorphism \cite{Saito10}
\[
\phim{f}\boxtimes\phim{g}\rightarrow \left(\phim{f\boxplus g}(\bullet\boxtimes \bullet)\right)_{f^{-1}(0)\times g^{-1}(0)}\colon\MHM(\XX)\times\MHM(\YY)\rightarrow \MMHM(\XX\times\YY).
\]
%\subsubsection{Perverse filtrations}
%Let $p\colon \XX\rightarrow \YY$ and $r\colon \YY\rightarrow \ZZZ$ be morphisms of stacks, let $\mathcal{F}\in\Ob(\DlMMHM(\XX))$ satisfy 
%\blue{need this?}

\subsection{Quivers and their representations}
\label{qrepsec}
In this section we fix some notation regarding quiver representations.  By a quiver $Q$ we mean a pair of finite sets $Q_1$ and $Q_0$ (the arrows and vertices respectively) along with a pair of morphisms $s,t\colon Q_1\rightarrow Q_0$ taking each arrow to its source and target, respectively.  We say that $Q$ is symmetric if for every pair of vertices $i,j\in Q_0$ there are as many arrows $a$ satisfying $s(a)=i$ and $t(a)=j$ as there are arrows satisfying $s(a)=j$ and $t(a)=i$.  

We refer to elements $\dd\in\dvs$ as dimension vectors.  We denote the Euler form and its symmetrisation, respectively, on the set of dimension vectors by
\begin{align}
\label{Eulerform}
\chi_Q(\dd',\dd'')=\sum_{i\in Q_0}\dd'_i\dd''_i-\sum_{a\in Q_1}\dd'_{s(a)}\dd''_{t(a)};&&(\dd,\dd')_Q=\chi_Q(\dd,\dd')+\chi_Q(\dd',\dd).
\end{align}
If $Q$ is symmetric the Euler form $\chi_Q(\bullet,\bullet)$ is symmetric.

For $K$ a field, we denote by $K Q$ the free path algebra of $Q$ over $K$.  This algebra contains $\lvert Q_0\lvert $ mutually orthogonal idempotents $e_i$ for $i\in Q_0$, the ``lazy paths'' of length zero which begin and end at $i$.  We define the dimension vector $\udim(\rho)\in\dvs$ of a $K Q$-representation via $\udim(\rho)_i=\dim_K(e_i\cdot \rho)$.  Fixing $K=\CC$, if $W\in \CC Q_{\cyc}$ is a linear combination of cyclic words in $Q$, we define $\Jac(Q,W)$ as in \cite{ginz} to be the quotient of $\CC Q$ by the two-sided ideal generated by the noncommutative derivatives $\partial W/\partial a$ for $a\in Q_1$.
\subsubsection{Extra gauge group}
For each pair of (not necessarily distinct) vertices $i,j$ fix a complex vector space $V_{i,j}$ with basis the arrows from $i$ to $j$.  Set $\GQ{Q}\coloneqq \prod_{i,j\in Q_0}\Gl(V_{i,j})$.  Then $\GQ{Q}$ acts on $\AS_{\dd}(Q)$ via the isomorphism\footnote{Here we employ the standard abuse of notation, identifying vector spaces with their total spaces, considered as algebraic varieties.}
\[
\AS_{\dd}(Q)\cong\bigoplus_{i,j\in Q_0}V_{i,j}\otimes\Hom(\CC^{\dd_i},\CC^{\dd_j}).
\]
We fix a complex affine algebraic group $G$, and fix a homomorphism $G\rightarrow \GQ{Q}$.  We define
\[
\Gl_{\dd}\coloneqq\prod_{i\in Q_0}\Gl_{\dd_i};\quad\quad
\gl_{\dd}\coloneqq\prod_{i\in Q_0}\gl_{\dd_i};\quad\quad
\WT{\Gl}_{\dd}\coloneqq \Gl_{\dd}\times G.
\]
For $G$ an algebraic group we write $\HG\coloneqq\HO(\B G,\QQ)$.
\subsubsection{Stability conditions}
\label{SCsec}
By a \textbf{King stability condition} we mean a tuple $\zeta\in\QQ_+^{Q_0}$.  The slope of a nonzero dimension vector $\dd\in\dvs$ is defined by
\[
\mu^{\zeta}(\dd)=\frac{\zeta\cdot\dd}{\sum_{i\in Q_0}\dd_i}.
\]
We define the slope of a nonzero $KQ$-module by setting $\mu^{\zeta}(\rho)=\mu^{\zeta}(\udim(\rho))$.  For $\theta\in\QQ$ we define
\[
\Lambda_{\theta}^{\zeta}\coloneqq \{\dd\in\dvs\setminus\{0\}\colon \mu^{\zeta}(\dd)=\theta\}\cup\{0\}.
\]
A $K Q$-module $\rho$ is called $\zeta$\textit{-stable} if for all proper nonzero submodules $\rho'\subset \rho$ we have $\mu^{\zeta}(\rho')<\mu^{\zeta}(\rho)$, and is $\zeta$\textit{-semistable} if we have $\mu^{\zeta}(\rho')\leq \mu^{\zeta}(\rho)$.  We denote by
\[
\AS^{\zeta\sst}_{\dd}(Q)\subset \AS_{\dd}(Q)\coloneqq \prod_{a\in Q_1}\Hom(\CC^{\dd_{s(a)}},\CC^{\dd_{t(a)}})
\]
the open subvariety of $\zeta$-semistable $\CC Q$-modules.  We set
\[
\Mst^{G,\zeta\sst}_{\dd}(Q)\coloneqq \AS^{\zeta\sst}_{\dd}(Q)/\WT{\Gl}_{\dd}.
\]
If $G$ is trivial this stack is isomorphic to the stack of $\zeta$-semistable $\dd$-dimensional $\CC Q$-modules.  In \cite{King} King constructs $\Msp^{\zeta\sst}_{\dd}(Q)$, the coarse moduli space of $\zeta$-semistable $\dd$-dimensional $\CC Q$-representations, as a GIT quotient.  We define
\[
\Msp^{G,\zeta\sst}_{\dd}(Q)=\Msp^{\zeta\sst}_{\dd}(Q)/G,
\]
the stack theoretic quotient of this space by the residual $G$-action.  We denote by 
\[
\JH^{G}\colon\Mst^{G,\zeta\sst}(Q)\rightarrow \Msp^{G,\zeta\sst}(Q)
\]
the natural map.  At the level of points, this morphism takes a $\dd$-dimensional $\zeta$-semistable $\CC Q$-module $\rho$ with Jordan--H\"older filtration $0\subset \rho^{(1)}\subset\ldots\subset\rho^{(l)}\subset \rho$ (inside the category of semistable modules of fixed slope) to the associated polystable $\CC Q$-module $\bigoplus_{1\leq j\leq l}\rho^{(j)}/\rho^{(j-1)}$.

Given an algebra $A$, presented as a quotient of a free path algebra $\CC Q$ by some two-sided ideal $R$ preserved by $G$, we denote by $\Mst^{G,\zeta\sst}(A)$ the quotient by the $G$-action of the stack of $\zeta$-semistable $A$-modules, and by $\Mst^{G,\zeta\sst}_{\dd}(A)$ the substack of $\dd$-dimensional $G$-equivariant $A$-modules.  As above, we denote by $\Msp^{G,\zeta\sst}(A)$ the stack-theoretic quotient of the coarse moduli scheme by the residual $G$-action.

\subsubsection{Monoidal structure}
The stack $\Msp^{G,\zeta\sst}_{\theta}(Q)$ is a monoid in the category of stacks over $\B G$, via the morphism
\[
\oplus^G\colon\Msp^{G,\zeta\sst}_{\theta}(Q)\times_{\B G}\Msp^{G,\zeta\sst}_{\theta}(Q)\rightarrow \Msp^{G,\zeta\sst}_{\theta}(Q)
\]
taking a pair of $\zeta$-polystable $\CC Q$-modules to their direct sum.  This morphism is finite and commutative \cite[Lem.2.1]{Meinhardt14}, and so the monoidal product $\mathcal{F}\boxtimes_{\oplus^G}\mathcal{G}\coloneqq \oplus^G_*\!\left(\mathcal{F}\boxtimes_{\B G}\mathcal{G}\right)$ for $\mathcal{F},\mathcal{G}\in\DlMMHM(\Msp^{G,\zeta\sst}_{\theta}(Q))$ is bi-exact and symmetric.  The natural morphism $\Msp^{G,\zeta\sst}_{\theta}(Q)\times_{\B G}\Msp^{G,\zeta\sst}_{\theta}(Q)\rightarrow \Msp^{G,\zeta\sst}_{\theta}(Q)\times\Msp^{G,\zeta\sst}_{\theta}(Q)$ and the K\"unneth isomorphism induce the morphism $\HO(\mathcal{F})\otimes_{\HO_G} \HO(\mathcal{G})\rightarrow \HO(\mathcal{F}\boxtimes_{\oplus^G}\mathcal{G})$, making $\HO$ into a lax monoidal functor to the category of $\HO_G$-modules.

\subsubsection{Subscript conventions}
Throughout the paper, if $\mathbb{X}$ is some object that admits a decomposition with respect to dimension vectors $\dd\in\dvs$, we denote by $\mathbb{X}_{\dd}$ the subobject corresponding to the dimension vector $\dd$.  If $\mathcal{F}$ is a sheaf or mixed Hodge module defined on $\XX$, a stack that admits a decomposition indexed by dimension vectors, we denote by $\mathcal{F}_{\dd}$ its restriction to $\XX_{\dd}$.  Finally, if $f\colon\XX\rightarrow \YY$ is a morphism preserving natural decompositions of $\XX$ and $\YY$ indexed by dimension vectors, we denote by $f_{\dd}\colon\XX_{\dd}\rightarrow \YY_{\dd}$ the induced morphism.

If a stability condition $\zeta$ is fixed, we set $\XX_{\theta}=\coprod_{\dd\in\dvst}\XX_{\dd}$, and extend the conventions of the previous paragraph in the obvious way to objects admitting decompositions indexed by dimension vectors in $\dvst$, as well as morphisms that preserve these decompositions.

\subsubsection{Serre subcategories}
Throughout the paper, with a quiver $Q'$ fixed, $\Sp$ will be used to denote a Serre subcategory of the category of $\CC Q'$-modules, i.e. $\Sp$ is a full subcategory such that if
\[
0\rightarrow \rho'\rightarrow \rho\rightarrow\rho''\rightarrow 0
\]
is a short exact sequence of $\CC Q'$-modules then $\rho$ is an object of $\Sp$ if and only if $\rho'$ and $\rho''$ are.  We assume that $\Sp$ admits a geometric definition, in the sense that there is an inclusion of stacks
\[
\varpi\colon\Mst^{\Sp,G,\zeta\sst}(Q')\hookrightarrow \Mst^{G,\zeta\sst}(Q')
\]
which at the level of complex points is the inclusion of the set of objects of $\Sp$, and a corresponding inclusion $\varpi'\colon\Msp^{\SP,G,\zeta\sst}(Q')\hookrightarrow \Msp^{G,\zeta\sst}(Q')$ of coarse moduli spaces.

If the definition of an object $\mathcal{F}^{\SP}$ depends on a choice of some Serre subcategory $\SP$ of the category of $\CC Q'$-modules, for some quiver $Q'$, we omit the superscript $\SP$ as shorthand for the case in which we choose $\SP$ to be the entire category of $\CC Q'$-modules.

\subsection{Critical CoHAs}
\label{CCSec}
The CoHA associated to a quiver with potential was first introduced by Kontsevich and Soibelman in \cite[Sec.7]{KS2}, in order to categorify the theory of Donaldson--Thomas invariants arising in the study of 3--Calabi--Yau categories (see \cite{Thomas1} and then \cite{JoyceDT} and then references therein), and provide a mathematical framework for the study of algebras of BPS states initiated by Harvey and Moore \cite{HM98}.  We begin by recalling a generalisation of their construction, concentrating on the relative, or sheaf-theoretic version, studied in \cite{QEAs}.

For fixed quiver $Q$, Serre subcategory $\Sp$ of the category of $\CC Q$-modules, slope $\theta$, potential $W$, stability condition $\zeta$, and gauge group $G$, we set 
\begin{align*}
\HA_{Q,W,\theta}^{\SP,G,\zeta}\coloneqq &\varpi_*\varpi^!\phim{\WWW}\nnIC_{\Mst^{G,\zeta\sst}(Q)}\\
\rCoha_{Q,W,\theta}^{\Sp,G,\zeta}\coloneqq &\JH^G_*\HA_{Q,W,\theta}^{\Sp,G,\zeta}\cong \varpi'_*\varpi'^!\JH^G_*\HA_{Q,W,\theta}^{G,\zeta}\\
\HCoha_{Q,W,\theta}^{\Sp,G,\zeta}\coloneqq &\HO\!\left(\Mst^{G,\zeta\sst}(Q),\varpi_*\varpi^!\phim{\WWW}\nnIC_{\Mst^{G,\zeta\sst}(Q)}\right).
\end{align*}
\begin{assumption}
\label{3dassumption}
We will assume throughout that we have chosen $Q,W,\theta,\SP,G,\zeta$ so that $\HCoha_{Q,W,\theta}^{\Sp,G,\zeta}$ is a free $\HG$-module.  
\end{assumption}
\begin{remark}
Via the usual spectral sequence argument the purity of $\HCoha_{Q,W,\theta}^{\SP,\zeta}$ is a sufficient, but not necessary condition for Assumption \ref{3dassumption} to hold.  The assumption implies that the natural transformation $\HCoha_{Q,W,\theta}^{\Sp,G,\zeta}\otimes_{\HO_G}\HCoha_{Q,W,\theta}^{\Sp,G,\zeta}\rightarrow \HO\left(\JH^G_*\HA_{Q,W,\theta}^{\Sp,G,\zeta}\boxtimes_{\B G}\JH^G_*\HA_{Q,W,\theta}^{\Sp,G,\zeta}\right)$ is an isomorphism.  See \cite[Sec.9]{preproj} for more details.
\end{remark}
%; see \cite{preproj1} for an impure example for which the assumption holds.

Given dimension vectors $\dd',\dd''\in \dvst$ with $\dd=\dd'+\dd''$ we define $\AS_{\dd',\dd''}^{\zeta\sst}(Q)\subset \AS^{\zeta\sst}_{\dd}(Q)$ to be the subset of linear maps preserving the $Q_0$-graded subspace $\CC^{\dd'}\subset \CC^{\dd}$, and we define $\Gl_{\dd',\dd''}\subset \Gl_{\dd}$ to be the subgroup preserving the same subspace.  We define $\pi'_1,\pi'_2,\pi'_3$ to be the natural morphisms from $\AS_{\dd',\dd''}^{\zeta\sst}(Q)$ to $\AS_{\dd'}^{\zeta\sst}(Q)$, $\AS_{\dd}^{\zeta\sst}(Q)$ and $\AS_{\dd''}^{\zeta\sst}(Q)$ respectively.  We define
\[
\Mst_{\dd',\dd''}^{G,\zeta\sst}(Q)\coloneqq \AS_{\dd',\dd''}^{\zeta\sst}(Q)/\left(\Gl_{\dd',\dd''}\times G\right).
\]
Finally we define $\Mst_{\theta}^{G ,\zeta\sst}(Q)_{(2)}$ to be the union of the stacks $\Mst_{\dd',\dd''}^{G,\zeta\sst}(Q)$ across all $\dd',\dd''\in\dvst$.  

Consider the commutative diagram
\begin{equation}
\label{3dc}
\xymatrix{
&\Mst_{\theta}^{G ,\zeta\sst}(Q)_{(2)}\ar[dl]_{\pi_1\times\pi_3\;\;}\ar[dr]^{\pi_2}\\
\Mst_{\theta}^{G ,\zeta\sst}(Q)\times_{\B G} \Mst_{\theta}^{G ,\zeta\sst}(Q)\ar[d]^{\JH^{G }\times_{\B G} \JH^{G }}&&\Mst_{\theta}^{G ,\zeta\sst}(Q)\ar[d]^{\JH^{G }}\\
\Msp_{\theta}^{G ,\zeta\sst}(Q)\times_{\B G} \Msp_{\theta}^{G ,\zeta\sst}(Q)\ar[rr]^-{\oplus^{G }}&&\Msp_{\theta}^{G ,\zeta\sst}(Q)\\
}
\end{equation}
where $\pi_1,\pi_2,\pi_3$ are induced by $\pi'_1\pi'_2,\pi'_3$ respectively.  Set
\[
\mathbb{A}=\Mst_{\theta}^{G ,\zeta\sst}(Q)\times_{\B G} \Mst_{\theta}^{G ,\zeta\sst}(Q);\quad\quad
\mathbb{B}=\Mst_{\theta}^{G ,\zeta\sst}(Q)_{(2)};\quad\quad
\mathbb{O}=\Mst_{\theta}^{G ,\zeta\sst}(Q).
\]
Via the Thom--Sebastiani isomorphism and the composition of appropriate Tate twists of the morphisms
\[
\varpi'_*\varpi'^!(\JH^G\times_{\B G}\JH^G)_*\phim{\WWW}\!\left(\ul{\QQ}_{\mathbb{A}}\rightarrow (\pi_1\times\pi_3)_*\ul{\QQ}_{\mathbb{B}}\right)
\]
and
\[
\varpi'_*\varpi'^!\JH^G_*\phim{\WWW}\Dmon_{\mathbb{O}}\!\left(\ul{\QQ}_{\mathbb{O}}\rightarrow \pi_{2,*}\ul{\QQ}_{\mathbb{B}}\right)
\]
we define the morphism
\begin{equation}
\label{stardef}
\star\colon\rCoha_{Q,W,\theta}^{\Sp,G,\zeta}\boxtimes_{\oplus^G}\rCoha_{Q,W,\theta}^{\Sp,G,\zeta}\rightarrow \rCoha_{Q,W,\theta}^{\Sp,G,\zeta}
\end{equation}
i.e. a multiplication operation on $\rCoha_{Q,W,\theta}^{\Sp,G,\zeta}$.  The proof that this operation is associative is standard, and is as in \cite[Sec.2.3]{KS2}.  Applying $\HO$ to this morphism we obtain a morphism
\begin{equation}
\label{RelHH}
\mult\colon\HCoha_{Q,W,\theta}^{\Sp,G,\zeta}\otimes_{\HG}\HCoha_{Q,W,\theta}^{\Sp,G,\zeta}\rightarrow \HCoha_{Q,W,\theta}^{\Sp,G,\zeta}
\end{equation}
and obtain a $\HG$-linear (associative) product on $\HCoha_{Q,W}^{\Sp,G,\zeta}$ by precomposing \eqref{RelHH} with the surjection
\[
\HCoha_{Q,W,\theta}^{\Sp,G,\zeta}\otimes\HCoha_{Q,W,\theta}^{\Sp,G,\zeta}\rightarrow\HCoha_{Q,W,\theta}^{\Sp,G,\zeta}\otimes_{\HG}\HCoha_{Q,W,\theta}^{\Sp,G,\zeta}.
\]
\subsubsection{$\psi$-twists}
\label{psitwists}
To state the PBW theorem, we need to slightly twist the multiplication on $\rCoha_{Q,W,\theta}^{\Sp,G,\zeta}$, as explained in \cite[Sec.1.6]{QEAs}.  First, we define the bilinear form
\[
\tau(\dd',\dd'')\coloneqq \chi_Q(\dd',\dd')\chi_Q(\dd'',\dd'')+\chi_Q(\dd',\dd'').
\]
Then $\tau(\dd,\dd)$ is even for all $\dd\in\mathbb{Z}^{Q_0}$, and so we can find a bilinear form $\psi$ on $\mathbb{Z}^{Q_0}$ such that 
\begin{align}
\label{psidef}
\psi(\dd',\dd'')+\psi(\dd'',\dd')=\tau(\dd',\dd'')&\quad\textrm{mod }2.
\end{align}
We define $\rCoha_{Q,W,\theta}^{\Sp,G,\zeta,\psi}$ to be the algebra object with the same underlying complex of MMHMs as $\rCoha_{Q,W,\theta}^{\Sp,G,\zeta}$, but with the modified product $\star'_{\dd',\dd''}=(-1)^{\psi(\dd',\dd'')}\star_{\dd',\dd''}$, where the LHS and RHS of this equation denote the restrictions of the respective products to $\rCoha_{Q,W,\dd'}^{\Sp,G,\zeta}\boxtimes_{\oplus^G}\rCoha_{Q,W,\dd''}^{\Sp,G,\zeta}$.  We define $\HCoha_{Q,W,\theta}^{\Sp,G,\zeta,\psi}$ via the same $\psi$-twist.
\subsection{The PBW theorem}
\label{PBWsec}
We next recall some fundamental results for critical CoHAs from %\footnote{For the extension to the $G$-equivariant case considered here, we refer the reader to \cite{preproj1}.} 
\cite{QEAs}.  For ease of exposition we assume that $Q$ is symmetric, though for generic stability conditions more general results are results are stated more generally in \cite{QEAs}, to which we refer for a more leisurely presentation of what follows.

Firstly, there is an isomorphism $\JH^G_*\HA_{Q,W,\theta}^{G,\zeta}\cong\Ho\!\left(\JH^G_*\HA_{Q,W,\theta}^{G,\zeta}\right)$, meaning (see \S \ref{NaCs}) $\JH^G_*\HA_{Q,W,\theta}^{G,\zeta}$ is isomorphic to the direct sum of its (cohomologically shifted) cohomology MMHMs.  In other words \textit{one half} of the BBDG decomposition theorem \eqref{decompp1} holds. Moreover, 
\begin{equation}
\label{JHvanishing}
\bm{\tau}^{G,\leq 0}\!\left(\JH^G_*\HA_{Q,W,\theta}^{G,\zeta}\right)=0.  
\end{equation}
We thus have 
\begin{equation}
\label{decomp}
\rCoha_{Q,W,\theta}^{\Sp,G,\zeta}\cong\varpi'_*\varpi'^!\Ho\!\left(\JH^G_*\HA_{Q,W,\theta}^{G,\zeta}\right).
\end{equation}

Setting
\begin{equation}
\label{BPShdef}
\BPSh_{Q,W,\theta}^{\SP,G,\zeta}\coloneqq \varpi'_*\varpi'^!\bm{\tau}^{G,\leq 1}\JH^G_*\HA_{Q,W,\theta}^{G,\zeta}\otimes\LLL^{-1/2}
\end{equation}
for each $\dd\in\dvst$ there is an isomorphism
\begin{equation}
\label{BPSgive}
\BPSh_{Q,W,\dd}^{\SP,G,\zeta}\cong \begin{cases} \varpi'_*\varpi'^!\phim{\WW}\nnIC_{\Msp_{\dd}^{G,\zeta\sst}(Q)}& \textrm{if }\Msp_{\dd}^{\zeta\stab}(Q)\neq \emptyset\\
0& \textrm{otherwise.}\end{cases}
\end{equation}
We define the BPS cohomology
\[
\BPS_{Q,W,\theta}^{\SP,G,\zeta}\coloneqq \HO\!\left(\Msp^{G,\zeta\sst}_{\theta}(Q),\BPSh_{Q,W,\theta}^{\SP,G,\zeta}\right).
\]
Denote by $J_{\dd}$ the composition $\Mst_{\dd}^{G ,\zeta\sst}(Q)\rightarrow \B \Gl_{\dd}\xrightarrow{\B \mathrm{Det}}\B \CC^*$, and set $J^+_{\dd}\colon\Mst_{\dd}^{G ,\zeta\sst}(Q)\xrightarrow {\id\times J_{\dd}}\Mst_{\dd}^{G ,\zeta\sst}(Q)\times \B\CC^*$.    If $\mathcal{F}\in \DlMMHM(\pt)$ is a complex of monodromic mixed Hodge structures, and $\mathfrak{X}$ is a stack, we abuse notation and denote by $\mathcal{F}$ also the inverse image $(\mathfrak{X}\rightarrow \pt)^*\mathcal{F}$.  The morphisms $J^+_{\dd}$ induce an action of $\HO_{\CC^*}$ on $\rCoha_{Q,W,\theta}^{\SP,G,\zeta}$, i.e. a morphism $\rCoha_{Q,W,\theta}^{\SP,G,\zeta}\otimes  \HO_{\CC^*}\rightarrow \rCoha_{Q,W,\theta}^{\SP,G,\zeta}$, and this induces the morphism
\begin{equation}
\label{BCA}
\BPSh_{Q,W,\theta}^{\SP,G,\zeta}\otimes\HO(\B \CC^*,\QQ)_{\vir}\rightarrow \rCoha_{Q,W,\theta}^{\SP,G,\zeta}.
\end{equation}
Given an algebra $A$ and an $A$-bimodule $L$ we define
\[
\mathrm{T}_A(L)\coloneqq \bigoplus_{i\geq 0} \underbrace{L\otimes_A\cdots\otimes_A L}_{i\textrm{ times}}
\]
where the $i=0$ summand is the $A$-bimodule $A$.  Given an $A$-linear Lie algebra $\mathfrak{g}$, i.e. a Lie algebra $\mathfrak{g}$, with an $A$-action such that the Lie algebra map $\mathfrak{g}\otimes\mathfrak{g}\rightarrow\mathfrak{g}$ factors through the surjection $\mathfrak{g}\otimes\mathfrak{g}\rightarrow \mathfrak{g}\otimes_A\mathfrak{g}$
we define the $A$-linear universal enveloping algebra as the quotient algebra
\[
\UEA_A(\mathfrak{g})\coloneqq \mathrm{T}_A(\mathfrak{g})/\langle a\otimes b-b\otimes a-[a,b]_{\mathfrak{g}}\rangle.
\]
Likewise, if $N$ is an $A$-bimodule we define $\Sym_A(N)\coloneqq \mathrm{T}_A(\mathfrak{g})/\langle a\otimes b-b\otimes a\rangle$.  The main structural result for the critical CoHA is the following PBW theorem:
\begin{theorem}
\label{PBWtheorem}
The morphism 
\begin{equation}
\label{relPBW}
\Phi^{\SP,\psi}\colon \Sym_{\oplus^G}\!\left(\BPSh_{Q,W,\theta}^{\SP,G,\zeta}\otimes\HO(\B \CC^*,\QQ)_{\vir}\right)\rightarrow \rCoha_{Q,W,\theta}^{\SP,G,\zeta,\psi}
\end{equation}
obtained by combining \eqref{BCA} with the iterated CoHA multiplication map is an isomorphism in $\DlMMHM(\Msp^{\SP,G,\zeta\sst}_{\theta}(Q))$.  Moreover 
\[
\Ho(\Phi)\colon \Sym_{\oplus^G}\!\left(\BPSh_{Q,W,\theta}^{G,\zeta}\otimes\HO(\B \CC^*,\QQ)_{\vir}\right)\rightarrow \Ho(\rCoha_{Q,W,\theta}^{G,\zeta,\psi})
\]
is an isomorphism of algebra objects in $\DlMMHM(\Msp^{G,\zeta\sst}_{\theta}(Q))$, and $\HO(\Phi^{\SP,\psi})$ is an isomorphism of $\dvst$-graded monodromic mixed Hodge structures
\begin{equation}
\label{absPBW}
\Sym_{\HG}\!\left(\BPS_{Q,W,\theta}^{\SP,G,\zeta}\otimes \HO(\B \CC^*,\QQ)_{\vir}\right)\rightarrow \HCoha^{\SP,G,\zeta}_{Q,W,\theta}.
\end{equation}
\end{theorem}
\begin{remark}
The above Theorem was proved in \cite{QEAs} for $G$ trivial, although extending the proof to general $G$ is routine: let $q\colon \Msp^{G,\zeta\sst}_{\theta}(Q)\rightarrow \Msp^{\zeta\sst}_{\theta}(Q)$ be the quotient map.  Then it is proved in \cite{QEAs} that $q^*\Phi^{\psi}$ is an isomorphism in $\DlMMHM(\Msp^{G,\zeta\sst}_{\theta}(Q))$, and that $\Ho(q^*\Phi^{\psi})$ is an isomorphism of algebra objects.  The case for general $G$ follows, since the statement that a given morphism in $\DlMMHM(\Msp^{G,\zeta\sst}_{\theta}(Q))$ is an isomorphism can be checked on a smooth cover, as can the statement that $\Ho(\alpha)=\Ho(\alpha')$, for two morphisms $\alpha,\alpha'\colon L\boxtimes_{\oplus^G}L\rightarrow \Ho(\rCoha_{Q,W,\theta}^{G,\zeta})$, where $L$ is the LHS of \ref{relPBW}.  The case for general $\SP$ then follows by restriction.
\end{remark}
%\begin{remark}
%Strictly speaking, the symmetric monoidal structure on $\DlMMHM(\Msp^{\SP,G,\zeta}_{\theta}(Q))$ should be twisted by a sign depending on the Euler form of $Q$, in the definition of the domain of $\Phi$.  In this paper we only consider Hall algebras $\rCoha_{Q,W,\theta}^{\SP,G,\zeta}$ and $\HCoha_{Q,W,\theta}^{\SP,G,\zeta}$ for quivers $Q$ satisfying the condition that $\chi_Q(\dd',\dd'')\in 2\cdot\ZZ$ for all $\dd',\dd''\in\dvs$, and so we may omit this added complication (see \cite[Sec.1.6]{QEAs} for details).
%\end{remark}
By \eqref{decomp} the monodromic mixed Hodge structure $\HCoha_{Q,W,\theta}^{\SP,G,\zeta}$ carries a filtration defined by 
\begin{equation}
\label{OPdef}
\mathfrak{P}_{i}\!\HCoha_{Q,W,\theta}^{\SP,G,\zeta}\coloneqq \HO\!\left(\Msp^{G,\zeta\sst}_{\theta}(Q),\varpi'_*\varpi'^!\bm{\tau}^{G,\leq i}\JH^G_*\HA_{Q,W,\theta}^{\SP,G,\zeta}\right).
\end{equation}
We define\footnote{A priori, the underlying object of $\mathfrak{g}_{Q,W,\theta}^{\SP,G,\zeta,\psi}$ does not depend on $\psi$, although the Lie algebra structure might.}
\begin{align*}
\mathfrak{g}_{Q,W,\theta}^{\SP,G,\zeta,\psi}\coloneqq &\mathfrak{P}_{1}\!\HCoha_{Q,W,\theta}^{\SP,G,\zeta}
\\
\cong&\HO\!\left(\Msp^{\SP,G,\zeta\sst}_{\theta}(Q),\varpi'_*\varpi'^!\!\Ho^{G,1}\!\left(\JH^G_*\HA_{Q,W,\theta}^{\SP,G,\zeta}\right)\right)\\
\cong&\BPS_{Q,W,\theta}^{\SP,G,\zeta}\otimes\LLL^{1/2}.
\end{align*}

By Theorem \ref{PBWtheorem} the associated graded algebra $\Gr_{\mathfrak{P}}\!\HCoha_{Q,W,\theta}^{\SP,G,\zeta,\psi}$ is supercommutative, and so $\mathfrak{g}_{Q,W,\theta}^{\SP,G,\zeta,\psi}$ is closed under the commutator bracket in $\HCoha_{Q,W,\theta}^{\SP,G,\zeta,\psi}$.  The resulting Lie algebra is called the \textbf{BPS Lie algebra} \cite{QEAs}.  Since we have an inclusion of Lie algebras $\mathfrak{g}_{Q,W,\theta}^{\SP,G,\zeta,\psi}\hookrightarrow  \HCoha_{Q,W,\theta}^{\SP,G,\zeta,\psi}$ where the Lie bracket on the target is the commutator, by the universal property of the enveloping algebra there is a unique morphism of algebras 
\[
\tau\colon \UEA_{\HG}(\mathfrak{g}_{Q,W,\theta}^{\SP,G,\zeta,\psi})\rightarrow \HCoha_{Q,W,\theta}^{\SP,G,\zeta,\psi}
\]
extending this embedding.

\begin{proposition}
\label{UEApart}
The universal map $\tau$ is an inclusion of algebras.
\end{proposition}
\begin{proof}
The projection $\BPS_{Q,W,\theta}^{\SP,G,\zeta}\otimes\HO(\B \CC^*,\QQ)_{\vir}\rightarrow \mathfrak{g}_{Q,W,\theta}^{\SP,G,\zeta,\psi}$ induces a morphism
\[
\pi\colon \Sym_{\HG}\!\left(\BPS_{Q,W,\theta}^{\SP,G,\zeta}\otimes\HO(\B \CC^*,\QQ)_{\vir}\right)\rightarrow \Sym_{\HG}\!\left( \mathfrak{g}_{Q,W,\theta}^{\SP,G,\zeta,\psi} \right)
\]
which is a left inverse to the morphism
\[
\Sym_{\HG}\!\left( \mathfrak{g}_{Q,W,\theta}^{\SP,G,\zeta,\psi} \right)\rightarrow \Sym_{\HG}\!\left(\BPS_{Q,W,\theta}^{\SP,G,\zeta}\otimes\HO(\B \CC^*,\QQ)_{\vir}\right)
\]
induced by the inclusion $\mathfrak{g}_{Q,W,\theta}^{\SP,G,\zeta,\psi}\hookrightarrow \BPS_{Q,W,\theta}^{\SP,G,\zeta}\otimes\HO(\B \CC^*,\QQ)_{\vir}$.  We obtain the commutative diagram of $\dvst$-graded cohomologically graded mixed Hodge structures
\[
\xymatrix{
\UEA_{\HG}(\mathfrak{g}_{Q,W,\theta}^{\SP,G,\zeta,\psi})\ar[r]^{\tau}&\HCoha_{Q,W,\theta}^{\SP,G,\zeta,\psi}\ar[r]^-{\Phi^{-1}}&\Sym_{\HG}\!\left(\mathfrak{g}_{Q,W,\theta}^{\SP,G,\zeta,\psi}\otimes\HO(\B \CC^*,\QQ)\right)\ar[d]^\pi\\
\ar[u]^{\mathrm{PBW}}_\cong\Sym_{\HG}\!\left(\mathfrak{g}_{Q,W,\theta}^{\SP,G,\zeta,\psi}\right)\ar[rr]^=&&\Sym_{\HG}\!(\mathfrak{g}_{Q,W,\theta}^{\SP,G,\zeta,\psi})
}
\]
so that $\tau$ is indeed injective.
\end{proof}

\section{Preprojective CoHAs}
\label{pcsec}
\subsection{The 2-dimensional approach}
\label{2desc}
Given a quiver $Q$ we define the doubled quiver $\ol{Q}$ by setting $\ol{Q}_0=Q_0$ and $\ol{Q}_1=Q_1\coprod Q_1^{\opp}$, where $Q_1^{\opp}$ is the set $\{a^*\colon a\in Q_1\}$, and we set $s(a^*)=t(a)$ and $t(a^*)=s(a)$.  We define the \textbf{preprojective algebra} as in the introduction: $\Pi_Q\coloneqq \CC \ol{Q}/ \langle \sum_{a\in Q_1}[a,a^*]\rangle$.  For each $i,j\in Q_0$ let $V_{i,j}$ be the vector space with basis given by the set of arrows from $i$ to $j$.  We set
\begin{align}
\Gl'_{\edge}\coloneqq\prod_{i\neq j}\Gl(V_{i,j}) \times \prod_{i}\Symp(V_{i,i});\quad\quad
\Gl_{\edge}\coloneqq \Gl'_{\edge}\times \CC^*_{\hbar}\label{Gledge}
\end{align}
where $\CC^*_{\hbar}$ is a copy of $\CC^*$.  Decomposing
\begin{align*}
\AS_{\dd}(\ol{Q})=&\prod_{i\neq j} \left(V_{i,j}\otimes \Hom(\CC^{\dd_{s(a)}},\CC^{\dd_{t(a)}})\right)^*\times \left(V_{i,j}\otimes \Hom(\CC^{\dd_{s(a)}},\CC^{\dd_{t(a)}})\right)\\ \times &\prod_i ((V_{i,i}\oplus V_{i,i}^*)\otimes\Hom(\CC^{\dd_i},\CC^{\dd_i}))
\end{align*}
it follows that $\AS_{\dd}(\ol{Q})$ carries an action of $\Gl'_{\edge}$ preserving the natural symplectic form.  We let $\CC^*_{\hbar}$ act by scaling all of $\AS_{\dd}(\ol{Q})$, so that it acts with weight two on the symplectic form.  In the following, we assume that the gauge group action $G\rightarrow \GQ{\ol{Q}}$ factors through the morphism $\Gl_{\edge}\rightarrow \GQ{\ol{Q}}$ that we have defined here.

We denote by 
\[
\oplus^G_{\red}\colon \Msp_{\theta}^{G,\zeta\sst}(\ol{Q})\times_{\B G}\Msp_{\theta}^{G,\zeta\sst}(\ol{Q})\rightarrow \Msp_{\theta}^{G,\zeta\sst}(\ol{Q})
\]
the morphism taking a pair of polystable $\CC\ol{Q}$-modules to their direct sum.
\subsubsection{Serre subcategories}
Let $\Sp$ be a Serre subcategory of the category of $\CC\ol{Q}$-modules  As in \cite{BSV17} we are mostly interested in the following examples:
\begin{enumerate}
\item
$\mathcal{N}$ is the full subcategory of $\CC\ol{Q}$-modules $\rho$ for which there is a flag of $Q_0$-graded subspaces $0\subset L_1\ldots\subset L$ of the underlying vector space of $\rho$ such that $\rho(a)(L^i)\subset L^{i-1}$ and $\rho(a^*)(L^i)\subset L^{i-1}$ for every $a\in Q_1$.
\item
$\mathcal{SN}$ is the full subcategory of $\CC\ol{Q}$-modules $\rho$ for which there is a flag of $Q_0$-graded subspaces as above, satisfying the weaker condition that $\rho(a)(L^i)\subset L^{i-1}$ and $\rho(a^*)(L^i)\subset L^i$.
\item
$\mathcal{SSN}$ is the full subcategory of $\CC\ol{Q}$-modules satisfying the same conditions as for $\SN$, but with the added condition that each of the subquotients $L^i/L^{i-1}$ is supported at a single vertex.
\end{enumerate}

Let 
\[
\varpi_{\red}\colon \Mst^{\Sp,G,\zeta\sst}(\ol{Q})\rightarrow \Mst^{G,\zeta\sst}(\ol{Q});\quad\quad
\varpi'_{\red}\colon \Msp^{\Sp,G,\zeta\sst}(\ol{Q})\rightarrow \Msp^{G,\zeta\sst}(\ol{Q})
\]
denote the inclusion of the stack, or respectively the stack-theoretic quotient of the coarse moduli space, of modules in $\Sp$.  We denote by
\[
\iota \colon \Mst^{G,\zeta\sst}(\Pi_Q)\rightarrow \Mst^{G,\zeta\sst}(\ol{Q});\quad\quad
\iota'\colon \Msp^{G,\zeta\sst}(\Pi_Q)\rightarrow \Msp^{G,\zeta\sst}(\ol{Q})
\]
the natural closed embeddings.  Fix a slope $\theta\in\QQ$.  We define
\begin{align*}
\HA_{\Pi_Q,\dd}^{\SP,G,\zeta}\coloneqq&\varpi_{\red,*}\varpi_{\red}^!\iota_*\iota^!\ul{\QQ}_{\Mst_{\dd}^{G,\zeta\sst}(\ol{Q})}\otimes\LLL^{\chi_{\WT{Q}}(\dd,\dd)/2} \\
\HCoha_{\Pi_Q,\dd}^{\Sp,G,\zeta}\coloneqq &\HO\!\left(\Mst_{\dd}^{G,\zeta\sst}(\ol{Q}),\HA_{\Pi_Q,\dd}^{\SP,G,\zeta}\right)\\
\HCoha_{\Pi_Q,\theta}^{\Sp,G,\zeta}\coloneqq &\bigoplus_{\dd\in\dvst}\HCoha_{\Pi_Q,\dd}^{\Sp,G,\zeta}.
\end{align*}
\begin{remark}
Since $\chi_{\WT{Q}}(\cdot,\cdot)$ only takes even values, these are genuine mixed Hodge structures, as opposed to monodromic mixed Hodge structures.  
\end{remark}

By (shifted) self Verdier duality of $\ul{\QQ}_{\Mst_{\dd}^{G,\zeta\sst}(\ol{Q})}$, there is an isomorphism
\[
\HCoha_{\Pi_Q,\dd}^{\Sp,G,\zeta}\cong \HO^{\BM}(\Mst^{\SP,G,\zeta\sst}_{\dd}(\Pi_Q),\QQ)\otimes\LLL^{-\dim(G)-\chi_Q(\dd,\dd)}.
\]

\begin{assumption}
\label{2dassumption}
We will always choose $\SP$ so that $\HCoha_{\Pi_Q,\theta}^{\Sp,G,\zeta}$ is free as a $\HG$-module.
\end{assumption}
It is a consequence of purity that this assumption holds if we set $\SP$ to be any of $\CC\ol{Q}\lmod$, $\mathcal{N}$, $\mathcal{SN}$ or $\mathcal{SSN}$ --- see \cite{preproj}, \cite{ScVa20} for further details.
\sssct
The $\dvst$-graded mixed Hodge structure $\HCoha_{\Pi_Q,\theta}^{\Sp,G,\zeta}$ carries a Hall algebra structure introduced by Schiffmann and Vasserot in the case of the Jordan quiver \cite{ScVa13}.  It is defined in terms of correspondences.  Since the algebra defined this way is isomorphic \cite{RS17,YZ18} to the critical CoHA introduced in \S \ref{CCSec} we refrain from giving this definition, instead referring the reader to \cite[Sec.4]{ScVa13} and \cite{YZ18} for details.  We define this algebra structure below, incorporating a sign twist as in \S \ref{psitwists}, so that we can prove the PBW theorem for this algebra.

Likewise if we set 
\[
\rCoha_{\Pi_Q,\dd}^{\Sp,G,\zeta}\coloneqq \JH_{\red}^G\varpi_{\red,*}\varpi_{\red}^!\iota_*\iota^!\ul{\QQ}_{\Mst_{\dd}^{G,\zeta\sst}(\ol{Q})}\otimes\LLL^{\chi_{\WT{Q}}(\dd,\dd)/2};\quad\quad
\rCoha_{\Pi_Q,\theta}^{\Sp,G,\zeta}\coloneqq \bigoplus_{\dd\in\dvst}\rCoha_{\Pi_Q,\dd}^{\Sp,G,\zeta}
\]
the correspondence diagrams that are used to define the Hall algebra structure on $\HCoha_{\Pi_Q,\theta}^{\Sp,G,\zeta}$ can be used to define an algebra structure on $\rCoha_{\Pi_Q,\theta}^{\Sp,G,\zeta}$ with respect to the monoidal structure $\boxtimes_{\oplus_{\red}^G}$.  Since this algebra object will again be isomorphic to the direct image of an algebra object $\rCoha^{\WT{\SP},G,\zeta,\chi}_{\WT{Q},\WT{W},\theta}$ for $\WT{\SP},\WT{Q},\WT{W}$ given in \S \ref{3desc} we do not recall the definition here, instead relying on the next section to provide us with a definition of the algebra structure.

\subsection{The 3-dimensional description}
\label{3desc}
For the description of the preprojective CoHA in terms of vanishing cycles, we introduce a particular class of quivers with potential.   We start with a quiver $Q$ (not assumed to be symmetric).  Then we define the tripled quiver $\WT{Q}$ to be the quiver $\ol{Q}$, with an additional set of edge-loops $\Omega=\{\omega_i\colon i\in Q_0\}$ added to the set of arrows $\ol{Q}_1$, where $s(\omega_i)=t(\omega_i)=i$.  The quiver $\WT{Q}$ is symmetric.  For all $\dd',\dd''\in\mathbb{N}^{Q_0}$ we have the equality
\begin{align}
\chi_Q(\dd',\dd'')+\chi_Q(\dd'',\dd')=\chi_{\tilde{Q}}(\dd',\dd')\chi_{\tilde{Q}}(\dd'',\dd'')+\chi_{\tilde{Q}}(\dd',\dd'')& \quad\textrm{modulo }2
\end{align}
so that $\chi_Q$ satisfies the equation \eqref{psidef} defining allowable choices of $\psi$-twist in \S \ref{psitwists}.  In what follows we use the superscript $\chi$ to denote this choice of bilinear form $\psi$.

We extend the action of $\Gl_{\edge}$ to an action on
\begin{equation}
\label{WTdecomp}
\AS_{\dd}(\WT{Q})\cong \AS_{\dd}(\ol{Q})\times\prod_{i\in Q_0}\gl_{\dd_i}
\end{equation}
by letting $\Gl'_{\edge}$ act trivially on $\prod_{i\in Q_0}\gl_{\dd_i}$ and letting $\CC_{\hbar}^*$ act on $\prod_{i\in Q_0}\gl_{\dd_i}$ with weight $-2$.  In what follows, we assume that the $G$-action, defined by some morphism $G\rightarrow \GQ{\WT{Q}}$, factors through the inclusion of $\Gl_{\edge}$.

We fix
\begin{equation}
\label{WTWdef}
\WT{W}=\sum_{a\in Q_1}[a,a^*]\sum_{i\in Q_0}\omega_i.
\end{equation}
The function $\Tr(\WT{W})$ is $\Gl_{\edge}$-invariant, and thus induces a function $\TTTr(\WT{W})$ on $\Mst^{G}(\WT{Q})$.

We denote by
\[
r\colon\Mst^{G}(\WT{Q})\rightarrow \Mst^{G}(\ol{Q})
\]
the forgetful map taking a $\CC \WT{Q}$-module to its underlying $\CC\ol{Q}$-module.  This morphism is the projection map from the total space of a vector bundle.  The function $\mathfrak{T}r(\WT{W})$ has weight one with respect to the function that scales along the fibres of $r$.

Recall that we denote by $\iota\colon \Mst^{G}(\Pi_Q)\hookrightarrow \Mst^G(\ol{Q})$ the inclusion of the substack of representations satisfying the preprojective algebra relations.  Then $\iota$ is also the inclusion of the set of points $x$ for which $r^{-1}(x)\subset\mathfrak{T}r(\WT{W})^{-1}(0)$.  By \cite[Thm.A.1]{Da13} there is a natural isomorphism, the ``dimensional reduction isomorphism'':
\begin{equation}
\label{DRI}
\iota_*\iota^!\rightarrow r_*\phim{\mathfrak{T}r(\WT{W})}r^*.
\end{equation}
\sssct
Let $\Sp$ be a Serre subcategory of the category of $\CC \ol{Q}$-modules.  We denote by $\WT{\SP}$ the Serre subcategory of the category of $\CC \WT{Q}$-modules $\rho$ satisfying the condition that the underlying $\CC\ol{Q}$-module of $\rho$ is an object of $\SP$.  As in \S \ref{CCSec} we define
\begin{align*}
\HA^{\WT{\SP},G,\zeta}_{\WT{Q},\WT{W},\theta}\coloneqq&\varpi_*\varpi^!\phim{\mathfrak{T}r(\WT{W})}\nnIC_{\Mst^{G,\zeta\sst}(\WT{Q})}\\
\rCoha^{\WT{\SP},G,\zeta,\chi}_{\WT{Q},\WT{W},\theta}\coloneqq&\JH^G_{*}\HA^{\WT{\SP},G,\zeta}_{\WT{Q},\WT{W},\theta}\\
\HCoha^{\WT{\SP},G,\zeta,\chi}_{\WT{Q},\WT{W},\theta}\coloneqq&\HO\!\left(\Mst^{G,\zeta\sst}(\WT{Q}),\HA^{\WT{\SP},G,\zeta}_{\WT{Q},\WT{W},\theta}\right)
\end{align*}
where the last two objects carry algebra structures via the diagram of correspondences \eqref{3dc}, with $\psi$-twist provided by $\chi$.  \subsubsection{Stability conditions and dimensional reduction}
Via the dimensional reduction isomorphism \eqref{DRI} there is a natural isomorphism
\[
r_*\HA^{G}_{\WT{Q},\WT{W},\dd}\cong\iota_*\iota^!\ul{\QQ}_{\Mst^G(\ol{Q})}\otimes \LLL^{\chi_{\WT{Q}}(\dd,\dd)/2}.
\]
Applying $\varpi_{\red,*}\varpi^!_{\red}$ and base change to this isomorphism, gives a natural isomorphism
\begin{equation}
\label{withvp}
r_*\HA^{\WT{\SP},G}_{\WT{Q},\WT{W},\dd}\cong\varpi_{\red,*}\varpi_{\red}^!\iota_*\iota^!\ul{\QQ}_{\Mst^G(\ol{Q})}\otimes \LLL^{\chi_{\WT{Q}}(\dd,\dd)/2}.
\end{equation}
We would like to be able to incorporate stability conditions into isomorphism \eqref{withvp} but there is an obvious problem: far from being the projection from a total space of a vector bundle, the forgetful morphism from $\Mst^{G,\zeta\sst}_{\dd}(\WT{Q})\rightarrow \Mst^{G,\zeta\sst}_{\dd}(\ol{Q})$ is not even defined!  This is because the underlying $\CC\ol{Q}$-module of a $\zeta$-semistable $\CC \WT{Q}$-module may be unstable.  On the way to resolving the problem, we define
\[
\Mst^{G,\zeta\ssst}_{\dd}(\WT{Q})\coloneqq r_{\dd}^{-1}(\Mst^{G,\zeta\sst}_{\dd}(\ol{Q})).
\]
Then the morphism $r_{\dd}^{\circ}\colon \Mst^{G,\zeta\ssst}_{\dd}(\WT{Q})\rightarrow \Mst^{G,\zeta\sst}_{\dd}(\ol{Q})$ obtained by restricting $r_{\dd}$ is the projection from the total space of a vector bundle, as required in the statement of the dimensional reduction theorem.  We will resolve the above problem by use of the following helpful fact.
\begin{proposition}\cite[Lem.6.5]{preproj}\label{SSL}
The critical locus of the function $\TTTr(\WT{W})$ on $\Mst^{G,\zeta\sst}(\WT{Q})$ lies inside $\Mst^{G,\zeta\ssst}(\WT{Q})$.  As a consequence, the support of $\HA^{\WT{\SP},G,\zeta}_{\WT{Q},\WT{W},\dd}$ is contained in $\Mst^{G,\zeta\ssst}_{\dd}(\WT{Q})$.
\end{proposition}
\subsubsection{The absolute CoHA}
\label{3dAbs}
Let $\kappa\colon \Mst^{G,\zeta\ssst}_{\dd}(\WT{Q})\hookrightarrow \Mst^{G,\zeta\sst}_{\dd}(\WT{Q})$ be the open embedding.  We define
\[
\HA^{\circ,\WT{\SP},G,\zeta}_{\WT{Q},\WT{W},\dd}\coloneqq \kappa^*\HA^{\WT{\SP},G,\zeta}_{\WT{Q},\WT{W},\dd}.
\]
By dimensional reduction \eqref{DRI} there is an isomorphism
\begin{equation}
\label{DRI2}
r^{\circ}_{\dd,*}\HA^{\circ,\WT{\SP},G,\zeta}_{\WT{Q},\WT{W},\dd}\cong\HA_{\Pi_Q,\dd}^{\SP,G,\zeta},
\end{equation}
and so there is an isomorphism of $\HG$-modules
\begin{equation}
\label{decr}
\HO\!\left(\Mst^{G,\zeta\ssst}_{\theta}(\WT{Q}),\HA^{\circ,\WT{\SP},G,\zeta}_{\WT{Q},\WT{W},\theta}\right)\cong \HCoha_{\Pi_Q,\theta}^{\Sp,G,\zeta,\chi}.
\end{equation}
On the other hand by Proposition \ref{SSL} we deduce that there are isomorphisms
\begin{align}
\label{deca}\HO\!\left(\Mst^{G,\zeta\ssst}_{\theta}(\WT{Q}),\HA^{\circ,\WT{\SP},G,\zeta}_{\WT{Q},\WT{W},\theta}\right)\cong &\HO\!\left(\Mst_{\theta}^{G,\zeta\sst}(\WT{Q}),\HA^{\WT{\SP},G,\zeta}_{\WT{Q},\WT{W},\theta}\right)\\
=&\HCoha_{\WT{Q},\WT{W},\theta}^{\WT{\SP},G,\zeta,\chi}.\nonumber
\end{align}
Combining \eqref{deca} and \eqref{decr} yields the following isomorphism of $\HG$-modules in the category of $\dvst$-graded, cohomologically graded mixed Hodge structures:
\begin{equation}
\label{drnoalg}
\HCoha_{\Pi_Q,\theta}^{\Sp,G,\zeta}\cong \HCoha_{\WT{Q},\WT{W},\theta}^{\WT{\SP},G,\zeta,\chi}.
\end{equation}
As such, $\HCoha_{\Pi_Q,\theta}^{\SP,G,\zeta}$ inherits a $\HG$-linear algebra structure from the algebra structure on $\HCoha_{\WT{Q},\WT{W},\theta}^{\WT{\SP},G,\zeta,\chi}$.

\subsubsection{The relative CoHA}
\label{2drelCoHA}
In this section we lift the absolute CoHA constructed in \S \ref{3dAbs} to an algebra structure on the object $\rCoha_{\Pi_Q,\theta}^{\Sp,G,\zeta}$ in the category $\DlMHM(\Msp^{G,\zeta\sst}_{\theta}(\ol{Q}))$.

We denote by $\kappa'\colon \Msp^{G,\zeta\ssst}(\WT{Q})\rightarrow \Msp^{G,\zeta\sst}(\WT{Q})$ the inclusion of the substack\footnote{It is not hard to show that this is an open substack; we leave the proof to the reader.} of $\CC\WT{Q}$-modules for which the underlying $\CC\ol{Q}$-module is $\zeta$-semistable, and we denote by
\[
r'\colon \Msp^{G,\zeta\ssst}(\WT{Q})\rightarrow \Msp^{G,\zeta\sst}(\ol{Q})
\]
the forgetful morphism.  To make it easier to keep track of them all, we arrange some of the morphisms introduced in this section into a commutative diagram\footnote{We indicate the version where $G=\{1\}$.  In general, there should be $G$ superscripts everywhere.}:
\begin{align}
\label{drcd}
\xymatrix{
\Mst(\WT{Q})\ar[rr]^r&&\Mst(\ol{Q})
\\
\ar@{^{(}->}[u]\Mst^{\zeta\sst}(\WT{Q})\ar[d]^{\JH}&\ar@{_{(}->}[l]_{\kappa}\Mst^{\zeta\ssst}(\WT{Q})\ar[d]^{\JH^{\circ}}\ar[r]^-{r^{\circ}}& \Mst^{\zeta\sst}(\ol{Q})\ar[d]^{\JH_{\red}}\ar@{^{(}->}[u]\\
\Msp^{\zeta\sst}(\WT{Q})&\ar@{_{(}->}[l]_{\kappa'}\Msp^{\zeta\ssst}(\WT{Q})\ar[r]^-{r'}&\Msp^{\zeta\sst}(\ol{Q}).
}
\end{align}
By dimensional reduction there is a natural isomorphism
\begin{equation}
\label{2DIR}
r'_{*}\kappa'^*\rCoha_{\WT{Q},\WT{W},\dd}^{\SP,G,\zeta,\chi}\cong \rCoha_{\Pi_Q,\dd}^{\Sp,G,\zeta}
\end{equation}
obtained via commutativity of the diagram \eqref{drcd}.  Since $\kappa'$ and $r'$ are morphisms of monoids, and $\kappa'$ is smooth, the object $\rCoha_{\Pi_Q,\dd}^{\Sp,G,\zeta}$ inherits an algebra structure, via the isomorphism \eqref{2DIR}, as promised in \S \ref{2desc}.  Furthermore, by Proposition \ref{SSL} we obtain the first of the isomorphisms
\begin{align*}
\HO(\Msp^{G,\zeta\ssst}(\WT{Q}),\kappa'^*\rCoha_{\WT{Q},\WT{W}}^{\SP,G,\zeta,\chi})&\cong \HO(\Msp^{G,\zeta\sst}(\WT{Q}),\rCoha_{\WT{Q},\WT{W}}^{\SP,G,\zeta,\chi})\\
&\cong \HCoha_{\Pi_Q,\theta}^{\Sp,G,\zeta}
\end{align*}
(the second is \eqref{drnoalg}).  The Hall algebra structure on $\HCoha_{\Pi_Q,\theta}^{\Sp,G,\zeta}$ comes from applying $\HO$ to the Hall algebra structure on $\rCoha_{\Pi_Q,\theta}^{\Sp,G,\zeta}$, i.e. $\rCoha_{\Pi_Q,\theta}^{\Sp,G,\zeta}$ is a lift of the CoHA $\HCoha_{\Pi_Q,\theta}^{\Sp,G,\zeta}$ to the category $\DlMHM(\Msp^{G,\zeta\sst}_{\theta}(\ol{Q}))$.  Since the support of $\rCoha_{\Pi_Q,\theta}^{\Sp,G,\zeta}$ clearly lies in $\Msp^{G,\zeta\sst}_{\theta}(\Pi_Q)$, we may alternatively consider it as a lift of the CoHA to $\DlMHM(\Msp^{G,\zeta\sst}_{\theta}(\Pi_Q))$.

\section{Results on BPS sheaves}
\label{bpsss}
\subsection{Generalities on BPS sheaves for $(\WT{Q},\WT{W})$}
Let $Q$ be a quiver, then we define $\WT{Q}$ and $\WT{W}$ as in \S \ref{3desc}, pick a stability condition $\zeta\in\QQ^{Q_0}$ and a slope $\theta\in \QQ$, as well as an extra gauge group $G$, along with a homomorphism $G\rightarrow \Gl_{\edge}$ as in \S \ref{2desc}.  We furthermore pick a $G$-invariant Serre subcategory $\SP$ of the category of $\CC\ol{Q}$-modules, satisfying Assumption \ref{2dassumption}, and define $\WT{\SP}$ as in \S \ref{3desc}.  Then $\WT{\SP}$ satisfies Assumption \ref{3dassumption} via the isomorphism \eqref{drnoalg}.  With this data fixed, we define the BPS sheaf
\[
\BPSh_{\WT{Q},\WT{W},\theta}^{\WT{\SP},G,\zeta}\in\DbMMHM(\Msp_{\theta}^{G,\zeta\sst}(\WT{Q}))
\]
as in \eqref{BPShdef}.  If $\SP=\CC\ol{Q}\lmod$, so $\WT{\SP}=\CC\WT{Q}\lmod$, this is a $G$-equivariant monodromic mixed Hodge module, otherwise, it may be a complex of monodromic mixed Hodge modules with cohomology in several degrees.

\subsubsection{The 2d BPS sheaf}

We define $\Gl_{\edge}$ as in \eqref{Gledge}.  We let $\Gl_{\edge}$ act on $\AAA{1}$ via the projection to $\CC^*_{\hbar}$, and the weight -2 action of $\CC^*_{\hbar}$ on $\AAA{1}$.  The inclusion 
\[
\AAA{1}\rightarrow \gl_{\dd};\quad\quad
t\mapsto (t\cdot \Id_{\CC^{\dd_i}})_{i\in Q_0},
\]
along with the decomposition \eqref{WTdecomp}, induces a $\Gl_{\edge}$-equivariant inclusion $\AS_{\dd}(\ol{Q})\times\AAA{1}\hookrightarrow \AS_{\dd}(\WT{Q})$.
This induces the inclusion
\[
l\colon \Msp^{G,\zeta\sst}(\ol{Q})\times_{\B G}(\AAA{1}/G)\hookrightarrow \Msp^{G,\zeta\sst}(\WT{Q}).
\]
We denote the projection by 
\[
h\colon \Msp^{G,\zeta\sst}(\ol{Q})\times_{\B G}(\AAA{1}/G)\rightarrow \Msp^{G,\zeta\sst}(\ol{Q}).
\]
The following theorem is proved as in \cite[Lem.4.1]{preproj}.
%, though see \cite{preproj1} for the adjustments necessary to incorporate the additional data of $\SP,G,\zeta$.
\begin{thmdef}
\label{2dBPSdef}
There exists a unique (up to isomorphism) object
\begin{equation}
\label{rbps1}
\BPSh_{\Pi_Q,\theta}^{\SP,G,\zeta}\in\DbMMHM(\Msp^{G,\zeta\sst}_{\theta}(\ol{Q}))
\end{equation}
along with an isomorphism
\begin{equation}
\label{rbps2}
\BPSh_{\WT{Q},\WT{W},\theta}^{\WT{\SP},G,\zeta}\cong l_*(h^*\BPSh_{\Pi_Q,\theta}^{\SP,G,\zeta}\otimes\LLL^{-1/2}).
\end{equation}
\end{thmdef}
In words, the theorem says that $\BPSh_{\WT{Q},\WT{W},\theta}^{\WT{\SP},G,\zeta}$ is supported on the locus containing those $\CC\WT{Q}$-modules for which all of the generalised eigenvalues of all of the operators $\omega_i\cdot$ are the same complex number $t$, and the sheaf does not depend on this complex number.

By \eqref{rbps2} there is an isomorphism of $\dvst$-graded, cohomologically graded mixed Hodge structures
\begin{equation}
\label{rbps3}
\HO\!\left(\Msp^{G,\zeta\sst}_\theta(\WT{Q}),\BPSh_{\WT{Q},\WT{W},\theta}^{\WT{\SP},G,\zeta}\right)\otimes\LLL^{1/2}\cong \HO\!\left(\Msp^{G,\zeta\sst}_\theta(\ol{Q}),\BPSh_{\Pi_Q,\theta}^{\SP,G,\zeta}\right).
\end{equation}
\begin{definition}
We define the Lie algebra\footnote{Since, in defining the Lie algebra $\fg_{\Pi_Q,\theta}^{\SP,G,\zeta}$ we fix the $\psi$-twist $\chi$, we omit it from the notation.}
\[
\fg_{\Pi_Q,\theta}^{\SP,G,\zeta}\coloneqq \HO\!\left(\Msp^{G,\zeta\sst}_\theta(\ol{Q}),\BPSh_{\Pi_Q,\theta}^{\SP,G,\zeta}\right).
\]
The Lie algebra structure is induced by isomorphism \eqref{rbps3} and the Lie algebra structure on
\[
\fg_{\Pi_Q,\theta}^{\SP,G,\zeta}\cong \fg_{\WT{Q},\WT{W},\theta}^{\WT{\SP},G,\zeta,\chi}\cong \HO\!\left(\Msp^G_{\theta}(\WT{Q}),\BPSh_{\WT{Q},\WT{W},\theta}^{\SP,G,\zeta}\right)\otimes\LLL^{1/2}.
\]
\end{definition}

Combining with \eqref{drnoalg} and \eqref{absPBW} there is a PBW isomorphism
\begin{equation}
\label{rPBW}
\Sym_{\HG}\!\left(\fg_{\Pi_Q,\theta}^{\SP,G,\zeta}\otimes \HO(\B\CC^*,\QQ)\right) \rightarrow \HCoha^{\SP,G,\zeta}_{\Pi_Q,\theta}.
\end{equation}
\begin{remark}
In contrast with \eqref{absPBW} there is no half Tate twist in \eqref{rPBW}, and all of the terms on the right hand side of \eqref{rPBW} are defined as mixed Hodge structures without any monodromy.
\end{remark}
We note that the image of $l$ lies within $\Mst^{G,\zeta\ssst}(\WT{Q})$, and thus there is an isomorphism
\begin{equation}
\label{pfBP}
r'_*\kappa'^*\BPSh^{\WT{\SP},G,\zeta}_{\WT{Q},\WT{W},\theta}\otimes\LLL^{1/2}\cong\BPSh_{\Pi_Q,\theta}^{\SP,G,\zeta}
\end{equation}
and so, via \eqref{2DIR} and the PBW theorem \eqref{relPBW}, an isomorphism
\begin{equation}
\label{2dInt}
\Sym_{\oplus_{\red}^G}\!\left( \BPSh_{\Pi_Q,\theta}^{\SP,G,\zeta}\otimes \HO(\B\CC^*,\QQ)\right)\cong \rCoha_{\Pi_Q,\theta}^{\SP,G,\zeta}
\end{equation}
lifting \eqref{rPBW}.

\subsection{Restricted Kac polynomials}
In this section we assume that $G$ is trivial, so we drop it from the notation.  Also, we will work with the degenerate stability condition $\zeta=(0,\ldots,0)$ and slope $\theta=0$, so that we may drop $\zeta$ and $\theta$ from the notation too.  We will explain the connection between the BPS Lie algebra $\mathfrak{g}^{\SP}_{\Pi_Q}\cong \HO\!\left(\Msp(\WT{Q}),\BPSh_{\WT{Q},\WT{W}}^{\WT{\SP}}\right)\otimes\LLL^{1/2}$ and (possibly restricted) Kac polynomials.  
\sssct
In the case in which $\SP=\CC\ol{Q}\lmod$, it is proved in \cite{preproj} that $\mathfrak{g}^{\SP}_{\Pi_Q}$ is pure, of Tate type, and has vanishing odd cohomology.  Thus we have the equality of polynomials
\begin{align}
\label{pureKac}
\chi_t(\mathfrak{g}_{\Pi_Q,\dd})= &\wt(\mathfrak{g}_{\Pi_Q,\dd}),
\end{align}
with definitions as in \eqref{PPdef} and \eqref{VPPdef}.  By \cite{Moz11} there is an equality
\begin{equation}
\label{MozId}
\wt(\mathfrak{g}_{\Pi_Q,\dd})=\kac_{Q,\dd}(t^{-1})
\end{equation}
where $\kac_{Q,\dd}(t)$ is the Kac polynomial for $Q$ \cite{Kac83}, defined to be the polynomial such that if $q=p^r$ is a prime power, $\kac_{Q,\dd}(q)$ is the number of isomorphism classes of absolutely indecomposable $\dd$-dimensional $\CC Q$-modules.  Combining \eqref{pureKac} and \eqref{MozId} we deduce
\begin{equation}
\label{Kaccha}
\chi_t(\mathfrak{g}_{\Pi_Q,\dd})=\kac_{Q,\dd}(t^{-1}).
\end{equation}

Similarly, by \cite[Sec.7.2]{preproj} the mixed Hodge structures on $\mathfrak{g}_{\Pi_Q}^{\mathcal{SN}},\mathfrak{g}_{\Pi_Q}^{\mathcal{SSN}}$ are pure, of Tate type.  In a little more detail, by purity of $\HCoha_{\Pi_Q}^{\mathcal{SN}}$ and $\HCoha_{\Pi_Q}^{\mathcal{SSN}}$, proved in \cite[Sec.4.3]{ScVa20}, along with the PBW theorem \eqref{absPBW}, we deduce that $\mathfrak{g}^{\mathcal{SN}}_{\Pi_Q,\dd}$ and $\mathfrak{g}^{\mathcal{SSN}}_{\Pi_Q,\dd}$ are pure of Tate type, since they are subobjects of pure mixed Hodge structures of Tate type.
\sssct
There are analogues of the Kac polynomials for these Serre subcategories.  We recall from \cite{BSV17} that a representation of $Q$ is called 1-nilpotent if there is a flag $0=F_0\subset F_1\subset\ldots\subset F_r=\CC^{\dd_i}$ for every $i\in Q_0$ such that $\rho(a)(F_n)\subset F_{n-1}$ for every $n$, for every $a$ an edge-loop at $i$.  We define $\kac_{Q,\dd}^{\mathcal{SN}}(t)$ to be the polynomial counting the isomorphism classes of absolutely indecomposable 1-nilpotent $\dd$-dimensional $\mathbb{F}_qQ$-modules, and $\kac_{Q,\dd}^{\mathcal{SSN}}(t)$ to be the analogous count of absolutely indecomposable nilpotent representations.  Then by \cite{BSV17} (see also \cite[Sec.7.2]{preproj} for details on the passage to BPS cohomology) there are identities
\begin{align}
\label{kssn0}
\wt(\fg_{\Pi_Q,\dd}^{\mathcal{SN}})=\kac^{\mathcal{SN}}_{Q,\dd}(t);\quad\quad
\wt(\fg_{\Pi_Q,\dd}^{\mathcal{SSN}})=\kac^{\mathcal{SSN}}_{Q,\dd}(t).
\end{align}
On the other hand since the mixed Hodge structures on $\fg_{\Pi_Q,\dd}^{\mathcal{SN}}$ and $\fg_{\Pi_Q,\dd}^{\mathcal{SSN}}$ are pure, their weight polynomials agree with their characteristic polynomials.  So \eqref{kssn0} yields
\begin{align}
\label{kssn}
\chi_t(\fg_{\Pi_Q,\dd}^{\mathcal{SN}})=\kac^{\mathcal{SN}}_{Q,\dd}(t);\quad\quad
\chi_t(\fg_{\Pi_Q,\dd}^{\mathcal{SSN}})=\kac^{\mathcal{SSN}}_{Q,\dd}(t)
\end{align}
respectively.

\subsubsection{Serre relation for BPS algebras}
Let $i,j\in Q_0$ be distinct elements of $Q_0$, and assume that $Q$ has no edge-loops at $i$.  Let $\dd=(e+1)\cdot 1_{i}+1_j$ where $e$ is the number of edges between $i$ and $j$ in the underlying graph of $Q$.  In Proposition \ref{BPSvanProp} we prove a vanishing theorem for BPS sheaves, that strengthens the identity $\HO^0(\fg_{\Pi_Q,\dd})=0$ resulting from the Serre relations in $\fg_Q$ (see \S \ref{KacH} below), by showing that the identity remains true for all Serre subcategories, stability conditions, and cohomological degrees.  Firstly we will need a proposition due to Yukinobu Toda:

\begin{proposition}\cite[Lem.4.7]{To17}
\label{TodaProp}
Let $Q'$ be a symmetric quiver, let $W'\in\CC Q'_{\cyc}$ be a superpotential, let $\zeta\in\mathbb{Q}^{Q_0}$ be a stability condition, and let $\dd\in\dvs$ be a dimension vector.  Let $q_{\dd}\colon \Msp^{\zeta\sst}_{\dd}(Q')\rightarrow \Msp_{\dd}(Q')$ be the affinization map.  Then there is an isomorphism $q_{\dd,*}\BPSh_{Q',W',\dd}^{\zeta}\cong\ \BPSh_{Q',W',\dd}$.
\end{proposition}

\begin{proposition}
\label{BPSvanProp}
Let $\zeta\in\QQ^{Q_0}$ be an arbitrary stability condition, let $\SP$ be arbitrary, and let $\dd=(e+1)\cdot 1_i+1_j$ with $i$, $j$ as above.  There is an identity in $\DbMMHM(\Msp^{\zeta\sst}_{\dd}(\ol{Q}))$
\begin{equation}
\label{BPSvan}
\BPSh_{\Pi_Q,\dd}^{\SP,G,\zeta}=0.
\end{equation}
\end{proposition}
\begin{proof}
Since $\BPSh_{\Pi_Q,\dd}^{\SP,G,\zeta}=\varpi'_{\red,*}\varpi'^!_{\red}\BPSh_{\Pi_Q,\dd}^{G,\zeta}$ it is sufficient to prove \eqref{BPSvan} under the assumption that $\SP=\CC\ol{Q}\lmod$.  In addition, we may assume that $G$ is trivial, since a $G$-equivariant perverse sheaf is trivial if and only if the underlying perverse sheaf is.  

By Theorem/Definition \ref{2dBPSdef}, we may equivalently prove that $\BPSh_{\WT{Q},\WT{W},\dd}^{\zeta}=0$.  There are three cases to consider:
\begin{enumerate}
\item
$\zeta_i<\zeta_j$
\item
$\zeta_i=\zeta_j$
\item
$\zeta_i>\zeta_j$.
\end{enumerate}
The proofs for (1) and (3) are the same, while (2) follows from (1) and the identity
\[
\BPSh_{\WT{Q},\WT{W},\dd}\cong \left(\Msp_{\dd}^{\zeta\sst}(\WT{Q})\rightarrow \Msp_{\dd}(\WT{Q})\right)_*\BPSh_{\WT{Q},\WT{W},\dd}^{\zeta},
\]
which is a special case of Proposition \ref{TodaProp}.  So we concentrate on (1).  We claim that
\begin{equation}
\label{epty}
\Mst_{\dd}^{\zeta\sst}(\WT{Q})\cap \crit(\TTTr(\WT{W}))=\emptyset.
\end{equation}
We first note that a point in the left hand side of \eqref{epty} represents a $\zeta$-semistable $\Jac(\WT{Q},\WT{W})$-module.  By Proposition \ref{SSL}, the underlying $\Pi_Q$-module of $\rho$ is $\zeta$-semistable.  On the other hand, there are no $\zeta$-semistable $\dd$-dimensional $\Pi_Q$-modules $\rho$, as for such a $\rho$ the subspace spanned by $e_j\cdotsh \rho,b_1\cdotsh\rho,\ldots,b_{e}\cdotsh\rho$ is a submodule, where $b_1,\ldots,b_{e}$ are the arrows in $\ol{Q}$ with source $j$ and target $i$.  This proves the claim.

Now the proposition follows from the definition \eqref{BPShdef} and the equality
\[
\supp\!\left(\HA_{\WT{Q},\WT{W},\theta}^{\zeta}\right)=\Mst_{\dd}^{\zeta\sst}(\WT{Q})\cap \crit(\TTTr(\WT{W}))
\]
which follows from the fact that for $f$ a regular function on a smooth space $X$, $\phi_f\ul{\QQ}_X$ is supported on the critical locus of $f$.
\end{proof}

\subsection{Purity of BPS sheaves}
\begin{definition}
Let $\XX$ be a stack.  We say that $\mathcal{G}\in\DlMMHM(\XX)$ is \textit{pure below} if for all integers $m<n$ there is an equality $\Gr^m_{\Wt}\!\left(\Ho^n\!\mathcal{G}\right)= 0$.  Similarly, we say that $\mathcal{G}$ is \textit{pure above} if the equality holds for all $m>n$.  
\end{definition}

Purity is the combination of being pure above and pure below.  If $X$ is a smooth variety then $\HO(X,\QQ)$ is pure below (considered as a mixed Hodge module on a point), while if $X$ is projective, $\HO(X,\QQ)$ is pure above.  By Poincar\'e duality, it follows that $\HO_c(X,\QQ)$ is pure above if $X$ is smooth.  From the long exact sequence in compactly supported cohomology, and the fact that a variety can be stratified into smooth pieces, it follows that $\HO_c(X,\QQ)$ is pure above for \textit{all} varieties $X$.  We will use the following generalisation of this fact.
\begin{lemma}
\label{PureAbove}
Let $\XX$ be a finite type stack.  Let $p\colon \XX\rightarrow \YY$ be a morphism of stacks.  Then $p_!\ul{\QQ}_{\XX}$ is pure above.
\end{lemma}
\begin{proof}
Since $p_!\ul{\QQ}_{\XX}$ only depends on the reduced structure of $\XX$, we may assume that $\XX$ is reduced.  We first claim that $\XX$ can be written as a disjoint union $\XX=\bigcup_{i\in I}\XX_i$ of locally closed smooth substacks, where $I=\{1,\ldots n\}$ is ordered so that $\XX_i$ is open inside 
\[
\XX_{\leq i}\coloneqq \bigcup_{j\leq i}\XX_j.
\]
This follows from the fact that $\XX_{\smooth}$ is smooth and dense inside the reduced stack $\XX$, Noetherian induction, and our assumption that $\XX$ is of finite type.

For $q$ a morphism of varieties, $q_!$ decreases weights \cite[\S 4.5]{Saito90}.  Since, for $q\colon \mathcal{Z}\rightarrow \mathcal{Z}'$ a morphism of stacks, $q_!$ is still defined in terms of morphisms of varieties, it still decreases weights.  Thus $q_!\ul{\QQ}_{\mathcal{Z}}$ is pure above if $\mathcal{Z}$ is smooth.

We define 
\[
p_i\colon \XX_i\rightarrow \YY;\quad\quad
p_{\leq i}\colon \XX_{\leq i}\rightarrow \YY
\]
to be the restrictions of $p$.  Under our assumptions on $\XX$, there are distinguished triangles
\[
p_{i,!}\ul{\QQ}_{\XX_i}\rightarrow p_{\leq i,!}\ul{\QQ}_{\XX_{\leq i}}\rightarrow p_{\leq i-1}\ul{\QQ}_{\XX_{\leq i-1}}.
\]
The first term term is pure above, the last term is pure above by induction on $i$, and so the middle term is pure above, by the long exact sequence in cohomology.  In particular, since $p_{\leq n}=p$, we deduce that $p_!\ul{\QQ}_{\XX}$ is pure above.
\end{proof}

\subsubsection{Main result on BPS sheaves}
We come to our main result on 2d BPS sheaves:
\begin{theorem}
\label{purityThm}
The monodromic mixed Hodge module complex $\BPSh_{\Pi_Q,\theta}^{G,\zeta}\in\DbMMHM(\Msp^{G,\zeta\sst}_{\theta}(\ol{Q}))$ is pure.  Moreover, it is a mixed Hodge module, meaning that it lies in the heart of the natural t structure, and is in the image of the embedding $\MHM(\Msp^{G,\zeta\sst}_{\theta}(\ol{Q}))\rightarrow \MMHM(\Msp^{G,\zeta\sst}_{\theta}(\ol{Q}))$.
\end{theorem}
\begin{proof}
The second statement follows from the isomorphism \eqref{rbps2} in the special case in which $\SP$ is the entire subcategory of $\Pi_Q$-modules, the fact that $\BPSh_{\WT{Q},\WT{W},\theta}^{G,\zeta}$ is a monodromic mixed Hodge module by \eqref{JHvanishing} and \ref{BPShdef}, and the observation that the right hand side of \eqref{2dInt} has is monodromy-free since the constant sheaf is.  So the main job is to prove purity, which we do now.

Setting $q\colon \Msp^{\zeta\sst}_{\theta}(\tilde{Q})\rightarrow \Msp^{G,\zeta\sst}_{\theta}(\tilde{Q})$ to be the quotient map, we have 
\[
\BPSh_{\WT{Q},\WT{W},\theta}^{\zeta}\cong q^*\BPSh_{\WT{Q},\WT{W},\theta}^{G,\zeta}
\]
and so it is enough to prove that $\mathcal{F}=\BPSh_{\WT{Q},\WT{W},\theta}^{\zeta}$ is pure, i.e. we can assume that $G=\{1\}$.  By \eqref{BPSgive} $\mathcal{F}_{\dd}\cong \phim{\TTr(W)}\nIC_{\Msp^{\zeta\sst}_{\dd}(\WT{Q})}$ or $\mathcal{F}_{\dd}=0$.  Since $\phim{\TTr(W)}$ commutes with Verdier duality \cite{Sai89duality}, and $\nIC_{\Msp_{\dd}^{\zeta\sst}(\WT{Q})}\cong\VD_{\Msp_{\dd}^{\zeta\sst}(\WT{Q})}\nIC_{\Msp_{\dd}^{\zeta\sst}(\WT{Q})}$, there is an isomorphism
\begin{equation}
\label{VSD}
\VD_{\Msp_{\dd}^{\zeta\sst}(\WT{Q})}\mathcal{F}_{\dd}\cong\mathcal{F}_{\dd}.
\end{equation}
Now for $X$ a variety and $\mathcal{G},\mathcal{L}$ objects of $\DlMMHM(X)$ and $\MMHM(X)$ respectively, there are isomorphisms
\[
\Ho^n\!\left(\VD_{X}\mathcal{G}\right)\cong \VD_{X}\!\Ho^{-n}\mathcal{G};\quad\quad
\VD_X(\Gr^n_{\Wt}\mathcal{L})\cong\Gr^{-n}_{\Wt}\VD_X\!\left(\mathcal{L}\right)
\]
and so the existence of the isomorphism \eqref{VSD} implies that for $m,n\in\mathbb{Z}$ we have $\Gr^m_{\Wt}\!\left(\Ho^n\mathcal{F}\right)\neq 0$ if and only if $\Gr^{-m}_{\Wt}\!\left(\Ho^{-n}\mathcal{F}\right)\neq 0$.  In particular, $\mathcal{F}$ is pure below if and only if it is pure above.  Since by \eqref{rbps2} we may write $\mathcal{F}\cong l_*\!\left(\BPSh^{\zeta}_{\Pi_Q,\theta}\boxtimes\nIC_{\AAA{1}}\right)$ and $\nIC_{\AAA{1}}$ is pure, we deduce that the same \textit{symmetry of impurity} holds for $\BPSh^{\zeta}_{\Pi_Q,\theta}$: 
\begin{itemize}
\item[($\ast$)]
The mixed Hodge module $\BPSh^{\zeta}_{\Pi_Q,\theta}$ is pure below if and only if it is pure above.
\end{itemize}
We will complete the proof by showing that $\BPSh^{\zeta}_{\Pi_Q,\theta}$ is pure below.  Fix $\dd\in\Lambda^{\zeta}_{\theta}$.  We again consider the commutative diagram \eqref{drcd}.  By \ref{DRI2} there is an isomorphism
\[
r^{\circ}_*\kappa^*\HA^{\zeta}_{\WT{Q},\WT{W},\dd}\cong\iota_*\iota^!\ul{\QQ}_{\Mst_{\dd}^{\zeta\sst}(\ol{Q})}\otimes\LLL^{\chi_{\WT{Q}}(\dd,\dd)/2}
\]
and thus an isomorphism
\begin{align}
\label{dre}
r'_*\JH^{\circ}_*\kappa^*\HA^{\zeta}_{\WT{Q},\WT{W},\dd}\cong \JH_{\red,*}\iota_*\iota^!\ul{\QQ}_{\Mst_{\dd}^{\zeta\sst}(\ol{Q})}\otimes\LLL^{\chi_{\WT{Q}}(\dd,\dd)/2}.
\end{align}
Applying Verdier duality to the right hand side of \eqref{dre}, we get
\begin{equation}
\label{Dbpu}
\VD_{\Msp_{\dd}^{\zeta\sst}(\ol{Q})}\JH_{\red,*}\iota_*\iota^!\ul{\QQ}_{\Mst_{\dd}^{\zeta\sst}(\ol{Q})}\otimes\LLL^{\chi_{\WT{Q}}(\dd,\dd)/2}\cong \JH_{\red,!}\iota_!\ul{\QQ}_{\Mst_{\dd}^{\zeta\sst}(\Pi_Q)}\otimes\LLL^{\chi_{Q}(\dd,\dd)}.
\end{equation}
The Tate twist comes from the calculations (assuming $\Mst_{\dd}^{\zeta\sst}(\ol{Q})$ non-empty)
\[
\dim(\Mst_{\dd}^{\zeta\sst}(\ol{Q}))=-\chi_{\WT{Q}}(\dd,\dd)-\dd\cdot \dd;\quad\quad
\chi_Q(\dd,\dd)=\dd\cdot \dd+\chi_{\WT{Q}}(\dd,\dd)/2.
\]
By Lemma \ref{PureAbove} the isomorphic objects of \eqref{Dbpu} are pure above, and thus the objects of \eqref{dre} are pure below.  On the other hand, there are isomorphisms
\begin{align}
\label{drb}
r'_*\JH^{\circ}_*\kappa^*\HA^{\zeta}_{\WT{Q},\WT{W},\theta}\cong& r'_*\kappa'^*\JH_*\HA^{\zeta}_{\WT{Q},\WT{W},\theta}\\ \nonumber
\cong &r'_*\kappa'^*\Sym_{\oplus}\!\left( \BPSh_{\WT{Q},\WT{W},\theta}^{\zeta}\otimes\HO(\B \CC^*,\QQ)_{\vir}\right)\\ \nonumber
\cong &\Sym_{\oplus}\!\left( \BPSh_{\Pi_Q,\theta}^{\zeta}\otimes \HO(\B \CC^*,\QQ)_{\vir}\otimes\LLL\right)
\end{align}
where we have used the PBW theorem \eqref{relPBW}, and the fact that $r'$ and $\kappa'$ are morphisms of monoids to commute them past $\Sym_{\oplus}$.  Combining \eqref{dre} and \eqref{drb} there is a split inclusion
\begin{equation}
\label{splinc}
\BPSh_{\Pi_Q,\dd}^{\zeta}\otimes\LLL^{1/2}\subset \JH_{\red,*}\iota_*\iota^!\ul{\QQ}_{\Mst_{\dd}^{\zeta\sst}(\ol{Q})}\otimes\LLL^{\chi_{\WT{Q}}(\dd,\dd)/2}.
\end{equation}
We deduce that $\BPSh_{\Pi_Q,\dd}^{\zeta}$ is pure below, since the right hand side of \eqref{splinc} is, and thus also pure above by ($\ast$).
\end{proof}
We deduce, by Saito's results (in particular, Theorem \ref{SDC}) that there is a decomposition
\[
\BPSh_{\Pi_Q,\theta}^{G,\zeta}\cong \bigoplus_{i\in S_n}\ICS_{\ol{Z_i}}(\underline{\mathcal{F}}_i)
\]
with $Z_i\subset \Msp^{G,\zeta\sst}_{\theta}(\ol{Q})$ locally closed, smooth, connected substacks, and $\ICS_{\ol{Z_i}}(\underline{\mathcal{F}}_i)$ intermediate extensions of pure weight zero variations of mixed Hodge structure on them.
\begin{corollary}
\label{3dpurityThm}
The BPS sheaf $\BPSh^{G,\zeta}_{\WT{Q},\WT{W},\theta}\in\MMHM^G(\Msp_{\theta}^{\zeta\sst}(\WT{Q}))$ is a pure monodromic mixed Hodge module.
\end{corollary}
\begin{proof}
Again, the fact that $\BPSh^{G,\zeta}_{\WT{Q},\WT{W}}$ is in the heart of the t structure on $\DbMMHM^G(\Msp^{\zeta\sst}(\WT{Q}))$ follows by \eqref{JHvanishing} and \ref{BPShdef}.  The result then follows from Theorem \ref{purityThm} and the isomorphism \eqref{rbps2}.
\end{proof}

The purity statement of Theorem \ref{thma} is a special case of the following corollary of Theorem \ref{purityThm}:
\begin{corollary}
\label{relPurity}
The underlying objects of the relative CoHAs in $\Msp^{G,\zeta\sst}_{\theta}(\WT{Q})$ and $\Msp^{G,\zeta\sst}_{\theta}(\ol{Q})$, i.e.
\[
\rCoha_{\Pi_Q,\theta}^{G,\zeta}\in \DlMHM(\Msp^{G,\zeta\sst}_{\theta}(\ol{Q}));\quad\quad
\rCoha_{\WT{Q},\WT{W},\theta}^{G,\zeta}\in \DlMMHM(\Msp^{G,\zeta\sst}_{\theta}(\WT{Q}))
\]
respectively, are pure.  In particular, applying Verdier duality to the first of these statements, and taking the appropriate Tate twist, the complex of mixed Hodge modules $\JH_{\red,!}\ul{\QQ}_{\Mst^{G,\zeta\sst}(\Pi_Q)}$ is pure.
\end{corollary}
\begin{proof}
These purity statements follow from Theorem \ref{purityThm} and Corollary \ref{3dpurityThm}, respectively, via \eqref{2dInt} and \eqref{relPBW}, respectively.  
\end{proof}

\begin{remark}
Since the morphism $\JH\colon \Mst^{G,\zeta\sst}(\WT{Q})\rightarrow \Msp^{G,\zeta\sst}(\WT{Q})$ is approximated by proper maps (in the sense of \cite{QEAs}) and thus sends pure mixed Hodge modules to pure mixed Hodge modules, it might feel natural, in light of Corollary \ref{relPurity}, to conjecture that $\HA_{\WT{Q},\WT{W}}^{G,\zeta}$ is a pure monodromic mixed Hodge module on $\Mst^{G,\zeta\sst}(\WT{Q})$.  However this statement turns out to be false.  For example in the case of $Q$ the Jordan quiver, $G=\{1\}$ $\zeta=(0,\ldots,0)$ and $\dd=4$, impurity follows from the main result of \cite{DS09}.  It seems that purity goes no ``higher'' than BPS sheaves.
\end{remark}
\begin{remark}
For more impurity above the level of BPS sheaves, consider the constant sheaf $\mathbb{Q}_{\Mst_{(1,1)}(\Pi_Q)}[-1]$ where $Q$ is the quiver with two vertices and one arrow between them.  While it is easy to check that this sheaf is actually perverse, it is also not hard to show that the weight filtration on its lift to a mixed Hodge module is nontrivial; see \cite[Ex.3.2]{Da21a} for details.
\end{remark}

\section{The less perverse filtration}
\label{lpfsec}
\subsection{The Hall algebra in $\DlMHM(\Msp^{G,\zeta\sst}(\overline{Q}))$}
We consider the following diagram, where the top three rows are defined from diagram \eqref{3dc} (substituting $\WT{Q}$ for $Q$ there) by pulling back along the open embeddings
\[
\kappa\colon \Mst_{\theta}^{G ,\zeta\ssst}(\WT{Q})\hookrightarrow \Mst_{\theta}^{G ,\zeta\sst}(\WT{Q});\quad\quad
\kappa'\colon \Msp_{\theta}^{G ,\zeta\ssst}(\WT{Q})\hookrightarrow \Msp_{\theta}^{G ,\zeta\sst}(\WT{Q}).
\]
\begin{equation}
\label{2dc}
\xymatrix{
&\Mst_{\theta}^{G ,\zeta\ssst}(\WT{Q})_{(2)}\ar[dl]_{\pi_1\times\pi_3\;\;}\ar[dr]^{\pi_2}\\
\Mst_{\theta}^{G ,\zeta\ssst}(\WT{Q})\times_{\B G} \Mst_{\theta}^{G ,\zeta\ssst}(\WT{Q})\ar[d]^{\JH^{\circ}\times_{\B G} \JH^{\circ}}&&\Mst_{\theta}^{G ,\zeta\ssst}(\WT{Q})\ar[d]^{\JH^{\circ}}\\
\Msp_{\theta}^{G ,\zeta\ssst}(\WT{Q})\times_{\B G} \Msp_{\theta}^{G ,\zeta\ssst}(\WT{Q})\ar[d]^{r'\times_{\B G}\;r'}\ar[rr]^-{\oplus^{G }}&&\Msp_{\theta}^{G ,\zeta\ssst}(\WT{Q})\ar[d]^{r'}\\
\Msp_{\theta}^{G ,\zeta\sst}(\ol{Q})\times_{\B G} \Msp_{\theta}^{G ,\zeta\sst}(\ol{Q})\ar[rr]^-{\oplus_{\red}^G}&&\Msp_{\theta}^{G ,\zeta\ssst}(\ol{Q}).
}
\end{equation}
By dimensional reduction there is a natural isomorphism
\begin{equation}
\label{dwg}
\psi\colon \rCoha_{\Pi_Q,\theta}^{\SP,G,\zeta\sst}\cong r'_*\kappa^*\rCoha_{\WT{Q},\WT{W},\theta}^{\WT{\SP},G,\zeta\sst}.
\end{equation}
Since the bottom square of \eqref{2dc} commutes, the complex of monodromic mixed Hodge modules on the left hand side of \eqref{dwg} inherits an algebra structure in $\DlMHM(\Msp_{\theta}^{\Sp,G,\zeta\sst}(\ol{Q}))$, i.e. we define the morphism
\begin{equation}
\label{rcmap}
m\colon \rCoha_{\Pi_Q,\theta}^{\SP,G,\zeta\sst}\boxtimes_{\oplus^{G}_{\red}}\rCoha_{\Pi_Q,\theta}^{\SP,G,\zeta\sst}\rightarrow \rCoha_{\Pi_Q,\theta}^{\SP,G,\zeta\sst}
\end{equation}
so that the diagram
\[
\xymatrix{
\rCoha_{\Pi_Q,\theta}^{\SP,G,\zeta\sst}\boxtimes_{\oplus^G_{\red}}\rCoha_{\Pi_Q,\theta}^{\SP,G,\zeta\sst}\ar[rr]^-{\psi\boxtimes_{\oplus^G_{\red}}\psi}\ar[dd]^m&&r'_*\kappa^*\rCoha_{\WT{Q},\WT{W},\theta}^{\WT{\SP},G,\zeta\sst}\boxtimes_{\oplus^G_{\red}}r'_*\kappa^*\rCoha_{\WT{Q},\WT{W},\theta}^{\WT{\SP},G,\zeta\sst}\ar[d]^{\cong}\\ && r'_*\kappa^*\left(\rCoha_{\WT{Q},\WT{W},\theta}^{\WT{\SP},G,\zeta\sst}\boxtimes_{\oplus^G}\rCoha_{\WT{Q},\WT{W},\theta}^{\WT{\SP},G,\zeta\sst}\right)\ar[d]^{r'_*\kappa^*m'}\\
\rCoha_{\Pi_Q,\theta}^{\SP,G,\zeta\sst}&&r'_*\kappa^*\rCoha_{\WT{Q},\WT{W},\theta}^{\WT{\SP},G,\zeta\sst}\ar[ll]^{\psi^{-1}}
}
\]
commutes, where $m'$ is the relative CoHA product for $\rCoha_{\WT{Q},\WT{W},\theta}^{\WT{\SP},G,\zeta,\chi}$.  Applying the functor $\HO$ to this morphism we recover the algebra structure on $\HCoha_{\Pi_Q,\theta}^{\SP,G,\zeta}$.

\subsection{The relative Lie algebra in $\MHM^G(\Msp_{\theta}^{\zeta\sst}(\ol{Q}))$}
\label{RLAsec}
By Theorem \ref{PBWtheorem} there is a split inclusion
\begin{equation}
\label{uwg}
\BPSh_{\WT{Q},\WT{W},\theta}^{G,\zeta}\otimes\LLL^{1/2}\rightarrow \rCoha_{\WT{Q},\WT{W},\theta}^{G,\zeta,\chi}.
\end{equation}
The commutator Lie bracket coming from the Hall algebra product on the right hand side of \eqref{uwg} provides a morphism
\[
[\cdot,\cdot]\colon \left(\BPSh_{\WT{Q},\WT{W},\theta}^{G,\zeta}\otimes\LLL^{1/2}\right)\boxtimes_{\oplus^G}\left(\BPSh_{\WT{Q},\WT{W},\theta}^{G,\zeta}\otimes\LLL^{1/2}\right)\rightarrow \mathcal{T}=\bm{\tau}^{G,\leq 2}\rCoha_{\WT{Q},\WT{W},\theta}^{G,\zeta,\chi}
\]
and by Theorem \ref{PBWtheorem} again, the target $\mathcal{T}$ fits into a split triangle
\[
\BPSh_{\WT{Q},\WT{W},\theta}^{G,\zeta}\otimes\LLL^{1/2}\rightarrow \mathcal{T}\rightarrow \Ho^{G,2}(\mathcal{T}).
\]
and the composition
\[
\left(\BPSh_{\WT{Q},\WT{W},\theta}^{G,\zeta}\otimes\LLL^{1/2}\right)\boxtimes_{\oplus^G}\left(\BPSh_{\WT{Q},\WT{W},\theta}^{G,\zeta}\otimes\LLL^{1/2}\right)\rightarrow \Ho^{G,2}(\mathcal{T})
\]
is the zero morphism, since after applying the functor $\Ho(\cdot)$, the Hall algebra product is supercommutative (here it is essential that we twist by the sign $\chi$).  The commutator Lie bracket thus induces a morphism
\begin{equation}
\label{relLA}
[\cdot,\cdot]\colon \left(\BPSh_{\WT{Q},\WT{W},\theta}^{G,\zeta}\otimes\LLL^{1/2}\right)\boxtimes_{\oplus^G}\!\left(\BPSh_{\WT{Q},\WT{W},\theta}^{G,\zeta}\otimes\LLL^{1/2}\right)\rightarrow  \BPSh_{\WT{Q},\WT{W},\theta}^{G,\zeta}\otimes\LLL^{1/2}
\end{equation}
and applying $r'_*\kappa'^*$ we obtain the Lie bracket
\begin{equation}
\label{relLA2}
[\cdot,\cdot]\colon \BPSh_{\Pi_Q,\theta}^{G,\zeta}\boxtimes_{\oplus_{\red}^G}\BPSh_{\Pi_Q,\theta}^{G,\zeta}\rightarrow \BPSh_{\Pi_Q,\theta}^{G,\zeta}.
\end{equation}
This is a Lie algebra object inside the category\footnote{We use here that the 2d BPS sheaf is supported on $\Msp^{\zeta\sst}_{\theta}(\Pi_Q)$.} $\MHM^G(\Msp^{\zeta\sst}_{\theta}(\Pi_Q))$, from which we obtain $\fg_{\Pi_Q,\theta}^{\zeta}$ by applying $\HO$.  Likewise, applying $\varpi'_{\red,*}\varpi'^!_{\red}$ to \eqref{relLA2} we obtain a Lie algebra structure on the object $\BPSh_{\Pi_Q,\theta}^{\SP,G,\zeta}\in \DbMHM(\Msp^{\zeta\sst}_{\theta}(\Pi_Q))$, which becomes the Lie algebra $\fg^{\SP,G,\zeta}_{\Pi_Q,\theta}$ after applying $\HO$.
\subsection{Definition of the filtration}
By base change, the algebra morphism \eqref{rcmap} is given by applying $\varpi'_{\red,*}\varpi'^!_{\red}$ to 
\[
\oplus_{\red,*}^G\!\left(\rCoha_{\Pi_Q,\theta}^{G,\zeta}\boxtimes_{\B G}\rCoha_{\Pi_Q,\theta}^{G,\zeta}\right)\rightarrow \rCoha_{\Pi_Q,\theta}^{G,\zeta},
\]
a morphism in $\DlMHM(\Msp_{\theta}^{G,\zeta\sst}(\Pi_Q))$.  Furthermore by \eqref{2dInt} there is an isomorphism of underlying complexes
\[
\rCoha_{\Pi_Q,\theta}^{G,\zeta}\cong \Sym_{\oplus_{\red}^G}\!\left(\BPSh_{\Pi_Q,\theta}^{G,\zeta}\otimes\HO(\B \CC^*,\QQ) \right).
\]
Since $\BPSh_{\Pi_Q,\theta}^{G,\zeta}\in\MHM^G(\Msp_{\theta}^{\zeta\sst}(\Pi_Q))$ there is an isomorphism
\begin{equation}
\label{nisoBPS}
\Ho\!\left( \BPSh_{\Pi_Q,\theta}^{G,\zeta}\otimes\HO(\B \CC^*,\QQ) \right)\cong \BPSh_{\Pi_Q,\theta}^{G,\zeta}\otimes\HO(\B \CC^*,\QQ).
\end{equation}
In other words the right hand side is isomorphic to its total cohomology.  Since $\boxtimes_{\oplus_{\red}^G}$ is exact, the same property holds for the symmetric product, meaning that there is also an isomorphism
\begin{equation}
\label{decomp2}
\Sym_{\oplus_{\red}^G}\!\left(\BPSh_{\Pi_Q,\theta}^{G,\zeta}\otimes\HO(\B \CC^*,\QQ) \right)\cong \Ho\!\left(\Sym_{\oplus_{\red}^G}\!\left(\BPSh_{\Pi_Q,\theta}^{G,\zeta}\otimes\HO(\B \CC^*,\QQ) \right)\right).
\end{equation}

%\begin{remark}
%The existence of some isomorphism isomorphism as in \eqref{decomp2} follows from the purity of the left hand side.  
%\end{remark}
It follows that for every $p\in\mathbb{Z}$ the morphism $\bm{\tau}^{G,\leq p}\!\rCoha_{\Pi_Q,\theta}^{G,\zeta}\rightarrow \rCoha_{\Pi_Q,\theta}^{G,\zeta}$ has a left inverse $\alpha_p$, and so $\HO\varpi'_{\red,*}\varpi'^!_{\red}\alpha_p$ provides a left inverse to the natural morphism
\[
\HO\!\left(\Msp_{\theta}^{G,\zeta\sst}(\ol{Q}),\bm{\tau}^{G,\leq p}\varpi'_{\red,*}\varpi'^!_{\red}\rCoha_{\Pi_Q,\theta}^{G,\zeta}\right)\rightarrow \HCoha_{\Pi_Q,\theta}^{\SP,G,\zeta}.
\]
Thus the objects
\[
\lP_{p}\!\HCoha_{\Pi_Q,\theta}^{\SP,G,\zeta}\coloneqq \HO\!\left(\Msp_{\theta}^{G,\zeta\sst}(\ol{Q}),\varpi'_{\red,*}\varpi'^!_{\red}\bm{\tau}^{G,\leq p}\!\rCoha_{\Pi_Q,\theta}^{G,\zeta}\right)
\]
provide an ascending filtration of $\HCoha_{\Pi_Q,\theta}^{\SP,G,\zeta}$, the \textbf{less perverse filtration}.
\begin{proposition}
\label{respProp}
The algebra structure on $\HCoha_{\Pi_Q,\theta}^{\SP,G,\zeta}$ respects the less perverse filtration; if $\mult$ is the multiplication on $\HCoha_{\Pi_Q,\theta}^{\SP,G,\zeta}$ defined in \eqref{RelHH} then 
\[
\mult\left(\lP_{p}\!\HCoha_{\Pi_Q,\theta}^{\SP,G,\zeta}\otimes_{\HO_G}\lP_{p'}\!\HCoha_{\Pi_Q,\theta}^{\SP,G,\zeta}\right)\subset \lP_{p+p'}\!\HCoha_{\Pi_Q,\theta}^{\SP,G,\zeta}.
\]
\end{proposition}
\begin{proof}
The morphism $\mult$ is obtained by applying $\HO \omega_*\omega^!$ to the morphism $\star\colon\rCoha_{\Pi_Q,\theta}^{G,\zeta}\boxtimes_{\oplus_{\red}^G}\rCoha_{\Pi_Q,\theta}^{G,\zeta}\rightarrow \rCoha_{\Pi_Q,\theta}^{G,\zeta}$ defined in \S \ref{2drelCoHA} for which $\SP$ is the entire category of $\Pi_Q$-modules.  By exactness of $\boxtimes_{\oplus_{\red}^G}$, the natural morphism 
\[
\bm{\tau}^{G,\leq p}\rCoha_{\Pi_Q,\theta}^{G,\zeta}\boxtimes_{\oplus_{\red}^G} \bm{\tau}^{G,\leq p'}\rCoha_{\Pi_Q,\theta}^{G,\zeta}\rightarrow \rCoha_{\Pi_Q,\theta}^{G,\zeta}\boxtimes_{\oplus_{\red}^G}\rCoha_{\Pi_Q,\theta}^{G,\zeta}
\]
factors through 
\[
\bm{\tau}^{G,\leq p+p'}\left(\rCoha_{\Pi_Q,\theta}^{G,\zeta}\boxtimes_{\oplus_{\red}^G}\rCoha_{\Pi_Q,\theta}^{G,\zeta}\right)\rightarrow \rCoha_{\Pi_Q,\theta}^{G,\zeta}\boxtimes_{\oplus_{\red}^G}\rCoha_{\Pi_Q,\theta}^{G,\zeta}.
\]
Applying $\HO \omega_*\omega^!\bm{\tau}^{G,\leq p+p'}$ to the morphism $\star$, the result follows.
\end{proof}

%\subsection{BPS sheaves for $\Coha^{\Sp,G,\zeta}_{\WT{Q},\WT{W}}$}
%In general it seems to be quite hard to calculate BPS sheaves, but in the case of the preprojective CoHA we can at least say something rather restrictive about the support of these sheaves.  

\subsubsection{Splitting the filtration}
Since $\BPSh_{\Pi_Q,\theta}^{G,\zeta}$ is a mixed Hodge module, and $\HO(\B \CC^*,\QQ)$ is a cohomologically graded vector space, the isomorphism \eqref{nisoBPS} is \textit{natural}, and thus, the isomorphism \eqref{decomp2} is also natural.  Explicitly, writing $\HO(\B \CC^*,\QQ)=\bigoplus_{n\geq 0}\LLL^n$, we may write
\[
\Sym_{\oplus_{\red}^G}\!\left(\BPSh_{\Pi_Q,\theta}^{G,\zeta}\otimes\HO(\B \CC^*,\QQ) \right)=\bigoplus_{\substack{a_1,\ldots,a_r\in\NN\\r\in\NN}}\bigotimes_{i=1}^r\Sym^{a_i}_{\oplus_{\red}^G}\!\left(\BPSh_{\Pi_Q,\theta}^{G,\zeta}\otimes\LLL^i\right).
\]
Then if $n$ is even, $\Ho^n$ of the left hand side is given by taking the sum of all summands on the right hand side satisfying $\sum_{i=1}^ria_i=n/2$, while if $n$ is odd, $\Ho^n$ of the left hand side is zero.  Passing to derived global sections, we deduce that the perverse filtration we have defined here naturally splits.  We use only the filtration in what follows, since this splitting is not in general respected by the cohomological Hall algebra structure.  For an example of this see \cite{Da22}.  In particular, the left hand side of the relation in \cite[Prop.5.2]{Da22} is of homogeneous degree 2 while the right hand side is of homogeneous degree 0.

\subsubsection{A warning}
A variant of \cite[Warning 5.5]{QEAs} is in force here; if $\SP$ is not the entire category $\CC \ol{Q}\lmod$, the perverse filtration that we have defined here may be quite different from the perverse filtration given by applying perverse truncation functors to $\JH_{\red,*}\HA_{\Pi_Q,\theta}^{\SP,G,\zeta}$. (assuming this latter filtration makes sense)  For instance, let $Q$ be the Jordan quiver, with one loop, and consider the Serre subcategory $\SSN$.  Then one may easily verify that 
\begin{equation}
\label{warningeq}
\JH_{\red,*}\HA_{\Pi_Q,1}^{\SSN}\cong i_*\ul{\QQ}_{\AAA{1}}\otimes\HO(\B\CC^*,\QQ)
\end{equation}
where $i\colon \AAA{1}\hookrightarrow \AAA{2}$ is the inclusion of a coordinate hyperplane.  In particular, the zeroth perverse cohomology of \eqref{warningeq} is zero, while if instead we apply $\varpi'_{\red,*}\varpi'^!_{\red}$ to the zeroth cohomology of 
\[
\JH_{\red,*}\HA_{\Pi_Q,1}\cong \ul{\QQ}_{\AAA{2}}\otimes\HO(\B\CC^*,\QQ)\otimes\LLL^{-1}
\]
we get the (shifted) mixed Hodge module $\ul{\QQ}_{\AAA{1}}$, and we find $\lP_{0}\!\HCoha_{\Pi_Q,1}^{\SSN}\cong \HO(\AAA{1},\QQ)\neq 0$.  This distinction between the two choices of filtration on $\HCoha_{\Pi_Q}^{\SSN}$ is crucial in \S \ref{BBA}.

\subsection{Comparison with the perverse filtration of \cite{QEAs}}
\label{morePerverse}
We call the filtration introduced in Theorem \ref{thma} the \textit{less} perverse filtration, in order to distinguish it from a different perverse filtration, that was introduced in \cite{QEAs} on the way towards the definition of BPS sheaves, which recalled in \eqref{OPdef}.  This is a perverse filtration on the critical CoHA $\HCoha^{\WT{\SP}}_{\WT{Q},\WT{W}}$, defined via perverse truncation functors on the category of monodromic mixed Hodge modules on the coarse moduli space of $\CC\tilde{Q}$-modules.  This perverse filtration is the crucial part of the definition of the BPS Lie algebra for a \textit{general} quiver with potential, and in particular the definition of the BPS Lie algebra $\fg_{\tilde{Q},\tilde{W}}^{\tilde{\Sp}}$ for the tripled quiver with its canonical cubic potential in Theorem \ref{thma}.%  We recall some of the main facts regarding critical CoHAs for \textit{general} quivers with potential, before explaining the relationship between the two perverse filtrations.
Switching to the ordinary English meaning of the word, the filtration $\lP_{\bullet}\!\HCoha_{\Pi_Q}^{\SP}$ of Theorem \ref{thma} seems less perverse than $\mathfrak{P}_{ \bullet}\!\HCoha_{\tilde{Q},\tilde{W}}^{\tilde{\SP}}$, since it comes directly from the geometry of the map $\Mst(\Pi_Q)\rightarrow \Msp(\Pi_Q)$, rather than the more circuitous route of dimensional reduction, vanishing cycles, and the semisimplification morphism $\Mst(\WT{Q})\rightarrow \Msp(\WT{Q})$ for the auxiliary quiver $\WT{Q}$.  The two filtrations are rather different\footnote{As a consequence of this difference, there is value in considering them both simultaneously; see \S \ref{DPFsec} for example.}; for instance, the BPS Lie algebra embeds in $\lP_{0}\!\HCoha_{\Pi_Q}^{\SP,\psi}$, while $\mathfrak{P}_{0}\!\HCoha_{\tilde{Q},\tilde{W}}^{\tilde{\SP},\psi}=0$.  In general, perverse degrees with respect to the new filtration are lower than for the old one.  It is for these two reasons that we call the new filtration the \textit{less} perverse filtration.

%While the definition of the less perverse filtration does not reference the 3D theory, and nor does Theorem \ref{thma} or the purity part of Theorem \ref{mainThmB}, the proofs of all of the results in this paper use in an essential way the cohomological Donaldson--Thomas theory coming from the 3D theory, and the theory of BPS sheaves established there.

\subsection{Deformed dimensional reduction}
\label{DDRsec}
We can generalise the results of this paper, incorporating deformed potentials as introduced in joint work with Tudor P\u adurariu \cite{DaPa20}.  We indicate how this goes in this section.  We will not use this generalisation of the less perverse filtration, except in the statement of Proposition \ref{bels} and the example of \S \ref{defEx}.

Let $W_0\in\CC \ol{Q}_{\cyc}$ be a $G$-invariant  linear combination of cyclic words in $\ol{Q}$.  We make the assumption that there is a grading of the arrows of $\WT{Q}$ so that $\WT{W}+W_0$ is quasihomogeneous of positive degree.  Then in \cite{DaPa20} it was shown that there is a natural isomorphism.
\[
\HCoha^{G,\chi}_{\WT{Q},\WT{W}+W_0}=\HO(\HA^G_{\WT{Q},\WT{W}+W_0})\cong \HO(\phim{\TTr(W_0)}\JH^G_{\red,*}\HA^G_{\Pi_Q})
\]
which one may show is an isomorphism of algebras, i.e. $\HCoha^{G,\chi}_{\WT{Q},\WT{W}+W_0}$ is obtained by applying the derived global sections functor to the algebra object $\phim{\TTr(W_0)}\JH^G_{\red,*}\HA^G_{\Pi_Q}$ obtained by applying the symmetric monoidal functor $\phim{\TTr(W_0)}$ to the algebra object $\JH^G_{\red,*}\HA^G_{\Pi_Q}$, with $\psi$-twist given by $\chi$ as in \S \ref{3desc}.  Since $\phim{\TTr(W_0)}$ is exact, and $\JH^G_{\red,*}\HA^G_{\Pi_Q}$ is pure by Theorem \ref{thma}, there is an isomorphism
\[
\Ho(\phim{\TTr(W_0)}\JH^G_{\red,*}\HA^G_{\Pi_Q})\cong \phim{\TTr(W_0)}\JH^G_{\red,*}\HA^G_{\Pi_Q}
\]
and so $\HCoha^{G,\chi}_{\WT{Q},\WT{W}+W_0}$ carries a less perverse filtration, defined in the same way as the less perverse filtration for $\HCoha^{G,\chi}_{\WT{Q},\WT{W}}$, i.e.
\[
\lP_{p}\!\HCoha^{G,\chi}_{\WT{Q},\WT{W}+W_0}\coloneqq \HO\!\left(\Msp_{\theta}^{G,\zeta\sst}(\ol{Q}),\bm{\tau}^{G,\leq p}\!\:\phim{\TTr(W_0)}\JH^G_{\red,*}\HA^G_{\Pi_Q}\right)
\]
As in \S \ref{RLAsec}, we obtain a Lie algebra structure on $\BPSh^G_{\Pi_Q,W_0}\coloneqq \phim{\TTr(W_0)}\BPSh^G_{\Pi_Q}$ and we define the BPS Lie algebra $\fg^{G}_{\Pi_Q,W_0}\coloneqq \HO(\phim{\TTr(W_0)}\BPSh^G_{\Pi_Q})$.

For a final layer of generality, for $\SP$ a Serre subcategory of $\CC\ol{Q}\lmod$ and $\WT{W}+W_0$ a quasihomogeneous potential as above, one may consider the object $\HA^{\WT{\SP},G}_{\WT{Q},\WT{W}+W_0}\coloneqq\varpi_*\varpi^!\HA_{\WT{Q},\WT{W}+W_0}^G$, and the associated Hall algebra $\HCoha^{\WT{\SP},G,\chi}_{\WT{Q},\WT{W}+W_0}$ which carries a (less) perverse filtration defined by
\[
\lP_{i}\! \HCoha^{\WT{\SP},G,\chi}_{\WT{Q},\WT{W}+W_0}\coloneqq \HO \varpi'_{\red,*}\varpi_{\red}'^!\bm{\tau}^{G,\leq i}\phim{\TTr(W_0)}\JH_{\red,*}^G\HA^G_{\Pi_Q}
\]
with 2d BPS sheaf $\BPSh^{\SP,G}_{\Pi_Q,W_0}=\varpi'_{\red,*}\varpi'^!_{\red}\phim{\TTr(W_0)}\BPSh^{\SP,G}_{\Pi_Q}$ and associated BPS Lie algebra $\fg_{\Pi_Q,W_0}^{\SP,G}=\HO\!\left(\Msp^G(\ol{Q}),\BPSh^{\SP,G}_{\Pi_Q,W_0}\right)$.  Even at this maximal level of generality, we will find that the spherical Lie subalgebra is a Kac--Moody Lie algebra: see Proposition \ref{bels}.

\section{The zeroth piece of the filtration}
\label{zpf}
\subsection{The subalgebra $\lP_{0}\!\HCoha^{\Sp,G,\zeta}_{\Pi_Q,\theta}$}
By \eqref{2dInt}, the less perverse filtration on $\HCoha^{\Sp,G,\zeta}_{\Pi_Q,\theta}$ begins in degree zero, and since the multiplication respects the less perverse filtration, the cohomologically graded vector space $\lP_{0}\!\HCoha^{\Sp,G,\zeta}_{\Pi_Q,\theta}$ is closed under the CoHA multiplication.  It turns out that this subalgebra has a very natural description in terms of the BPS Lie algebra, completing the proof of Theorem \ref{thma}:
\begin{theorem}
\label{UEAthm}
There is an isomorphism of algebras $\lP_{0}\!\HCoha^{\Sp,G,\zeta}_{\Pi_Q,\theta}\cong \UEA_{\HO_G}(\mathfrak{g}_{\Pi_Q,\theta}^{G,\SP,\zeta})$.
\end{theorem}
\begin{proof}
Applying $\bm{\tau}^{G,\leq 0}$ to the isomorphism \eqref{2dInt} in the special case $\SP=\CC\ol{Q}\lmod$ yields the isomorphism of mixed Hodge modules
\begin{equation}
\label{halfzero}
\Sym_{\oplus_{\red}^G}\!\left( \BPSh_{\Pi_Q,\theta}^{G,\zeta}\right)\cong \bm{\tau}^{G,\leq 0}\!\left(\rCoha_{\Pi_Q,\theta}^{G,\zeta}\right)
\end{equation}
defined via the relative CoHA multiplication (although this is not an isomorphism of algebra objects).  Now applying $\HO\varpi'_{\red,*}\varpi'^!_{\red}$ to \eqref{halfzero} we obtain the isomorphism
\[
\Sym_{\HO_G}(\fg_{\Pi_Q,\theta}^{\SP,G,\zeta})\xrightarrow{\cong}\lP_{0}\!\HCoha^{\Sp,G,\zeta}_{\Pi_Q,\theta}.
\]
By Proposition \ref{UEApart}, the image of the induced embedding $\Sym_{\HO_G}(\fg_{\Pi_Q,\theta}^{\SP,G,\zeta})\hookrightarrow \HCoha^{\Sp,G,\zeta}_{\Pi_Q,\theta}$ is precisely the subalgebra $\UEA_{\HO_G}(\mathfrak{g}_{\Pi_Q,\theta}^{G,\SP,\zeta})$.
\end{proof}

\subsubsection{The perverse filtration $\mathfrak{P}_{\bullet}\!\lP_{0}\!\HCoha^{\Sp,G,\zeta}_{\Pi_Q,\theta}$}
\label{DPFsec}

Since via \eqref{drnoalg} there is an inclusion 
\[
I\colon \lP_{ 0}\!\HCoha^{\Sp,G,\zeta}_{\Pi_Q,\theta}\hookrightarrow \HCoha_{\WT{Q},\WT{W},\theta}^{\WT{\SP},G,\zeta}
\]
and $\HCoha_{\WT{Q},\WT{W},\theta}^{\WT{\SP},G,\zeta}$ carries the ``more'' perverse filtration $\mathfrak{P}_{\bullet}\!\HCoha_{\WT{Q},\WT{W},\theta}^{\WT{\SP},G,\zeta}$ defined in \cite{QEAs}, we obtain a perverse filtration on 
$\lP_{0}\!\HCoha^{\Sp,G,\zeta}_{\Pi_Q,\theta}$ itself, for which the $i$th piece is
\begin{equation}
\label{PLP}
\lP_{0}\!\HCoha^{\Sp,G,\zeta}_{\Pi_Q,\theta}\cap\: I^{-1}\!\left(\mathfrak{P}_{i}\!\HCoha^{\WT{\Sp},G,\zeta}_{\WT{Q},\WT{W},\theta}\right).
\end{equation}
Writing 
\begin{align*}
\lP_{0}\!\HCoha^{\Sp,G,\zeta}_{\Pi_Q,\theta}\cong \HO\varpi'_*\varpi'^!\Sym_{\oplus^G}\!\left(\BPSh_{\WT{Q},\WT{W},\theta}^{G,\zeta}\otimes\LLL^{1/2}\right),
\end{align*}
we have that
\begin{align*}
\lP_{0}\!\HCoha^{\Sp,G,\zeta}_{\Pi_Q,\theta}\cap\: I^{-1}\!\left(\mathfrak{P}_{n}\!\HCoha^{\WT{\Sp},G,\zeta}_{\WT{Q},\WT{W},\theta}\right)\cong&\HO\varpi'_*\varpi'^!\bm{\tau}^{G,\leq n}\Sym_{\oplus^G}\!\left(\BPSh_{\WT{Q},\WT{W},\theta}^{G,\zeta}\otimes\LLL^{1/2}\right)\\
\cong& \HO\varpi'_*\varpi'^!\bigoplus_{i=0}^n\left(\Sym^i_{\oplus^G}\!\left(\BPSh_{\WT{Q},\WT{W},\theta}^{G,\zeta}\otimes\LLL^{1/2}\right)\right)\\
\cong& \HO\bigoplus_{i=0}^n\left(\Sym^i_{\oplus^G}\!\left(\BPSh_{\WT{Q},\WT{W},\theta}^{\WT{\SP},G,\zeta}\otimes\LLL^{1/2}\right)\right)\\
\cong& \bigoplus_{i=0}^n\left(\Sym^i_{\HO_G}\!\left(\fg_{\Pi_Q,\theta}^{\SP,G,\zeta}\right)\right).
\end{align*}
We deduce the following
\begin{proposition}
Under the isomorphism of Theorem \eqref{UEAthm}, the perverse filtration \eqref{PLP} is sent to the order filtration on the universal enveloping algebra $\UEA_{\HO_G}(\mathfrak{g}_{\Pi_Q,\theta}^{G,\SP,\zeta})$.
\end{proposition}

\subsection{Nakajima quiver varieties}
\label{NQVs}

As preparation for the proof of Theorem \ref{KMLA} below, we recall some fundamental results regarding the action of $\HCoha_{\Pi_Q}$ on the cohomology of Nakajima quiver varieties, recasting these results in terms of vanishing cycle cohomology along the way.  For the rest of \S \ref{zpf} we focus on the case $G=1$.  
\subsubsection{Nakajima quiver varieties as critical loci}
Given a quiver $Q$ and a dimension vector $\ff\in\dvs$, we define the quiver $Q_{\ff}$ by adding one vertex $\infty$ to the vertex set $Q_0$, and for each vertex $i\in Q_0$ we add $\ff_i$ arrows $a_{i,1},\ldots,a_{i,\ff_i}$ with source $\infty$ and target $i$.

Given a dimension vector $\dd\in\dvs$ we denote by $\dd^+$ the dimension vector for $Q_{\ff}$ defined by $\dd^+\lvert_{Q_0}=\dd$ and $\dd^+_{\infty}=1$.  From the quiver $Q_{\ff}$ we form the quiver $\WT{Q_{\ff}}$ via the tripling construction of \S \ref{3desc}.  For each $i\in Q_0$ there are $\ff_i$ arrows $a^*_{i,1},\ldots,a^*_{i,\ff_i}$ in $(\WT{Q_{\ff}})_1$ with source $i$ and target $\infty$.  We denote by $\WT{W_{\ff}}$ the canonical cubic potential for $\WT{Q_{\ff}}$.  We form $Q^+$ by removing the loop $\omega_{\infty}$ from $\WT{Q_{\ff}}$, and form $W^+$ from $\WT{W_{\ff}}$ by removing all paths containing $\omega_{\infty}$.  So we have
\[
W^+=\left(\sum_{a\in Q_1}[a,a^*]+\sum_{i\in Q_0}\sum_{m=1}^{\ff_i}a_{i,m}a^*_{i,m}\right)\left(\sum_{i\in Q_0}\omega_i\right).
\]

We define the stability condition $\zeta\in\QQ^{Q^+_0}$ by setting $\zeta_i=0$ for $i\in Q_0$ and $\zeta_{\infty}=1$.  Then a $\dd^+$-dimensional $\CC Q^+$-module $\rho$ is $\zeta^+$-stable if and only if it is $\zeta^+$-semistable.  This occurs if and only if the vector space $e_{\infty}\cdotsh \rho\cong\CC$ generates $\rho$ under the action of $\CC Q^+$.

We define the fine moduli space $\Msp_{\ff,\dd}(Q)=\AS^{\zeta^+\sst}_{\dd^+}(Q^+)/\Gl_{\dd}$, which carries the function $\TTr(W^+)_{\dd}$.  Following Nakajima \cite{Nak98}, we define $\Nak(\ff,\dd)\subset \AS_{\dd^+}^{\zeta^+\sst}(\ol{Q_{\ff}})/\Gl_{\dd}$ to be the $\Gl_{\dd}$ quotient of the intersection of the stable locus with the zero set of the moment map
\begin{align*}
\AS_{\dd}(Q)\times\AS_{\dd}(Q^{\opp})\times \left(\prod_{i\in Q_0} (\CC^{\dd_i})^{\ff_i}\right)\times \left(\prod_{i\in Q_0} ((\CC^{\dd_i})^{\ff_i})^*\right)&\rightarrow \g_{\dd}\\
(A,A^*,I,J)&\mapsto [A,A^*]+IJ.
\end{align*}
We define the embedding $\iota\colon\Nak(\ff,\dd)\hookrightarrow \Msp_{\ff,\dd}(Q)$ by extending a $\CC\ol{Q_{\ff}}$-module to a $\CC Q^+$ module, setting the action of each of the $\omega_i$ for $i\in Q_0$ to be zero.  If $\Nak(\ff,\dd)\neq \emptyset$ then it is smooth and
\begin{equation}
\label{dimM}
\dim(\Nak(\ff,\dd))=2\ff\cdot \dd-2\chi_Q(\dd,\dd).
\end{equation}

\begin{proposition}
\label{NQVDR}
There is an equality of subschemes $\crit\!\left(\TTr(W^+)_{\dd}\right)=\Nak(\ff,\dd)$.  Moreover, there is an isomorphism of mixed Hodge modules\footnote{These are indeed mixed Hodge modules, since by \eqref{dimM} there are an even number of half Tate twists in the definition \eqref{nICdef} of the right hand side of \eqref{rDIM}.}
\begin{equation}
\label{rDIM}
\phim{\TTr(W^+)_{\dd}}\nIC_{\Msp_{\ff,\dd}(Q)}\cong \iota_*\nIC_{\Nak(\ff,\dd)},
\end{equation}
so that, in particular, there is an isomorphism of mixed Hodge structures
\begin{equation}
\label{HDIM}
\HO\!\left(\Msp_{\ff,\dd}(Q),\phim{\TTr(W^+)_{\dd}}\nIC_{\Msp_{\ff,\dd}(Q)}\right)\cong \HO\!\left( \Nak(\ff,\dd),\QQ\right)\otimes\LLL^{\chi_Q(\dd,\dd)-\ff\cdot \dd}.
\end{equation}
\end{proposition}

\begin{proof}
The (Verdier dual of the) isomorphism \eqref{HDIM} in derived global sections is constructed in \cite[Thm.6.6]{preproj} via dimensional reduction, so the proof consists of lifting this isomorphism to the level of an isomorphism of mixed Hodge modules.

Let $c$ be an arrow from $\infty$ to $i$ in $Q^+$.  Then
\[
\partial W^+/\partial c^*=\omega_ic;\quad\quad
\partial W^+/\partial c =c^*\omega_i.
\]
For $a$ an arrow in $Q_1$ we have
\[
\partial W/\partial a^*= \omega_{t(a)}a-a\omega_{s(a)};\quad\quad
\partial W/\partial a=a^*\omega_{t(a)}-\omega_{s(a)}a^*.
\]
Putting these equalities together, we have an isomorphism of algebras
\begin{equation}
\label{Jacref}
\psi\colon\Jac(Q^+,W^+)\cong\Pi_{Q_{\ff}}[\omega]/ \langle \omega e_{\infty}\rangle
\end{equation}
where $\psi^{-1}(\omega)=\sum_{i\in Q_0}\omega_i$.  We consider the fine moduli space
\[
\mathcal{N}=\AS_{\dd^+}^{\zeta^+\sst}(\WT{Q_{\ff}})/\Gl_{\dd}.
\]
Then the critical locus of $\TTr(\WT{W_{\ff}})$ is identified with the total space of the bundle on $\Nak(\ff,\dd)$ for which the fibre over $\rho$ is the space of endomorphisms of $\rho$.  By stability, this is a rank one vector bundle.  By \eqref{Jacref}, $\crit(\TTr(W^+)_{\dd})$ is the zero section, i.e. it is $\Nak(\ff,\dd)$.

Since the (scheme-theoretic) critical locus of $\TTr(W^+)_{\dd}$ is smooth, by the holomorphic Bott--Morse lemma this function can be written analytically locally (on $\Nak(\ff,\dd)$) as 
\[
\TTr(W^+)_{\dd}=x_1^2+\ldots +x_e^2
\]
where $e$ is the codimension of $\Nak(\ff,\dd)$ inside $\Msp_{\ff,\dd}(Q)$, and so $\phim{\TTr(W^+)_{\dd}}\nIC_{\Msp_{\ff,\dd}(Q)}$ is analytically locally isomorphic to $\nIC_{\Nak(\ff,\dd)}$.  In particular, as a perverse sheaf it is locally isomorphic to $\QQ_{\Nak_{\ff,\dd}(Q)}[2\ff\cdot\dd-2\chi_Q(\dd,\dd)]$, and is thus determined by its monodromy\footnote{This is the monodromy around $\Nak_{\ff,\dd}(Q)$ and is unrelated to the ``monodromic'' in ``monodromic mixed Hodge module''.}.  Finally, in cohomological degree $2\chi_Q(\dd,\dd)-2\ff\cdot\dd$ the right hand side of \eqref{HDIM} is a vector space spanned by the components of $\Nak(\ff,\dd)$, while the left hand side is a vector space spanned by the components of $\Nak(\ff,\dd)$ on which the monodromy of the underlying perverse sheaf of $\phim{\TTr(W^+)_{\dd}}\nIC_{\Msp_{\ff,\dd}(Q)}$ is trivial.  It follows from the existence of the isomorphism \eqref{HDIM} that these dimensions are the same, and so the monodromy is trivial, and isomorphism \eqref{rDIM} follows.
\end{proof}

\subsubsection{CoHA modules from framed representations}
We recall a general construction, producing modules for cohomological Hall algebras out of moduli spaces of framed quiver representations, see \cite{Soi14} for related examples and discussion.

Let $\dd',\dd''\in \dvs$ be dimension vectors, with $\dd=\dd'+\dd''$.  We define $\AS_{\dd',\dd''^+}^{\zeta^+\sst}(Q^+)\subset \AS_{\dd^+}^{\zeta^+\sst}(Q^+)$ to be the subspace of $\CC Q^+$-modules $\rho$ such that the underlying $\CC \WT{Q}$-module of $\rho$ preserves the $Q_0$-graded flag 
\begin{equation}
\label{flag}
0\subset \CC^{\dd'}\subset \CC^{\dd}
\end{equation}
and for every arrow $c^*$ with $t(c^*)=\infty$ we have $\rho(c^*)(\CC^{\dd'})=0$.
Given such a $\rho$ we obtain a short exact sequence
\begin{equation}
\label{sesrho}
0\rightarrow \rho'\rightarrow \rho\rightarrow\rho''\rightarrow 0
\end{equation}
where $\udim(\rho')=(\dd',0)$ and $\udim(\rho'')=\dd''^+$.  We set
\[
\Msp_{\ff,\dd',\dd''}(Q)\coloneqq \AS_{\dd',\dd''^+}^{\zeta^+\sst}(Q^+)/\Gl_{\dd',\dd''}
\]
where $\Gl_{\dd',\dd''}\subset \Gl_{\dd}$ is the subgroup preserving the flag \eqref{flag}.  There are morphisms 
\begin{align*}
\pi_1\colon \Msp_{\ff,\dd',\dd''}(Q)\rightarrow &\Mst_{\dd'}(\WT{Q})\\
\pi_2\colon \Msp_{\ff,\dd',\dd''}(Q)\rightarrow&\Msp_{\ff,\dd}(Q)\\
\pi_2\colon \Msp_{\ff,\dd',\dd''}(Q)\rightarrow&\Msp_{\ff,\dd''}(Q)
\end{align*}
taking a point representing the short exact sequence \eqref{sesrho} to $\rho',\rho,\rho''$ respectively.  Then in the correspondence diagram
\[
\xymatrix{
&\Msp_{\ff,\dd',\dd''}(Q)\ar[dl]_{\pi_1\times \pi_3}\ar[dr]^{\pi_2}\\
\Mst_{\dd'}(\WT{Q})\times \Msp_{\ff,\dd''}(Q)&&\Msp_{\ff,\dd}(Q)
}
\]
the morphism $\pi_2$ is a proper morphism between smooth varieties, so that as in \S \ref{CCSec} we can use it to define a pushforward in critical cohomology.  Taking some care of the twists, we consider the morphisms of mixed Hodge modules
\begin{align*}
\alpha&\colon\ul{\QQ}_{\Mst_{\dd'}(\WT{Q})\times \Msp_{\ff,\dd''}(Q)}\otimes\LLL^{\heartsuit}\rightarrow (\pi_1\times \pi_3)_*\ul{\QQ}_{\Msp_{\ff,\dd',\dd''}(Q)}\otimes\LLL^{\heartsuit}\\
\beta&\colon\pi_{2,*}\ul{\QQ}_{\Msp_{\ff,\dd',\dd''}(Q)}\otimes\LLL^{\heartsuit}\rightarrow \ul{\QQ}_{\Msp_{\ff,\dd}(Q)}\otimes\LLL^{\chi_{\WT{Q}}(\dd,\dd)-\ff\cdot\dd}
\end{align*}
where $\heartsuit=\chi_{\WT{Q}}(\dd',\dd')+\chi_{\WT{Q}}(\dd'',\dd'')-\ff\cdot \dd''$ and $\beta$ is the Verdier dual of
\[
\ul{\QQ}_{\Msp_{\ff,\dd}(Q)}\otimes\LLL^{\chi_{\WT{Q}}(\dd,\dd)-\ff\cdot\dd}\rightarrow \pi_{2,*}\ul{\QQ}_{\Msp_{\ff,\dd',\dd''}(Q)}\otimes\LLL^{\chi_{\WT{Q}}(\dd,\dd)-\ff\cdot\dd}.
\]
For $\dd,\ff\in\dvs$ we define
\[
\HMall^*_{\ff,\dd}\coloneqq \HO\!\left(\Msp_{\ff,\dd}(Q),\phim{\TTTr(W^+)}\ul{\QQ}_{\Msp_{\ff,\dd}(Q)}\right)\otimes\LLL^{\chi_{\WT{Q}}(\dd,\dd)/2-\ff\cdot\dd}
\]
and we define $\HMall^*_{\ff}\coloneqq \bigoplus_{\dd\in\dvs}\HMall^*_{\ff,\dd}$.  Applying $\phim{\TTTr(W^+)}$ and taking hypercohomology, via the Thom--Sebastiani isomorphism we obtain the morphism
\[
\HO(\phim{\TTTr(W^+)}\beta)\circ\HO(\phim{\TTTr(W^+)}\alpha)\circ\TS\colon \HCoha_{\WT{Q},\WT{W},\dd'}\otimes\HMall^*_{\ff,\dd''}\rightarrow \HMall^*_{\ff,\dd}
\]
endowing $\HMall^*_{\ff}$ with the structure of a $\HCoha_{\WT{Q},\WT{W}}$-module.  By Proposition \ref{NQVDR} there is an isomorphism
\[
\HMall^*_{\ff,\dd}\cong \HO\!\left( \Nak(\ff,\dd),\QQ\right)\otimes\LLL^{\chi_Q(\dd,\dd)-\ff\cdot\dd}.
\]
Here we have used the equality $\dim(\Msp_{\ff,\dd}(Q))=-\chi_{\WT{Q}}(\dd,\dd)+2\ff\cdot\dd$ along with \eqref{HDIM}.  As such, we obtain an action of $\HCoha_{\Pi_Q}\cong\HCoha_{\WT{Q},\WT{W}}$ on the cohomology of Nakajima quiver varieties.  The proof that this action respects the multiplication in the cohomological Hall algebra is the standard modification of the standard proof of associativity in the cohomological Hall algebra.

\begin{remark}
Let $\Nak^{\SP}\!(\ff,\dd)\subset \Nak(\ff,\dd)$ be the subvariety for which the underlying $\ol{Q}$-module of $\rho$ lies in some Serre subcategory $\SP$.  Applying exceptional restriction functors to the above morphisms of mixed Hodge modules, we may likewise define an action of $\HCoha^{\SP}_{\Pi_Q}$ on 
\[
\bigoplus_{\dd\in\dvs}\HOBM(\Nak^{\SP}\!(\ff,\dd),\QQ)\otimes \LLL^{\ff\cdot\dd-\chi_Q(\dd,\dd)},
\]
although we will make no use of this generalisation here.
\end{remark}

\subsubsection{Kac--Moody Lie algebras and quiver varieties}
For the rest of \S \ref{NQVs} we assume that $Q$ has no edge-loops.  Let $i\in Q_0$ be a vertex and let $1_i\in\dvs$ be the basis vector for the vertex $i$.  Then there is an isomorphism
\[
\Psi_i\colon\HCoha_{\WT{Q},\WT{W},1_i}\cong\HO(\AAA{1}/\CC^*,\QQ)\cong \QQ[u]
\]
and we set 
\begin{equation}
\label{Fdef}
\alpha_i\coloneqq\Psi_i^{-1}(1)\in \HCoha_{\WT{Q},\WT{W},1_i}.
\end{equation}
The action of $\alpha_i$ provides a morphism $\alpha_i\cdot\colon \HMall_{\ff,\dd}^*\rightarrow \HMall_{\ff,\dd+1_i}^*$.  Consider the semisimplification map $\JH\colon \Msp_{\ff,\dd}(Q)\rightarrow \Msp_{\dd^+}(\ol{Q_{\ff}})$.  Following Lusztig \cite{Lus91} we consider the Lagrangian subvariety $\Lus(\ff,\dd)=\JH^{-1}(0)$.  There is a contracting $\mathbb{C}^*$ action on $\Nak(\ff,\dd)$, contracting it onto the projective variety $\Lus(\ff,\dd)$, and so we obtain the first in the sequence of isomorphisms
\begin{align}
\label{rtlSwitch}
\HO(\Nak(\ff,\dd),\QQ)\cong&\HO(\Lus(\ff,\dd),\QQ)\\
\nonumber
\cong& \HOc(\Lus(\ff,\dd),\QQ)\\
\nonumber
\cong&\HOBM(\Lus(\ff,\dd),\QQ)^*.
\end{align}
Note that since $\Lus(\ff,\dd)$ is Lagrangian, its top degree compactly supported cohomology is in degree 
\[
\dim(\Nak(\ff,\dd))=2\ff\cdot \dd-2\chi_Q(\dd,\dd).
\]
Setting 
\[
\HMall_{\ff,\dd}\coloneqq \HOBM\left(\Lus(\ff,\dd),\QQ\right)\otimes \LLL^{\ff\cdot\dd-\chi_Q(\dd,\dd)};\quad\quad
\HMall_{\ff}\coloneqq \bigoplus_{\dd\in\dvs}\HMall_{\ff,\dd}
\]
we deduce that there is an $\HCoha_{\Pi_Q}$ action on $\HMall_{\ff}$ by \textit{lowering} operators, induced by the $\HCoha_{\Pi_Q}$ action on $\bigoplus_{\dd}\HO(\Nak(\ff,\dd),\QQ)$ by \textit{raising} operators, for which $(\alpha_i\cdot )^*$ is the lowering operator constructed by Nakajima.  By the degree bound on the cohomology of $\HOc(\Lus(\ff,\dd),\QQ)$, 
\[
\HO^i\!\Mall_{\ff,\dd}=\begin{cases} 0&\textrm{if }i<0\\
\QQ\cdot\{\textrm{top-dimensional components of }\Lus(\ff,\dd)\}&\textrm{if }i=0.\end{cases}
\]
The main theorem regarding the operators $(\alpha_i\cdot)^*$ is the following part of Nakajima's work.
\begin{theorem}\cite{Nak94,Nak98}
There is an action of the Kac--Moody Lie algebra $\mathfrak{g}$ on each $\HMall_{\ff}$ sending the generators $f_i$ for $i\in Q_0$ to the operators $(\alpha_i\cdot)^*$.  With respect to this $\mathfrak{g}_Q$ action, the submodule $\HO^0\!\Mall_{\ff,\dd}$ is the irreducible highest weight module with highest weight $\ff$.
\end{theorem}

The original statement of Nakajima's theorem does not involve any vanishing cycles, i.e. it only involves the right hand side of the isomorphism \eqref{HDIM}.  Likewise, the correspondences considered in \cite{Nak94, Nak98} come from an action of the Borel--Moore homology of the stack of $\Pi_Q$-representations, not from the critical cohomology of the stack of $\Pi_Q$-representations.  For the compatibility between the two actions via the dimensional reduction isomorphisms \eqref{HDIM} and \eqref{drnoalg} see e.g. \cite[Sec.4]{YZ16}.

\subsection{The subalgebra $\lP_{0}\!\HO^0\!\!\Coha_{\Pi_Q}$}
\label{KacH}
In this section we concentrate on the case in which $\SP=\CC\ol{Q}\lmod$, $G$ is trivial and $\zeta=(0,\ldots,0)$ is the degenerate stability condition (i.e. we essentially do not consider stability conditions).  Note that by \eqref{Kaccha}, the Lie algebra $\mathfrak{g}_{\Pi_Q}$ is concentrated in cohomological degrees less than or equal to zero, and so by Theorem \eqref{UEAthm} there is an isomorphism
\begin{equation}
\label{UTO}
\lP_{0}\!\HO^0\!\!\Coha_{\Pi_Q}\cong \UEA\!\left(\HO^0(\mathfrak{g}_{\Pi_Q})\right).
\end{equation}
We thus reduce the problem of calculating the left hand side to that of calculating $\HO^0(\mathfrak{g}_{\Pi_Q})$.  By \eqref{Kaccha} again, there is an equality
\begin{equation}
\label{clue1}
\dim(\HO^0(\mathfrak{g}_{\Pi_Q,\dd}))=\kac_{Q,\dd}(0).
\end{equation}
If $\dd_i\neq 0$ for some $i\in Q_0$ for which there is an edge-loop $b$, there is a free action of $\mathbb{F}_q$ on the set of absolutely indecomposable $\dd$-dimensional $\CC Q$-modules, defined by
\[
z\cdot \rho(a)=\begin{cases} \rho(a)+ z\cdot \Id_{e_i\cdot \rho} &\textrm{if }a=b\\
\rho(a) &\textrm{otherwise}\end{cases}
\]
and thus $\kac_{Q,\dd}(0)=0$, and so $\HO^0(\mathfrak{g}_{\Pi_Q,\dd})=0$.  It follows that for arbitrary $\dd\in\dvs$, if $\dd_i\neq 0$ for a vertex $i$ supporting an edge loop, then $\lP_{0}\!\HO^0\!\!\Coha^{\zeta}_{\Pi_Q,\dd}=0$.

We define $Q'$, the \textit{real subquiver} of $Q$, to be the full subquiver of $Q$ containing those vertices of $Q$ that do not support any edge-loops, along with all arrows between these vertices.  From the above considerations, we deduce that the inclusion of algebras in $\DlMHM(\Msp(\ol{Q}))$
\[
\bigoplus_{\dd\in\mathbb{N}^{Q'_0}}\rCoha_{\Pi_Q,\dd}= \rCoha_{\Pi_{Q'}}\hookrightarrow \rCoha_{\Pi_Q}
\]
becomes an isomorphism after applying $\lP_{0}\!\HO^0$.

Hausel's (first) famous theorem regarding Kac polynomials \cite{Hau10} states that
\begin{equation}
\label{clue2}
\kac_{Q',\dd}(0)=\dim(\mathfrak{g}_{Q',\dd})
\end{equation}
where $\mathfrak{g}_{Q'}$ is the Kac--Moody Lie algebra associated to the quiver (without edge-loops) $Q'$.  We will not recall the definition of $\fg_{Q'}$, since in any case it is a special case of the Borcherds--Bozec algebra (i.e. the case in which $I^{\Imag}=\emptyset$), which we recall in \S \ref{BBA} below.  
\sssct
Comparing \eqref{clue1} and \eqref{clue2} leads to the identity
\begin{equation}
\label{khaus}
\dim(\HO^0(\mathfrak{g}_{\Pi_Q,\dd}))=\dim(\mathfrak{g}_{Q',\dd})
\end{equation}
and from there to the obvious conjecture regarding the algebra  $\lP_{0}\!\HO^0\!\!\Coha^{\zeta}_{\Pi_Q}$, which we now prove.

\begin{theorem}
\label{KMLA}
There is an isomorphism of algebras 
\begin{equation}
\label{zpkm}
\UEA(\mathfrak{n}_{Q'}^-)\cong\lP_{0}\!\HO^0\!\!\Coha^{\zeta}_{\Pi_Q}
\end{equation}
where $\mathfrak{n}_{Q'}^-$ is the negative part of the Kac--Moody Lie algebra for the real subquiver of $Q$.  Moreover the isomorphism restricts to an isomorphism with the BPS Lie algebra
\begin{equation}
\label{LAI}
\fn^-_{Q'}\cong \HO^0(\fg_{\Pi_Q})
\end{equation}
under the isomorphism \eqref{UTO}.
\end{theorem}
Although it is most natural to think of $\HO^0(\fg_{\Pi_Q})$ as a Lie algebra of raising operators, we have $\fn^-_{Q'}$ on the left hand side of \eqref{LAI} to match with the fact that, via Verdier duality \eqref{rtlSwitch} Nakajima's original action of this Lie algebra is via lowering operators on $\HOBM(\Lus(\ff,\dd),\QQ)$.  There is a canonical isomorphism $\fn^-_{Q'}\cong\fn^+_{Q'}$ given by the Chevalley involution..
\renewcommand*{\proofname}{Proof of Theorem \ref{KMLA}}
\begin{proof}
We construct the isomorphism \eqref{LAI}, then the isomorphism \eqref{zpkm} is constructed via \eqref{UTO}.

Consider the dimension vector $e=1_i$, where $i$ does not support any edge-loops.  The coarse moduli space $\Msp_{e}(\ol{Q})$ is just a point, and so the less perverse filtration on $\HCoha_{\Pi_Q,e}=\HO(\B \CC^*,\QQ)$ is just the cohomological filtration.  In particular, the element $\alpha_i$ from \eqref{Fdef} lies in less perverse degree 0.  On the other hand, there is an isomorphism $\Msp_{e}(\WT{Q})\cong \AAA{1}$ and writing
\[
\HCoha_{\WT{Q},\WT{W},e}=\HO(\AAA{1},\nIC_{\AAA{1}})\otimes \HO(\B\CC^*,\QQ)_{\vir}
\]
we see that $\alpha_i$ has perverse degree 1, i.e. by definition it is an element of $\fg_{\Pi_Q,e}$.

We claim that there is a Lie algebra homomorphism $\Phi\colon \mathfrak{n}^-_{Q'}\rightarrow \HO^0(\mathfrak{g}_{\Pi_Q})$ sending $f_i$ to $\alpha_i$.  The algebra $\mathfrak{n}_{Q'}^-$ is the free Lie algebra generated by $f_i$ for $i\in Q_0$ subject to the Serre relations:
\[
[f_i,\cdot]^{1-(1_i,1_j)_Q}(f_j)=0 
\]
where $(\cdot,\cdot)_Q$ is the symmetrized Euler form \eqref{Eulerform}.  So to prove the claim we only need to prove that the elements $\alpha_i$ satisfy the Serre relations.  This follows from the stronger claim: for any distinct pair of vertices $i,j\in Q_0$ if we set $\dd=1_j+(1-(1_i,1_j))1_i$ then there is an equality $\mathfrak{g}_{\Pi_Q,\dd}=0$.  This follows from Proposition \ref{BPSvanProp}.  Alternatively, this follows from \eqref{Kaccha} and the claim that $\kac_{Q,\dd}(t)=0$.  Since the Kac polynomial is independent of the orientation of $Q$ this equality is clear: if all the arrows are directed from $i$ to $j$, there are no indecomposable $\dd$-dimensional $KQ$-modules for any field $K$.

We next claim that the morphism $\Phi\colon \mathfrak{n}^-_{Q'}\rightarrow \HO^0(\mathfrak{g}_{\Pi_Q})$ is injective.  This follows from Nakajima's theorem, i.e. we have a commutative diagram
\[
\xymatrix{
\mathfrak{n}^-_{Q'}\ar[dr]_{u}\ar[r]_{\Phi}^{f_i\mapsto \alpha_i}&\HO^0(\mathfrak{g}_{\Pi_Q})\ar[d]\\
&\End_{\QQ}(\bigoplus_{\ff} \HMall^*_{\ff})
}
\]
via the module structure of each $\HMall^*_{\ff}$, and the morphism $u$ is injective since the representation $\bigoplus_{\ff}\HMall^*_{\ff}$ is a faithful representation (it contains the sum of all highest weight representations).

By Hausel's identity \eqref{clue2} the graded dimensions of the source and target of $\Phi$ are the same, and so $\Phi$ is an isomorphism.  Comparing with \eqref{UTO}, the induced morphism
\[
\UEA(\mathfrak{n}^-_{Q'})\rightarrow \lP_{0}\!\HO^0\!\!\Coha^{\zeta}_{\Pi_Q}
\]
is an isomorphism.
\end{proof}
\renewcommand*{\proofname}{Proof}
\subsubsection{Support dimension grading on $\fg_{\Pi_Q}$} The Lie algebra $\fg_{\Pi_Q}$ carries a cohomological grading, by construction, and we have shown that the highest graded piece with respect to the cohomological grading is isomorphic to $\fn_{Q'}^-$.  In this section we use the semisimplicity of the category of pure MHMs to provide an alternative grading on $\fg_{\Pi_Q}$, given by considering dimensions of supports in the decomposition theorem.  Again, we will find that $\fn_{Q'}^-$ occurs as a graded subalgebra.

By \S \ref{RLAsec} and Theorem \ref{purityThm}, there is a pure mixed Hodge module $\BPSh_{\Pi_Q}$ in $\MHM(\Msp(\ol{Q}))$ which moreover carries the structure of a Lie algebra object in this tensor category, such that we recover $\fg_{\Pi_Q}$ by applying the hypercohomology functor to $\BPSh_{\Pi_Q}$.  Since $\BPSh_{\Pi_Q}$ is pure, there is a canonical decomposition
\begin{equation}
\label{decar}
\BPSh_{\Pi_Q}\cong \bigoplus_{i\in S} \ICS_{\ol{Z_i}}(\mathcal{F}_i)
\end{equation}
where $Z_i\subset \Msp(\ol{Q})$ are distinct locally closed subvarieties indexed by a set $S$, $\mathcal{F}_i$ are pure semisimple variations of pure Hodge structure on $Z_i$, and $\ICS_{\ol{Z_i}}(\mathcal{F}_i)$ are their intermediate extensions.  Define
\[
S^{(d)}\coloneqq\{i\in S\;\lvert\; \dim(Z_i)=d\}.
\]
Then \eqref{decar} gives us the canonical decomposition
\[
\BPSh_{\Pi_Q}\cong\bigoplus_{d\geq 0}\BPSh_{\Pi_Q}^{(d)}
\]
where $\BPSh_{\Pi_Q}^{(d)}\coloneqq \bigoplus_{Z_i,\; i\in S^{(d)}} \ICS_{\ol{Z_i}}(\mathcal{F}_i)$.  Thus there is a canonical decomposition of $\fg_{\Pi_Q}=\bigoplus_{d\in\mathbb{N}}\fg_{\Pi_Q}^{(d)}$ with summands $\fg_{\Pi_Q}^{(d)}\coloneqq \HO(\BPSh_{\Pi_Q}^{(d)})$.
\begin{proposition}
The BPS Lie algebra $\fg_{\Pi_Q}$ carries an $\mathbb{N}$-grading, with graded pieces given by $\fg_{\Pi_Q}^{(d)}$.
\end{proposition}
\begin{proof}
It is sufficient to show that for $d'+d''\neq d$ the morphism
\begin{equation}
\label{agh}
\oplus_{\red,*}(\BPSh_{\Pi_Q}^{(d')}\boxtimes \BPSh_{\Pi_Q}^{(d'')})\rightarrow \BPSh_{\Pi_Q}^{(d)}
\end{equation}
is zero.  By purity of $\BPSh_{\Pi_Q}$ and finiteness of $\oplus_{\red}\colon \Msp(\ol{Q})\times \Msp(\ol{Q})\rightarrow \Msp(\ol{Q})$, the domain of \eqref{agh} is a direct sum of simple mixed Hodge modules supported on $(d'+d'')$-dimensional subvarieties of $\Msp(\ol{Q})$, while the target is by definition a direct sum of simple mixed Hodge modules supported on $d$-dimensional varieties.
\end{proof}
\begin{proposition}
There is an isomorphism of algebras $\fg_{\Pi_Q}^{(0)}\cong \fn_{Q'}^-$.
\end{proposition}
\begin{proof}
Firstly, if $\mathcal{F}$ is a MHM with zero-dimensional support, then $\HO(\mathcal{F})$ is concentrated in degree zero.  For each vertex $i\in Q'_0\subset Q_0$ we have $\BPSh_{\Pi_Q,1_i}\cong \ul{\QQ}_{\Msp_{1_i}(\ol{Q})}$, a simple MHM with zero-dimensional support.  Now by Theorem \ref{KMLA}, if we apply $\HO$ to the Lie subalgebra $\mathcal{F}$ generated by the objects $\BPSh_{\Pi_Q,1_i}$ for $i\in Q'_0$, we obtain the whole of $\HO^0\!\fg_{\Pi_Q}$, and so $\mathcal{F}$ contains all of the summands of $\BPSh_{\Pi_Q}$ with zero-dimensional support.
\end{proof}

\subsubsection{Ubiquity of Kac--Moody Lie algebras}
Theorem \ref{KMLA} shows that Kac--Moody Lie algebras are a natural piece of the BPS Lie algebra in the basic case in which we do not modify potentials, and do not restrict to a Serre subcategory.  The next proposition shows that Kac--Moody Lie algebras are a somewhat universal feature of the cohomological Hall algebras that we are considering.
\begin{proposition}
\label{bels}
Let $Q$ be a quiver, let $Q'$ be the real subquiver of $Q$, let $W_0\in\CC\ol{Q}_{\cyc}$ be a potential such that $\WT{W}+W_0$ is quasihomogeneous, and let $\SP$ be any Serre subcategory of $\CC\ol{Q}\lmod$ containing each of the 1-dimensional simple modules $S_i$ with dimension vector $1_i$, for $i\in Q'_0$.  Then there is a $\dvs$-graded inclusion of Lie algebras $\fn^-_{Q'}\hookrightarrow \fg_{\Pi_Q,W_0}^{\SP}$ with image the Lie subalgebra of $\fg_{\Pi_Q,W_0}^{\SP}$ generated by the graded pieces $\fg_{\Pi_Q,W_0,1_i}^{\SP}$ for $i\in Q'_0$.
\end{proposition}
\begin{proof}
By the proof of Theorem \ref{KMLA} the subalgebra $\fn^-_{Q'}\subset \fg_{\Pi_Q}$ is obtained by applying $\HO$ to the Lie subalgebra object $\mathcal{G}$ of $\BPSh_{\Pi_Q}$ generated by the objects $\BPSh_{\Pi_Q,1_{i}}$ for $i\in Q'_0$, i.e. $\fn^-_{Q'}\cong\HO\mathcal{G}$.  There is a decomposition
\[
\BPSh_{\Pi_Q}\cong \mathcal{G}\oplus\mathcal{L}
\]
of mixed Hodge modules, by purity of $\BPSh_{\Pi_Q}$, and hence the inclusion $\mathcal{G}\hookrightarrow \BPSh_{\Pi_Q}$ splits in the category of mixed Hodge modules equipped with a Lie algebra structure.  Recall that we denote by $\varpi'_{\red}\colon \Msp^{\Sp}(\ol{Q})\hookrightarrow \Msp(\ol{Q})$ the inclusion.  Applying $\HO\varpi'_{\red,*}\varpi_{\red}'^!\phim{W_0}$ gives an inclusion of Lie algebras
\[
\HO\varpi'_{\red,*}\varpi_{\red}'^!\phim{W_0}\mathcal{G}\hookrightarrow \fg_{\Pi_Q,W_0}^{\SP}.
\]
Each mixed Hodge module $\mathcal{G}_{\dd}$ is supported at the origin of $\Msp_{\dd}(\ol{Q})$, since $\mathcal{G}$ is generated by mixed Hodge modules supported on the nilpotent locus.  The Serre subcategory $\SP$ contains all nilpotent $\CC\ol{Q}$-modules supported on the subquiver $Q'$, since it is closed under extensions.  As such the natural morphisms
\[
\varpi'_{\red,*}\varpi_{\red}'^!\mathcal{G}\rightarrow \mathcal{G};\quad\quad
\phim{\TTr(W_0)}\mathcal{G}\rightarrow \mathcal{G},
\]
are isomorphisms, and $\varpi'_{\red,*}\varpi_{\red}'^!\phim{\TTr(W_0)}\mathcal{G}\cong\mathcal{G}$ as a Lie algebra object\footnote{In particular,  $\varpi'_{\red,*}\varpi_{\red}'^!\phim{\TTr(W_0)}\mathcal{G}$ is an object in the subcategory $\MHM(\Msp(\ol{Q}))$ of monodromy-free MMHMs.} in $\MMHM(\Msp(\ol{Q}))$, proving the corollary.
\end{proof}

\subsection{The subalgebra $\lP_{0}\!\HO^0\!\!\Coha^{\SSN}_{\Pi_Q}$}
\label{BBA}
Moving to the case of strongly semi-nilpotent $\Pi_Q$-modules (see \S \ref{2desc}) we calculate the subalgebra $\lP_{0}\!\HO^0\!\!\Coha^{\SSN}_{\Pi_Q}$, in order to compare with the work of Bozec \cite{Bo16} on Borcherds--Kac--Moody Lie algebras associated to quivers with loops.  Interestingly, we find that the BPS Lie algebra $\fg_{\Pi_Q}^{\SSN}$ is \textit{not} identified with Bozec's Lie algebra $\fg_Q$ under the natural isomorphism between their two enveloping algebras, although the two Lie algebras are isomorphic.
        
First we note that by \eqref{kssn} the Lie algebra $\fg_{\Pi_Q}^{\SSN}$ is concentrated in cohomologically nonnegative degrees.  Applying $\HO^0$ to \eqref{absPBW}, the morphism $\Sym(\HO^0\!(\fg^{\SSN}_{\Pi_Q}))\rightarrow \HO^0\!\!\Coha^{\SSN}_{\Pi_Q}$ is an isomorphism, and thus there is an identity
\begin{equation}
\label{Hzerodo}
\UEA(\HO^0\!(\fg_{\Pi_Q}^{\SSN}))=\HO^0\!\!\Coha^{\SSN}_{\Pi_Q}.
\end{equation}
Comparing with \eqref{UEAiso} we deduce that $\lP_{0}\!\HO^0\!\!\Coha^{\SSN}_{\Pi_Q}=\HO^0\!\!\Coha^{\SSN}_{\Pi_Q}$ and so for the rest of \S \ref{BBA} we just write $\HO^0\!\!\Coha^{\SSN}_{\Pi_Q}$ to denote this subalgebra.

Recall that we define the relative CoHA by restriction: $\rCoha^{\SSN}_{\Pi_Q}=\varpi'_{\red,*}\varpi'^!_{\red}\rCoha_{\Pi_Q}$.  For general Serre subcategories $\SP$, purity of $\rCoha^{\SP}_{\Pi_Q}$ does not follow from purity of $\rCoha_{\Pi_Q}$.  However in the current setting we can make use of the following observation\footnote{I am thankful to Geordie Williamson for pointing this out}:
\begin{lemma}
\label{G_lemma}
Let $X$ be a $\mathbb{C}^*$-equivariant variety, such that for all $x\in X$ the limit $\lim_{t\rightarrow 0} t\cdot x$ exists, and there is a retract of varieties $p\colon X\rightarrow X^{\mathbb{C}^*}$ sending $x$ to this limit.  Let $\mathcal{F}\in\Db(\MHM(X))$ be pure.  Let $i\colon X^{\mathbb{C}^*}\hookrightarrow X$ be the inclusion.  Then $i^*\mathcal{F}\in\Db(\MHM(X))$ is pure.
\end{lemma}
Since $i^*$ is not exact, it may be the case (and is the case in examples we consider in this section) that $i^*\mathcal{F}$ will not be a MHM, even if $\mathcal{F}$ is.
\begin{proof}
The conditions of the lemma ensure that there is an isomorphism $i^*\mathcal{F}\rightarrow p_*\mathcal{F}$ in $\Db(\MHM(X))$.  Purity of $\mathcal{F}$ ensures that $i^*\mathcal{F}$ is pure above, and $p_*\mathcal{F}$ is pure below.
\end{proof}
Since by Corollary \ref{relPurity} the relative Hall algebra $\rCoha_{\Pi_Q}$ is pure, we deduce:
\begin{corollary}
\label{PBBC}
The relative CoHA $\rCoha^{\SSN}_{\Pi_Q}\in\DlMHM(\Msp(\ol{Q}))$ is \textit{pure}.
\end{corollary}
\subsubsection{The Borcherds--Bozec algebra}
We write $Q_0=I^{\reel}\coprod I^{\Imag}$, where $I^{\reel}$ is the set of vertices that do not support an edge-loop, and $I^{\Imag}$ is the set vertices that do.  We furthermore decompose $I^{\Imag}=I^{\isot}\coprod I^{\hype}$ where $I^{\isot}$ is the set of vertices supporting exactly one edge-loop, and the vertices of $I^{\hype}$ support more than one edge-loop.  

Starting with the quiver $Q$ the \textbf{Borcherds--Bozec} algebra $\fg_Q$ is defined as follows.  We set
\[
I_{\infty}=(I^{\reel}\times\{1\})\coprod (I^{\Imag}\times\mathbb{Z}_{>0})
\]
and we extend the form \eqref{Eulerform} to a bilinear form on $\mathbb{N}^{I_{\infty}}$ by setting
\[
((1_{i',n}),(1_{j',m}))=mn(1_{i',}1_{j'})_Q
\]
and extending linearly.  The Lie algebra $\mathfrak{g}_Q$ is a Borcherds algebra associated to a generalised Cartan datum for which the Cartan matrix is the form $(\cdot,\cdot)$ expressed in the natural basis of $\mathbb{N}^{I_{\infty}}$.  More explicitly, we define $\fg_Q$ to be the free Lie algebra generated over $\QQ$ by $h_{i'},e_i,f_i$ for $i'\in Q_0$ and $i\in I_{\infty}$ subject to the relations
\begin{align*}
[h_{i'},h_{j'}]=&0\\
[h_{j'},e_{(i',n)}]=&n(1_{j'},1_{i'})_Q\cdot e_{(i',n)}\\
[h_{j'},f_{(i',n)}]=&-n(1_{j'},1_{i'})_Q\cdot f_{(i',n)}\\
[e_{j},\cdot]^{1-(j,i)}(e_i)=[f_{j},\cdot]^{1-(j,i)}(f_i)=&0 &\textrm{if }j\in I^{\reel}\times\{1\},\;i\neq j\\
[e_i,e_j]=[f_i,f_j]=&0&\textrm{if }(i,j)=0\\
[e_i,f_j]=&\delta_{i,j}nh_{i'}&\textrm{if }i=(i',n).
\end{align*}

The positive half $\mathfrak{n}_Q^+$ has an especially quick presentation: it is the Lie algebra over $\QQ$ freely generated by $e_{i}$ for $i\in I_{\infty}$, subject to the relations
\begin{align}
[e_{i},\cdot]^{1-(i,j)}(e_j)=0 &&\textrm{if }i\in I^{\reel}\times\{1\},\;i\neq j\label{Serre1}\\
[e_i,e_j]=0 &&\textrm{if }(i,j)=0.\label{Serre2}
\end{align}

\subsubsection{Lagrangian subvarieties}
Define $\Lambda(\dd)=\AS_{\dd}^{\mathcal{SSN}}(\ol{Q})\cap \mu^{-1}(0)$, the subvariety of $\AS_{\dd}(\ol{Q})$ parameterising strictly semi-nilpotent $\Pi_Q$-modules.  By \cite[Thm.1.15]{Bo16}, this is a Lagrangian subvariety of $\AS_{\dd}(\ol{Q})$.  If $\dd=n\cdotsh 1_{i'}$ for some $i'\in Q_0$ supporting $g$ loops, with $g\geq 1$, then by \cite[Thm.1.4]{Bo16} the irreducible components of $\Lambda(\dd)$ are indexed by tuples $(n^1,\ldots,n^r)$ such that $\sum n^s=n$, and the numbers $n^s$ are weakly decreasing if $g=1$.  Let $I_l$ be the two-sided ideal in $\CC \ol{Q}$ containing all those paths in $\ol{Q}$ containing at least $l$ instances of arrows $a\in Q_1$.  The tuple $c$ corresponding to a component $\Lambda(\dd)_c$ is given by the successive dimensions of the subquotients in the filtration
\[
0=I_r\cdotsh \rho\subset I_{r-1}\cdotsh \rho\subset\ldots \rho
\]
for $\rho$ a module parameterised by a generic point on $\Lambda(\dd)_c$.  For example there is an equality
\[
\Lambda(\dd)_{(n)}=\AS_{\dd}(Q^{\opp})\subset \AS^{\SSN}_{\dd}(\ol{Q})\cap \mu^{-1}(0).
\]
To translate Bozec's results into our setting we follow the arguments of \cite[Sec.2]{RS17}.  Unpacking the definitions, we have
\[
\HCoha^{\SSN}_{\Pi_Q}\coloneqq \bigoplus_{\dd\in\dvs}\HOBM\left(\Lambda(\dd)/\Gl_{\dd},\QQ\right)\otimes\LLL^{-\chi_Q(\dd,\dd)}.
\]
Since $\Lambda(\dd)$ is a Lagrangian subvariety of the $2(\dd\cdot\dd-\chi_Q(\dd,\dd))$-dimensional subvariety $\AS_{\dd}(\ol{Q})$, the irreducible components of $\Lambda(\dd)/\Gl_{\dd}$ are $-\chi_Q(\dd,\dd)$-dimensional.  It follows that $\HO^0\!\!\Coha^{\SSN}_{\Pi_Q}$ has a natural basis given by $[\Lambda(\dd)_e]$ where $\Lambda(\dd)_e$ are the irreducible components of $\Lambda(\dd)$.
\begin{theorem}\cite{Lucien22}(cf. also \cite[Prop.1.18, Thm.3.34]{Bo16})
\label{BoThm}
There is an isomorphism of algebras $\UEA(\mathfrak{n}_Q^+)\rightarrow \HO^0\!\!\Coha^{\SSN}_{\Pi_Q}$ which sends $e_{(i',n)}$ to $[\Lambda(n\cdotsh 1_{i'})_{(n)}]$.
\end{theorem}
Combining with \eqref{Hzerodo} we obtain an isomorphism
\begin{equation}
\label{UEAiso}
F\colon \UEA(\mathfrak{n}_Q^+)\xrightarrow{\cong}\UEA(\HO^0\!(\fg_{\Pi_Q}^{\SSN})).
\end{equation}
\begin{corollary}
\label{absIso}
There exists an isomorphism of Lie algebras $\mathfrak{n}_Q^+\cong\HO^0\!(\fg_{\Pi_Q}^{\SSN})$ where the left hand side is the Borcherds--Bozec algebra of the full quiver $Q$.
\end{corollary}
\begin{proof}
Let $S$ be a minimal $\dvs$-graded generating set of $\HO^0\!(\fg_{\Pi_Q}^{\SSN})$.  Then $F^{-1}(S)$ is a minimal generating set of $\UEA(\mathfrak{n}_Q^+)$ and so $\lvert S_{\dd}\lvert$ is the number of degree $\dd$ generators of $\mathfrak{n}_Q^+$.  We show that the elements of $S$ satisfy the Serre relations.

Let $g,g'\in S$ be of degree $m\cdotsh 1_{i'}$ and $n\cdotsh 1_{j'}$ respectively.  If $(1_{i'},1_{j'})_Q=0$ then there are no edges between $i'$ and $j'$, and $g$ and $g'$ commute by the definition of the CoHA multiplication, i.e. they satisfy the Serre relation $[g,g']=0$, dealing with \eqref{Serre2}.  

We next consider \eqref{Serre1}.  Assume that there are no edge-loops of $Q$ at $i'$, so that up to a scalar multiple $g=F(e_{(i',1)})$.  Write $g'$ as a linear combination of monomials $\prod_{r=1}^l F(e_{(j',t_r)})$ for $t_r\in\mathbb{N}$ summing to $n$.  Since $[F(e_{(i',1)}),\cdot]$ is a derivation, the identity 
\[
[F(e_{(i',1)}),g']^{1-n(1_{i'},1_{j'})_Q}=0
\]
follows from the identities 
\[
[F(e_{(i',1)}),F(e_{(j',t_r)})]^{1-t_r(1_{i'},1_{j'})_Q}=0.
\]
So the generators $S$ satisfy the Serre relations and there is a surjection $\mathfrak{n}_Q^+\rightarrow \HO^0\!(\fg_{\Pi_Q}^{\SSN})$, which is injective since the graded pieces have the same dimensions.
\end{proof}

Although Corollary \ref{absIso} establishes that they are abstractly isomorphic, we spend the rest of \S \ref{BBA} investigating the \textit{difference} between the two Lie subalgebras $F(\mathfrak{n}_Q^+)$ and $\HO^0\!(\fg_{\Pi_Q}^{\SSN})$.
\subsubsection{Isotropic vertices} 
\label{IsotSec} Let $i'\in I^{\isot}$, and set $\dd=n\cdotsh 1_{i'}$.  Let 
\[
\Delta_n\colon \AAA{3}\hookrightarrow \Msp_{\dd}(\WT{Q})
\]
be the inclusion, sending $(z_1,z_2,z_3)\in\AAA{3}$ to the $\Jac(\WT{Q},\WT{W})$-representation for which the action of the three arrows $a,a^*,\omega_{i'}$ is scalar multiplication by $z_1,z_2,z_3$ respectively.  Then by \cite[Thm.5.1]{preproj} there is an isomorphism
\begin{equation}
\label{commBPS}
\BPSh_{\WT{Q},\WT{W},\dd}\cong \Delta_{n,*}\ul{\QQ}_{\AAA{3}}\otimes \LLL^{-3/2}=\Delta_{n,*}\nIC_{\AAA{3}}.
\end{equation}
Thus by \eqref{pfBP} there is an isomorphism
\begin{equation}
\label{2dcBPS}
\BPSh_{\Pi_Q,\dd}\cong \Delta_{\red,n,*}\ul{\QQ}_{\AAA{2}}\otimes \LLL^{-1}=\Delta_{\red,n,*}\nIC_{\AAA{2}}
\end{equation}
where $\Delta_{\red,n}\colon\AAA{2}\hookrightarrow \Msp_{n\cdot 1_{i'}}(\ol{Q})$ is the inclusion taking $z_1,z_2$ to the module $\rho$ for which $a$ and $a^*$ act via multiplication by $z_1$ and $z_2$ respectively.  Thus we find that
\begin{align}
\BPSh^{\SSN}_{\Pi_Q,\dd}\cong \varpi'_{\red,*}\varpi'^!_{\red}\Delta_{\red,n,*}\nIC_{\AAA{2}}\label{BPSSSN}
\cong \Delta^{\SSN}_{\red,n,*}\ul{\QQ}_{\AAA{1}}
\end{align}
where $\Delta^{\SSN}_{\red,n,*}\colon \AAA{1}\hookrightarrow \Msp_{n\cdot 1_{i'}}(\ol{Q})$ takes $z$ to the module $\rho$ for which $a^*$ acts via multiplication by $z$ and $a$ acts via the zero map.  By \eqref{2dInt} we deduce the following proposition.
\begin{proposition}
\label{IsoPure}
There is an isomorphism in $\DlMHM(\Msp(\ol{Q}))$
\begin{equation}
\label{gen}
\bigoplus_{n\geq 0}\rCoha^{\SSN}_{\Pi_Q,n\cdot 1_{i'}}\cong \Sym_{\oplus_{\red}}\!\left(\bigoplus_{n\geq 1}\Delta^{\SSN}_{\red,n,*}\ul{\QQ}_{\AAA{1}}\otimes\HO(\B \CC^*,\QQ)\right).
\end{equation}
\end{proposition}
In particular, we see directly that $\bigoplus_{n\geq 0}\rCoha^{\SSN}_{\Pi_Q,n\cdot 1_{i'}}$ is pure (without using Corollary \ref{PBBC}), as is
\begin{equation}
\label{fen}
\bigoplus_{n\geq 1}\BPSh^{\SSN}_{\Pi_Q,n\cdot 1_{i'}}\cong \bigoplus_{n\geq 1}\Delta^{\SSN}_{\red,n,*}\ul{\QQ}_{\AAA{1}}.
\end{equation}

\begin{remark}
The Verdier dual of \eqref{gen} is the isomorphism
\[
\bigoplus_{n\geq 0}\left(\Mst^{\SSN}_{n\cdot 1_{i'}}(\Pi_Q)\rightarrow \Msp_{n\cdot 1_{i'}}(\ol{Q})\right)_!\ul{\QQ}_{\Mst^{\SSN}_{n\cdot 1_{i'}}(\Pi_Q)}\cong \Sym_{\oplus}\!\left(\bigoplus_{n\geq 1,m\geq 0}\Delta^{\SSN}_{\red,n,*}\ul{\QQ}_{\AAA{1}}\otimes\LLL^{-m-1}\right).
\]
\end{remark}

From \eqref{BPSSSN} we deduce that there are isomorphisms
\begin{equation}
\label{IsotDone}
\Psi\colon \fg^{\SSN}_{\Pi_Q,\dd}\cong \HO(\AAA{1},\QQ)\cong\QQ
\end{equation}
as cohomologically graded vector spaces, i.e. $\fg^{\SSN}_{\Pi_Q,\dd}$ is one dimensional and concentrated in cohomological degree zero.  We denote by 
\begin{equation}
\label{Gdef}
\alpha_{i',n}=\Psi^{-1}(1)
\end{equation}
a basis element, so $\fg^{\SSN}_{\Pi_Q,\dd}=\QQ\cdotsh \alpha_{i',n}$.  By using the lift of the BPS Lie algebra $\mathfrak{g}_{\Pi_Q}$ to a Lie algebra object in $\MHM(\Msp(\Pi_Q))$ one can see directly (e.g. without the aid of Corollary \ref{absIso}) that for $m,n\in\ZZ_{\geq 1}$ with $m\neq n$ there is an equality $[\alpha_{i',m},\alpha_{i',n}]=0$.  In a little more detail, this follows because the Lie bracket 
\[
\fg^{\SSN}_{\Pi_Q,n\cdot 1_i}\otimes \fg^{\SSN}_{\Pi_Q,m\cdot 1_i}\rightarrow \fg^{\SSN}_{\Pi_Q,(m+n)\cdot 1_i}
\]
is defined by applying $\HO\varpi'_{\red,*}\varpi'^!_{\red}$ to the morphism of mixed Hodge modules
\begin{equation}
\label{nzer}
\Delta_{\red,n,*}\nIC_{\AAA{2}}\boxtimes_{\oplus_{\red}} \Delta_{\red,m,*}\nIC_{\AAA{2}}\rightarrow \Delta_{\red,m+n,*}\nIC_{\AAA{2}}.
\end{equation}
Since $m\neq n$ the morphism $\oplus_{\red}\circ (\Delta_{\red,n}\times\Delta_{\red,m})\colon\AAA{4}\rightarrow \Msp_{m+n}(\ol{Q})$ is injective.  It follows that the left hand side of \eqref{nzer} is simple, and not isomorphic to the (simple) right hand side, which has 2-dimensional support.  It follows that \eqref{nzer} is the zero map, since it is a morphism between distinct simple objects.  
%Alternatively, one may deduce \eqref{alpharel} from the fact that there is an embedding 
%\[
%\bigoplus_{n\geq 1}\fg^{\SSN}_{\Pi_Q,n\cdot 1_i}\hookrightarrow \HO^0\!\!\Coha_{\Pi_Q}^{\SSN}
%\]
%and the target is commutative by Theorem \ref{BoThm}.
%Alternatively, the identity \eqref{alpharel} is explained by the following
\begin{proposition}
Let $i'\in Q_0^{\isot}$, and set $\dd= n\cdotsh 1_{i'}$ as above.  Up to multiplication by a scalar, there is an identity $\alpha_{i',n}=[\Lambda(\dd)_{(1^n)}]\in \HO^0\!\!\Coha_{\Pi_Q,\dd}^{\mathcal{SSN}}$.
\end{proposition}
\begin{proof}
Let $\Lambda(\dd)_{(1^n)}^{\circ}\subset \Lambda(\dd)_{(1^n)}$ be the complement to the intersection with the union of components $\Lambda(\dd)_\pi$ for $\pi\neq (1^n)$.  Denote by 
\begin{align*}
j\colon&\Lambda(\dd)_{(1^n)}/\Gl_{\dd}\rightarrow \Mst_{\dd}(\ol{Q})\\
j^{\circ}\colon &\Lambda(\dd)^\circ_{(1^n)}/\Gl_{\dd}\rightarrow \Mst_{\dd}(\ol{Q})
\end{align*}
the inclusions.  We have $\chi_{\tilde{Q}}(\dd,\dd)=-2n^2$.  Define
\begin{align*}
\mathcal{G}&=j_*j^!\ul{\QQ}_{\Mst_{\dd}(\ol{Q})}\otimes\LLL^{-n^2}\\
\mathcal{G}^\circ&=j^\circ_*j^{\circ,!}\ul{\QQ}_{\Mst_{\dd}(\ol{Q})}\otimes\LLL^{-n^2}\cong j^\circ_*\ul{\QQ}_{\Lambda(\dd)^{\circ}_{(1^n)}/\Gl_{\dd}}
\end{align*}
where the isomorphism is due to the fact that $\Lambda(\dd)^{\circ}_{(1^n)}/\Gl_{\dd}$ is smooth and has codimension $n^2$ inside $\Mst_{\dd}(\ol{Q})$.  Then since $j$ is closed, and $\Lambda(\dd)^\circ_{(1^n)}$ is open in $\Lambda(\dd)_{(1^n)}$, we have a diagram
\begin{equation}
\label{comp_emb}
\xymatrix{
\JH_{\red,*}\mathcal{G}^{\circ}&\ar[l]_-\xi\JH_{\red,*}\mathcal{G}\ar[r]^-q& \JH_{\red,*}\HA^{\SSN}_{\Pi_Q,\dd}
}
\end{equation}
Applying $\HO^0$, $\xi$ is an isomorphism, and $[\Lambda(\dd)_{(1^n)}]$ is defined to be $\HO^0(q)(\HO^0(\xi))^{-1}(1)$.  Applying $\bm{\tau}_{\mathrm{con}}^{\leq 0}$, the truncation functor induced by the non-perverse t structure, the diagram \eqref{comp_emb} becomes
\[
\xymatrix{
\Delta_{\red,n,*}^{\mathcal{SSN}}\ul{\QQ}_{\AAA{1}}&\ar[l]_-{\cong}\Delta_{\red,n,*}^{\mathcal{SSN}}\ul{\QQ}_{\AAA{1}}\ar[r]& \bm{\tau}_{\mathrm{con}}^{\leq 0}\JH_{\red,*}\HA^{\SSN}_{\Pi_Q,\dd}.
}
\]
The element $\alpha_{i',n}$ is likewise obtained by applying $\HO$ to a homomorphism
\[
q'\colon \BPSh^{\SSN}_{\Pi_Q,\dd}\cong\Delta_{\red,n,*}^{\mathcal{SSN}}\ul{\QQ}_{\AAA{1}}\rightarrow \JH_{\red,*}\HA^{\SSN}_{\Pi_Q,\dd}.
\]
The domain is a mixed Hodge module, shifted by cohomological degree 1, so that $q'$ factors through the morphism 
\[
\bm{\tau}^{\leq 1}\JH_{\red,*}\HA^{\SSN}_{\Pi_Q,\dd}\rightarrow \JH_{\red,*}\HA^{\SSN}_{\Pi_Q,\dd}.
\]
By \eqref{gen} there is an isomorphism $\bm{\tau}^{\leq 1}\JH_{\red,*}\HA^{\SSN}_{\Pi_Q,\dd}\cong \Delta_{\red,n,*}^{\mathcal{SSN}}\ul{\QQ}_{\AAA{1}}$ and so we deduce that \\$\dim\left(\Hom\left(\Delta_{\red,n,*}^{\mathcal{SSN}}\ul{\QQ}_{\AAA{1}},\JH_{\red,*}\HA^{\SSN}_{\Pi_Q,\dd}\right)\right)=1$, and the proposition follows.
\end{proof}

\begin{remark}
Comparing with Theorem \ref{BoThm}, we see that the Lie algebra $\HO^0\!\mathfrak{g}^{\SSN}_{\Pi_Q}\subset \HO^0\!\!\Coha^{\SSN}_{\Pi_Q}$ is \textit{not} identified with the Lie sub-algebra $\mathfrak{n}_Q^+$ under the Bozec-Hennecart isomorphism (Theorem \ref{BoThm}).
\end{remark}

\subsubsection{Hyperbolic vertices}
Suppose that $i'\in I^{\hype}$, i.e. $i'$ supports $l$ edge-loops with $l\geq 2$.  Let $n\in \mathbb{Z}_{>0}$, and set $\dd=n\cdot 1_{i'}$.  The variety $\Msp_{\dd}(\Pi_Q)$ is an irreducible variety of dimension $2+2(l-1)n^2$ by \cite[Thm.1.3]{CB01}.  We set
\begin{align*}
\Cusp_{i',n}\coloneqq&\nIC_{\Msp_{\dd}(\Pi_Q)}\in\MHM(\Msp_{\dd}(\ol{Q}))\\
\Cusp^{\mathcal{SSN}}_{i',n}\coloneqq&\varpi'_{\red,*}\varpi'^!_{\red}\nIC_{\Msp_{\dd}(\Pi_Q)}\in\Db(\MHM(\Msp_{\dd}(\ol{Q})))\\
\fc^{\mathcal{SSN}}_{i',n}\coloneqq &\HO\!\left(\Msp_\dd(\ol{Q}),\Cusp^{\mathcal{SSN}}_{i',n}\right).
\end{align*}
By Lemma \ref{G_lemma}, $\Cusp^{\mathcal{SSN}}_{i',n}$ is pure.
The following will be proved as a special case of the results in \S \ref{genSec}.
\begin{proposition}
Set $\mathcal{B}_n=\BPSh_{\Pi_Q,n \cdot 1_{i'}}\in\MHM(\Msp_{\dd}(\Pi_Q))$.  There is an inclusion $\Cusp_{i',n}\hookrightarrow \mathcal{B}_n$ and $\Cusp_{i',n}$ is primitive, i.e. the induced morphism
\begin{equation}
\label{prim1}
\Cusp_{i',n}\rightarrow \mathcal{B}_n/\left(\sum_{n'+n''=n} \image\left(\mathcal{B}_{n'}\boxtimes_{\oplus}\mathcal{B}_{n''}\xrightarrow{[\cdot,\cdot]}  \mathcal{B}_{n}\right)\right)
\end{equation}
is injective, where $[\cdot,\cdot]$ is the Lie bracket of \eqref{relLA2}.
\end{proposition}
By Theorem \ref{purityThm}, \eqref{prim1} is a morphism of semisimple objects, as well as being injective, and so it has a left inverse.  Applying $\HO\varpi'_{\red,*}\varpi'^!_{\red}$ we deduce that there is an injective morphism
\begin{equation}
\label{cin}
\mathfrak{c}^{\mathcal{SSN}}_{i',n}\hookrightarrow \mathfrak{g}^{\mathcal{SSN}}_{\Pi_Q,n}
\end{equation}
and moreover that the induced morphism
\begin{equation}
\label{csnin}
\mathfrak{c}^{\mathcal{SSN}}_{i',n}\rightarrow \mathfrak{g}^{\mathcal{SSN}}_{\Pi_Q,n}/\left(\sum_{n'+n''=n}\image \left( \mathfrak{g}^{\mathcal{SSN}}_{\Pi_Q,n'}\otimes \mathfrak{g}^{\mathcal{SSN}}_{\Pi_Q,n''}\xrightarrow{[\cdot,\cdot]}\mathfrak{g}^{\mathcal{SSN}}_{\Pi_Q,n}\right) \right)
\end{equation}
is injective.  Taking the zeroth cohomologically graded piece, we deduce from Corollary \ref{absIso} that 
\begin{equation}
\label{dimest}
\dim(\HO^0\!\fc^{\mathcal{SSN}}_{i',n})\leq 1.
\end{equation}

%%%%%%%%
Since $\Cusp^{\mathcal{SSN}}_{i',n}$ is pure by Lemma \ref{G_lemma}, there is an isomorphism $\Cusp^{\mathcal{SSN}}_{i',n}\cong \Ho(\Cusp^{\mathcal{SSN}}_{i',n})$.  We claim that moreover there is an inclusion
\[
\ICS_{\Msp_{\dd}(Q^{\opp})}(\mathbb{Q})[1-\chi_Q(\dd,\dd)]\subset \Ho^{1-\chi_Q(\dd,\dd)}\Cusp^{\mathcal{SSN}}_{i',n}.
\]
This follows from the fact that $\Ho^{1-\chi_Q(\dd,\dd)}\Cusp^{\mathcal{SSN}}_{i',n}$ is semisimple (again by Lemma \ref{G_lemma}), and the fact that the LHS and the RHS are isomorphic after restriction to $\Msp^{\simp}_{\dd}(Q^{\opp})$, i.e. the LHS is the summand of the RHS with full support.  Since $\HO^0(\ICS_{\Msp_{\dd}(Q^{\opp})}(\mathbb{Q}))\cong\mathbb{Q}$ we deduce from \eqref{dimest} that
\begin{equation}
\label{hHyp}
\HO^0(\fc^{\mathcal{SSN}}_{i',n})\cong\mathbb{Q}.
\end{equation}
%%%%%%%%

\begin{corollary}
\label{cortan}
The images of the inclusions $\HO^0\!\left(\fc_{i',n}^{\SSN}\right)\hookrightarrow \HO^0\!\left(\fg^{\SSN}_{\Pi_Q,n\cdot 1_{i'}}\right)$ generate $\bigoplus_{n\geq 1}\fg_{\Pi_Q,n\cdot 1_{i'}}$.
\end{corollary}
\begin{proof}
The result follows from the fact that $\bigoplus_{n\geq 1}\fg_{\Pi_Q,n\cdot 1_{i'}}$ has one simple imaginary root for each $n$, and injectivity of \eqref{csnin}.
\end{proof}

Let $\alpha_{i',n}$ be a generator of $\HO^0(\ICS_{\Msp_{\dd}(Q^{\opp})}(\mathbb{Q}))$, i.e
\begin{equation}
\label{Hdef}
\HO^0(\ICS_{\Msp_{\dd}(Q^{\opp})}(\mathbb{Q}))\cong \alpha_{i',n}\cdot\mathbb{Q}.
\end{equation}
Post-composing the inclusion from Corollary \ref{cortan} with the morphisms
\[
\HO^0\!\left(\fg^{\SSN}_{\Pi_Q,n\cdot 1_{i'}}\right)\hookrightarrow \HO^0\!\left(\UEA(\fg^{\SSN}_{\Pi_Q,\NN\cdot 1_{i'}})\right)\hookrightarrow \HO^0\!\!\Coha^{\SSN}_{\Pi_Q}
\]
we obtain the inclusion $j\colon \HO^0(\ICS_{\Msp_{\dd}(Q^{\opp})}(\mathbb{Q}))\hookrightarrow \HCoha^{\SSN}_{\Pi_Q}$.  If we express $j(\alpha_{i',n})$ in terms of Bozec's basis, it is easy to see that the coefficient of $[\Lambda(\dd)_{(n)}]$ is 1.  The question of what all of the other coefficients are (in particular, whether they are nonzero) seems to be quite difficult without an explicit description of $\rCoha_{\Pi_Q,\dd}^{\SSN}$  like Proposition \ref{IsoPure} in the hyperbolic case.  On the other hand \S \ref{IsotSec} already demonstrates that the isomorphism $F$ from \eqref{UEAiso} does not identify $\mathfrak{n}_Q^+$ and $\HO^0\!(\fg_{\Pi_Q}^{\SSN})$.

\begin{remark}
We have shown that the zeroth cohomologically graded pieces of $\fc^{\SSN}_{i',n}$ for $i'\in Q_0$ and $n\in\ZZ_{\geq 1}$ provide a complete set of generators for $\HO^0\!(\fg_{\Pi_Q}^{\SSN})$.  This provides evidence for Conjecture \ref{mainConj} below. 
\end{remark}

%It seems very natural to conjecture that the condition on loops of $Q$ can be dropped from the statement of Theorem \ref{BBthm}, especially in light of \cite[Thm.3.34]{Bo16}.  Proving the conjecture amounts to finding classes $\beta_{i',n}\in\fg^{\mathcal{SSN}}_{\Pi_Q,n1_{i'}}$ for each $i'\in I^{\hype}$ which generate a free Lie subalgebra of $\fg^{\mathcal{SSN}}_{\Pi_Q}$.  This, in turn, would follow from \cite[Thm.6.2]{BSV17}, \eqref{kssn}, and the claim that each $\HO^0\left(\fg^{\mathcal{SSN}}_{\Pi_Q,\ZZ_{\geq 1}\cdot 1_{i'}}\right)$ for $i'\in I^{\hype}$ is a Borcherds algebra, which we leave for future work.
\section{Cuspidal cohomology}
\label{ccsec}
\subsection{Generators of $\fg^{\SP,G,\zeta}_{\Pi_Q,\theta}$}
\label{genSec}
In \S \ref{RLAsec} we constructed a lift of $\fg^{G,\zeta}_{\Pi_Q,\theta}$ to a Lie algebra object in the category $\MHM^G(\Msp^{\zeta\sst}_{\theta}(\Pi_Q))$.  In this section we will use this lift to produce canonical subspaces of generators for $\fg^{\SP,G,\zeta}_{\Pi_Q,\theta}$, which for reasons that will become clear we call ``cuspidal cohomology.''
\subsubsection{}
Fix a dimension vector $\dd$.  Recall that $\Msp^{G,\zeta\ssst}_{\dd}(\WT{Q})$ parameterises those $\CC\WT{Q}$-modules $\rho$ for which the underlying $\CC\ol{Q}$-module of $\rho$ is $\zeta$-semistable.  Let $U\subset \Msp^{G,\zeta\ssst}_{\dd}(\WT{Q})$ be the substack parameterising those modules for which the underlying $\CC\ol{Q}$-module is $\zeta$-stable, and define
\[
\mathfrak{U}=(\JH^{\circ})^{-1}(U)\subset \Mst^{G,\zeta\ssst}_{\dd}(\WT{Q})
\]
to be the open substack parameterising such modules.  Note that $U\subset \Msp^{G,\zeta\stab}_{\dd}(\WT{Q})$, so the morphism $\mathfrak{U}\rightarrow U$ is a $\CC^*$-gerbe.  Define
\[
V=U\cap \Msp_{\dd}^{G,\zeta\ssst}(\Jac(\WT{Q},\WT{W}));\quad\quad
\mathfrak{V}=(\JH^{\circ})^{-1}(V).
\]
Since $\sum_{i\in Q_0}\omega_i\in \Jac(\WT{Q},\WT{W})$ is central, it acts via scalar multiplication on any module represented by a point in $\mathfrak{V}$.   Arguing as in the proof of Proposition \ref{NQVDR} we deduce that
\begin{equation}
\label{Vdesc}
V= \left(\Msp_{\dd}^{\zeta\stab}(\Pi_Q)\times\AAA{1}\right)/G\subset \Msp_{\dd}^{G,\zeta\ssst}(\WT{Q}).
\end{equation}
  The variety $\Msp_{\dd}^{\zeta\stab}(\Pi_Q)$ is smooth, and so both $V$ and $\mathfrak{V}$ are smooth stacks.  We denote by $\overline{V}$ the closure of $V$ in $\Msp^{G,\zeta\ssst}(\WT{Q})$.  Similarly, arguing as in the proof of \eqref{rDIM} we deduce that $\phim{\TTTr(\WT{W})}\nnIC_{\mathfrak{U}}\cong \nnIC_{\mathfrak{V}}$, and thus
\begin{equation}
\label{Vin}
(U\hookrightarrow \Msp^{G,\zeta\sst}_{\dd}(\WT{Q}))^*\JH^G_*\phim{\TTTr(\WT{W})}\nnIC_{\Mst^{G,\zeta\sst}_{\dd}(\WT{Q})}\cong \nnIC_{V}\otimes\HO(\B\CC^*,\QQ)_{\vir}.
\end{equation}
By Corollary \ref{relPurity} the object  $\JH^G_*\phim{\TTTr(\WT{W})}\nnIC_{\Mst^{G,\zeta\sst}_{\dd}(\WT{Q})}$ is pure, and so in particular its first cohomology is a semisimple monodromic mixed Hodge module.  From \eqref{Vin} and the inclusion $\LLL^{1/2}\hookrightarrow \HO(\B\CC^*,\QQ)_{\vir}$ we deduce that there is a canonical morphism
\begin{equation}
\label{PreV}
\Gamma\colon\nnIC_{\ol{V}}\otimes\LLL^{1/2}\rightarrow \JH^G_*\phim{\TTTr(\WT{W})}\nnIC_{\Mst^{G,\zeta\sst}_{\dd}(\WT{Q})}
\end{equation}
for which there is a left inverse $\alpha$, by purity of the target.  By \eqref{Vdesc} we have $\ol{V}=\left(\Msp_{\dd}^{\zeta\sst}(\Pi_Q)\times\AAA{1}\right)/G$,
and so $r'_*\nnIC_{\ol{V}}\otimes\LLL^{1/2}\cong \nnIC_{\Msp_{\dd}^{G,\zeta\sst}(\Pi_Q)}$.

Set 
\begin{align}
\nonumber
\Cusp_{\Pi_Q,\dd}^{\SP,G,\zeta}&\coloneqq \begin{cases} \varpi'_{\red,*}\varpi'^!_{\red}\nnIC_{\Msp^{G,\zeta\sst}_{\dd}(\Pi_Q)}& \textrm{if }\Msp^{\zeta\stab}_{\dd}(\Pi_Q)\neq \emptyset\\
0&\textrm{otherwise}\end{cases}\\
\label{fcdef}
\fc^{\SP,G,\zeta}_{\Pi_Q,\dd}&\coloneqq \HO\!\Cusp_{\Pi_Q,\dd}^{\SP,G,\zeta}.
\end{align}
Then by base change there is an isomorphism
\begin{align*}
\fc^{\SP,G,\zeta}_{\Pi_Q,\dd}\cong &\HO\!\left(\Msp^{G,\zeta\sst}_{\dd}(\WT{Q}),\varpi'_*\varpi'^!\nnIC_{\ol{V}}\otimes\LLL^{1/2}\right).
\end{align*}

Applying $\HO\varpi'_*\varpi'^!$ to \eqref{PreV} we obtain a morphism
\[
\beta\colon\fc^{\SP,G,\zeta}_{\Pi_Q,\dd}\hookrightarrow\fg^{\WT{\SP},G,\zeta}_{\WT{Q},\WT{W},\dd}=\mathfrak{P}_{1}\!\HCoha_{\WT{Q},\WT{W},\dd}^{\WT{\SP},G,\zeta}
\]
which is an injection, since it has a left inverse (e.g. $\HO\varpi_*\varpi^!\bm{\tau}^{\leq 1}\alpha$).
\sssct
We can now prove (a generalisation of) Theorem \ref{mainThmB}.

\begin{theorem}
\label{thmBdone}
Let $\zeta\in\QQ^{Q_0}$ be a stability condition, and let $\theta\in\QQ$ be a slope.  For each $\dd\in\dvst$ there is a canonical decomposition
\[
\fg_{\Pi_Q,\dd}^{\SP,G,\zeta}\cong \fc^{\SP,G,\zeta}_{\Pi_Q,\dd}\oplus \mathfrak{l}
\]
for some mixed Hodge structure $\mathfrak{l}$, with $\fc^{\SP,G,\zeta}_{\Pi_Q,\dd}$ as in \eqref{fcdef}, such that for $\dd',\dd''\in \dvst$ such that $\dd'+\dd''=\dd$, the morphism $\fg_{\Pi_Q,\dd'}^{\SP,G,\zeta}\otimes\fg_{\Pi_Q,\dd''}^{\SP,G,\zeta}\xrightarrow{[\cdot,\cdot]}\fg_{\Pi_Q,\dd}^{\SP,G,\zeta}$ factors through the inclusion of $\mathfrak{l}$.
\end{theorem}
\begin{proof}
We assume that $\Msp^{\zeta\stab}_{\dd}(\Pi_Q)\neq \emptyset$, as otherwise $\fc^{\SP,G,\zeta}_{\Pi_Q,\dd}=0$ and the statement is trivial.  Recall that by \eqref{2DIR} there is an isomorphism
\[
r'_*\JH^{\circ}_*\phim{\TTTr(\WT{W})}\nnIC_{\Mst^{G,\zeta\ssst}_{\dd}(\WT{Q})}\cong \JH^G_{\red,*}\iota_*\iota^!\ul{\QQ}_{\Mst^{G,\zeta\sst}_{\dd}(\ol{Q})}\otimes\LLL^{\chi_{\WT{Q}}(\dd,\dd)/2}=\rCoha_{\Pi_Q,\dd}^{G,\zeta}
\]
and by Corollary \ref{relPurity} these are pure complexes of monodromic mixed Hodge modules, so there is a decomposition
\begin{equation}
\label{2dDecomp}
\rCoha_{\Pi_Q,\dd}^{G,\zeta}\cong\bigoplus_{r\in R_n}\mathcal{F}_r[n]
\end{equation}
where each $\mathcal{F}_r$ is a simple mixed Hodge module, and each $R_n$ is some indexing set.

The stack $\Mst^{G,\zeta\stab}_{\dd}(\Pi_Q)$ is smooth, of codimension $\dd\cdot\dd-1$ inside $\Mst^{G,\zeta\stab}_{\dd}(\ol{Q})$, and so there is an isomorphism
\[
(\Mst^{G,\zeta\stab}_{\dd}(\Pi_Q)\hookrightarrow \Mst^{G,\zeta\sst}_{\dd}(\Pi_Q))^*\iota_*\iota^!\ul{\QQ}_{\Mst^{G,\zeta\sst}_{\dd}(\ol{Q})}\cong\ul{\QQ}_{\Mst^{G,\zeta\stab}_{\dd}(\Pi_Q)}\otimes\LLL^{\dd\cdot\dd-1}
\]
and thus an isomorphism
\begin{align*}
&(\Msp^{G,\zeta\stab}_{\dd}(\Pi_Q)\hookrightarrow \Msp^{G,\zeta\sst}_{\dd}(\Pi_Q))^*\rCoha_{\Pi_Q,\dd}^{G,\zeta}\\&\cong(\Mst^{G,\zeta\stab}_{\dd}(\Pi_Q)\rightarrow \Msp^{G,\zeta\stab}_{\dd}(\Pi_Q))_*\ul{\QQ}_{\Mst^{G,\zeta\stab}_{\dd}(\Pi_Q)}\otimes\LLL^{\chi_{\WT{Q}}(\dd,\dd)/2+\dd\cdot\dd-1}.
\end{align*}
Noting that $\Msp^{G,\zeta\stab}_{\dd}(\Pi_Q)$ is open inside $\Msp^{G,\zeta\sst}_{\dd}(\Pi_Q)$, of dimension $\chi_Q(\dd,\dd)-1=\chi_{\WT{Q}}(\dd,\dd)/2+\dd\cdot\dd-1$, we deduce that in the decomposition \eqref{2dDecomp}, in cohomological degree zero there is exactly one copy of the simple object $\nnIC_{\Msp^{G,\zeta\sst}_{\dd}(\Pi_Q)}$, and furthermore the morphism 
\[
r'_*\Gamma\colon\nnIC_{\Msp^{G,\zeta\sst}_{\dd}(\Pi_Q)}\rightarrow \rCoha_{\Pi_Q,\dd}^{G,\zeta}
\]
is the inclusion of this object.  Writing 
\begin{align*}
\Ho^{G,0}\!\left(\rCoha_{\Pi_Q,\dd}^{G,\zeta}\right)=&\nnIC_{\Msp^{G,\zeta\sst}_{\dd}(\Pi_Q)}\oplus \mathcal{G}\\
\mathcal{G}=&\bigoplus_{r\in R_0'}\mathcal{F}_r
\end{align*}
for $R'_0\subset R_0$, we claim that for all $\dd',\dd''\in\dvst$ with $\dd'\neq 0\neq \dd''$ and $\dd'+\dd''=\dd$ the multiplication
\begin{equation}
\label{alm}
\Ho^{G,0}\!\left(\rCoha^{G,\zeta}_{\Pi_Q,\dd'}\right)\boxtimes_{\oplus^G}\Ho^{G,0}\!\left(\rCoha^{G,\zeta}_{\Pi_Q,\dd''}\right)\rightarrow \Ho^{G,0}\!\left(\rCoha^{G,\zeta}_{\Pi_Q,\dd}\right)
\end{equation}
factors through the inclusion of $\mathcal{G}$.  This follows for support reasons: by our assumptions on $\dd',\dd''$ the supports of the semisimple object on the left hand side of \eqref{alm} are all in the boundary $\Msp^{G,\zeta\sst}_{\dd}(\Pi_Q)\setminus \Msp^{G,\zeta\stab}_{\dd}(\Pi_Q)$.  Applying $\HO\varpi'_{\red,*}\varpi'^!_{\red}$, there is a decomposition
\[
\lP_{0}\!\HO\!\Coha^{\Sp,G,,\zeta}_{\Pi_Q,\dd}\cong \fc^{\SP,G,\zeta}_{\Pi_Q,\dd}\oplus\HO\!\varpi'_{\red,*}\varpi'^!_{\red}\mathcal{G}
\]
and the multiplication 
\[
\lP_{0}\!\HO\!\!\Coha^{\Sp,G,,\zeta}_{\Pi_Q,\dd'}\otimes_{\HG}\lP_{0}\!\HO\!\!\Coha^{\Sp,G,,\zeta}_{\Pi_Q,\dd''}\rightarrow \lP_{0}\!\HO\!\!\Coha^{\Sp,G,,\zeta}_{\Pi_Q,\dd}
\]
factors through the inclusion of $\HO\!\varpi'_{\red,*}\varpi'^!_{\red}\mathcal{G}\eqqcolon\mathfrak{l}$, so that the commutator Lie bracket also factors through the inclusion of $\mathfrak{l}$.
\end{proof}

\subsection{The BPS Lie algebra $\mathfrak{g}_{\Pi_Q}$ for $Q$ affine}\label{AffineSec}
Let $Q$ be a quiver for which the underlying graph is an affine Dynkin diagram of extended ADE type.  We can use the fact that $\fg^G_{\Pi_Q}$ lifts to a Lie algebra object $\BPSh^G_{\Pi_Q}$ in $\MHM^G(\Msp(\Pi_Q))$ to calculate it completely.  

We denote by $q_{\dd}\colon \Msp^{\zeta\sst}_{\dd}(\Pi_Q)\rightarrow \Msp_{\dd}(\Pi_Q)$ the affinization map.  Let $\delta\in\dvs$ be the unique primitive imaginary simple root of the quiver $Q$.  Let $H\subset \Sl_2(\CC)$ be the Kleinian group corresponding to the underlying (finite) Dynkin diagram of $Q$ (obtained by removing a single vertex) via the McKay correspondence.  Then (see \cite{Kr89,CaSl98}) for a generic stability condition $\zeta\in\QQ^{Q_0}$ there is a commutative diagram
\[
\xymatrix{
X\ar[d]^p\ar[r]^-{\cong}&\Msp_{\delta}^{\zeta\sst}(\Pi_Q)\ar[d]^{q_{\delta}}\\
Y\ar[r]^-{\cong}&\Msp_{\delta}(\Pi_Q)
}
\]
where $p\colon X\rightarrow Y$ is the minimal resolution of the singularity $Y=\AAA{2}/H$.  Moreover by \cite{KV00} there is a derived equivalence
\[
\Psi\colon \Db(\Coh(X))\rightarrow \Db(\Pi_Q\lmod)
\]
restricting to an equivalence between complexes of modules with nilpotent cohomology sheaves and complexes of coherent sheaves with set-theoretic support on the exceptional locus of $p$.  For $\dd\in\dvs$ we denote by $0_{\dd}\in\Msp_{\dd}(\Pi_Q)$ the point corresponding to the unique semisimple nilpotent module of dimension vector $\dd$.

Set $r=\lvert Q_0\lvert -1$.  Via the explicit description of the representations of $KQ$ for $Q$ an affine quiver, we have the following identities
\begin{equation}
\label{affineKac}
\kac_{Q,\dd}(t)=\begin{cases} 1 &\textrm{if }\dd \textrm{ is a positive real root of }\mathfrak{g}_Q\\
t+r&\textrm{if } \dd\in\mathbb{Z}_{\geq 1}\cdot\delta\\
0&\textrm{otherwise.}\end{cases}
\end{equation}

\begin{proposition}
\label{affineProp}
There are isomorphisms
\begin{equation}
\label{affineBPS}
\BPSh_{\Pi_Q,\dd}\cong\begin{cases} \ul{\QQ}_{0_{\dd}}&\textrm{if }\dd \textrm{ is a positive real root of }\mathfrak{g}_Q\\
\Delta_{n,*}q_{\delta,*}\nIC_{\Msp_{\delta}^{\zeta\sst}(\Pi_Q)}&\textrm{if } \dd=n\cdot\delta\\
0&\textrm{otherwise,}\end{cases}
\end{equation}
where $\Delta_n\colon \Msp_{\delta}(\Pi_Q)\rightarrow \Msp_{n\cdot\delta}(\Pi_Q)$ is the embedding of the small diagonal.
\end{proposition}
\begin{proof}
By Proposition \ref{TodaProp} there is an isomorphism
\[
\BPSh_{\Pi_Q,\dd}\cong q_{\dd,*}\BPSh^{\zeta}_{\Pi_Q,\dd}.
\]
On the other hand, any complex of compactly supported coherent sheaves $\mathscr{F}$ on $X$ that is not entirely supported on the exceptional locus admits a direct sum decomposition $\mathscr{F}'\oplus\mathscr{F}''$ where $\mathscr{F}''$ is supported at a single point, so that $\Psi(\mathscr{F})$ admits a direct summand $N$ with dimension vector a multiple of $\delta$.  It follows that all points of $\Msp_{\dd}^{\zeta\sst}(\Pi_Q)$ correspond to nilpotent modules if $\dd$ is not a multiple of $\delta$, and so $q_{\dd,*}\BPSh^{\zeta}_{\Pi_Q,\dd}$ is supported at the origin.  Since by Theorem \ref{purityThm} $\BPSh_{\Pi_Q,\dd}$ is pure, and supported at a single point, it is determined by its hypercohomology $\fg_{\Pi_Q,\dd}$.  This hypercohomology is pure, of Tate type, with dimension given by the Kac polynomial, by the main result of \cite{preproj}.  This deals with the first and last cases of \eqref{affineBPS}.

For the second case, we consider the commutative diagram
\[
\xymatrix{
\CCoh_n(X)\ar[r]^h_{\cong}\ar[d]^g&\Mst_{n\cdot\delta}^{\zeta\sst}(\Pi_Q)\ar[d]^{\JH_{\red}}
\\
\Sym_n(X)\ar[r]^l_{\cong}&\Msp_{n\cdot\delta}^{\zeta\sst}(\Pi_Q).
}
\]
By Theorem \ref{purityThm}, $g_*h^*\iota^!\ul{\QQ}_{\Mst_{n\cdot \delta}^{\zeta\sst}(\ol{Q})}$
is pure, and we claim that it contains a single copy of $\Delta_{X,n,*}\nIC_{X}$.  Since $X$ is simply connected, we can cover $X$ by charts $U_i$ isomorphic to $\AAA{2}$ and check the claim on each of the open subvarieties $U_i$, at which point the claim follows by \eqref{2dcBPS}.  Finally, the BPS sheaf is supported on the small diagonal by the support lemma of \cite{preproj}, so $\Delta_{X,n,*}\nIC_X\cong l^*\BPSh_{\Pi_Q,n\cdot\delta}^{\zeta}$.  Then the second case follows by Proposition \ref{TodaProp}.
\end{proof}

Note that there is an isomorphism 
\begin{equation}
\label{affroo}
q_{\delta,*}\nIC_{\Msp_{\delta}^{\zeta\sst}(\Pi_Q)}\cong \nIC_{\Msp_{\delta}(\Pi_Q)}\oplus \ul{\QQ}_{0_{\delta}}^{\oplus d}
\end{equation}
since there are $r$ copies of $\mathbb{P}^1$ in the exceptional fibre of $q_{\delta}$.  We deduce from \eqref{affineKac} and \eqref{Kaccha} that $\HO^{-2}\fg_{\Pi_Q,n\cdot \delta}\cong\QQ$ is obtained by applying $\HO\Delta_{n,*}$ to the first summand of \eqref{affroo}.
\begin{proposition}
There is an isomorphism of Lie algebras
\begin{equation}
\label{aff}
\mathfrak{g}_{\Pi_Q}\cong \mathfrak{n}^-_{Q'}\oplus s\QQ[s]
\end{equation}
where $Q'$ is the real subquiver of $Q$ (i.e. it is equal to $Q$ unless $Q$ is the Jordan quiver, in which case it is empty) and $s\QQ[s]$ is given the trivial Lie bracket.  The monomial $s^n$ lives in $\dvs$-degree $n\cdot\delta$, and in cohomological degree $-2$.
\end{proposition}
\begin{proof}
By \eqref{affineKac} and \eqref{Kaccha}, the graded dimensions of the two sides of \eqref{aff} match.  Furthermore, by Theorem \ref{KMLA} there is an isomorphism of Lie algebras between the zeroth cohomologically graded pieces of the RHS and LHS of \eqref{aff}.  So it is sufficient to prove that $\HO^{-2}(\mathfrak{g}_{\Pi_Q})$ is central, which (for cohomological degree reasons) amounts to showing that $[s^n,b]=0\in\fg_{\Pi_Q,\beta}$ for $b\in \HO^0(\fg_{\Pi_Q,\alpha})$ where $\alpha=n'\cdot\delta$ and $\beta=n\cdot\delta+\alpha$.  On the other hand, lifting the Lie bracket to $\MHM(\Msp(\Pi_Q))$, the morphism
\[
\QQ\otimes \QQ\xrightarrow{\cdot s^n\otimes\cdot b}\fg_{\Pi_Q,n\delta}\otimes \fg_{\Pi_Q,\alpha}\xrightarrow{[\cdot,\cdot]}\fg_{\Pi_Q,\beta}
\]
is obtained by applying $\HO$ to the morphism of mixed Hodge modules
\[
\Delta_{n,*}\nIC_{\Msp_{\delta}(\Pi_Q)}\boxtimes_{\oplus_{\red}}\QQ_{0_{\alpha}}\rightarrow \Delta_{(n+n'),*}\nIC_{\Msp_{\delta}(\Pi_Q)}\oplus\ul{\QQ}_{0_{\beta}}^{\oplus d}
\]
which is a morphism between semisimple mixed Hodge modules with differing supports, and is thus zero.
\end{proof}
\subsubsection{A deformed example}
\label{defEx}
In this subsection we give a curious example, which will not be used later in the paper.  It is an example of how deforming the potential can modify the BPS Lie algebra.  

Let $Q$ be the oriented $\hat{A}_r$ quiver, i.e. it contains $r+1$ vertices, along with an oriented cycle connecting them all.  Let $W_0=a_{r+1}a_{r}\cdots a_1$ be this cycle.  We will consider the quiver with potential $(\tilde{Q},\WT{W}+W_0)$.  The potential $\WT{W}+W_0$ is quasihomogeneous, for example we can give the arrows  $a_s$ weight 1, the arrows $a_s^*$ weight $d$, and the arrows $\omega_i$ weight zero, so that $\WT{W}+W_0$ has weight $r+1$.

As in \S \ref{DDRsec} we can calculate $\BPSh_{\Pi_Q,W_0}$ by applying $\phim{\TTr(W_0)}$ to $\BPSh_{\Pi_Q}$.  For $\dd$ not a multiple of the imaginary simple root, $\BPSh_{\Pi_Q,\dd}$ is supported at $0_{\dd}$ and so it follows that
\[
\BPSh_{\Pi_Q,W_0,\dd}\cong \phim{\TTr(W_0)}\BPSh_{\Pi_Q,\dd} \cong \BPSh_{\Pi_Q,\dd}.
\]
In particular, it follows that there is an injective map $l$ from the Lie subalgebra of $\BPSh_{\Pi_Q}$ generated by $\bigoplus_{i\in Q_0}\BPSh_{\Pi_Q,1_i}$, and the only dimension vectors for which this morphism can fail to be an isomorphism are dimension vectors $\ee=(n,\ldots,n)$ for some $n$.

Let $\ee$ be such a dimension vector.  Propositions \ref{TodaProp} and \ref{affineProp} together yield
\[
\BPSh_{\Pi_Q,W_0,\ee}\cong \Delta_{n,*}q_{(1,\ldots,1),*}\phim{g}\nIC_X
\]
where $X$ is the minimal resolution of the singular surface defined by $xy=z^r$, and $g=y$ is the function induced on it by $\TTr(W_0)$.  The reduced vanishing locus $X_0=g^{-1}(0)$ is given by the exceptional chain of $r$ copies of $\mathbb{P}^1$, along with a line intersecting one of them transversally.  In particular, the cohomology of $X_0$ is pure.  The preimage $X_1\coloneqq g^{-1}(1)$ is isomorphic to a copy of $\AAA{1}$.  Via the long exact sequence
\[
\rightarrow \HO^i(X,\phim{g}\ul{\QQ}_X)\rightarrow \HO^i(X_0,\QQ)\rightarrow \HO^i(X_1,\QQ)\rightarrow
\]
we deduce that there is an isomorphism $\HO(X,\phim{g}\ul{\QQ}_X)\cong \HO^2(X,\QQ)[-2]$, i.e. the vanishing cycle cohomology is isomorphic to the \textit{reduced} cohomology of $X$.  It follows by counting dimensions that the injective map $\HO(l_{\ee})$ is surjective.  We deduce that
\begin{equation}
\label{KacOut}
\fg_{\Pi_Q,W_0}\cong \mathfrak{n}^-_Q
\end{equation}
i.e. the BPS Lie algebra for the deformed potential is isomorphic to negative half of the usual Kac--Moody Lie algebra for $Q$.  

It is an interesting question whether for more general quivers there is a quasihomogeneous deformation $\WT{W}+W_0$ of the standard cubic potential so that \eqref{KacOut} holds.  A related question is: does the nonzero degree cohomology of the BPS Lie algebra $\fg_{\Pi_Q,W_0}$ vanish for generic (quasihomogeneous) deformations $W_0$?  In other words, is the BPS Lie algebra for a generic deformed 3-Calabi--Yau completion \cite{Kel11} of the preprojective algebra $\Pi_Q$ equal to the canonical Lie subalgebra $\fn_Q^-$?
\subsection{The spherical Borcherds algebra}
In this section we construct a natural homomorphism of cohomologically graded Lie algebras $\Phi\colon\fg^{\exte,\SP}_{\Pi_Q}\rightarrow \fg^{\SP}_{\Pi_Q}$ from the positive half of a Borcherds algebra, extending the inclusion of the Kac--Moody Lie algebra from \S \ref{KacH}.  In the case in which $\SP=\SSN$ the zeroth cohomologically graded piece of this morphism is the inclusion of the Borcherds--Bozec algebra.  The existence of the morphism $\Phi$ serves as further evidence towards Conjecture \ref{mainConj}, which we will finish the paper with.
%a conjecture of Bozec and Schiffmann \cite[Conj.1.3]{BoSch19}, which concerns the case $\SP=\CC\ol{Q}\lmod$, and states that the \textit{whole} of  $\fg_{\Pi_Q}$ is the positive half of a Borcherds algebra\footnote{More accurately, their conjecture is equivalent to the statement that the graded dimensions of $\fg_{\Pi_Q}$ are equal to those of \textit{some} Borcherds algebra.}.
\subsubsection{}
We introduce a little notation, in order to make the presentation uniform.  Given a tensor category $\mathscr{C}$ we denote by $\mathscr{C}_{\dvs}$ the category of $\dvs$-graded objects in $\mathscr{C}$.  Given $F\in\mathscr{C}_{\mathbb{N}^{Q_0}}$ we denote by $\Lie(F)$ the free Lie algebra generated by $F$.  I.e. we consider an (essentially unique) symmetric monoidal embedding $\Vect\hookrightarrow \mathscr{C}$, and embed $\mathscr{C}\hookrightarrow \mathscr{C}_{\dvs}$ as the category of objects concentrated in degree zero, and thus consider $\mathcal{L}ie$ as an operad in $\mathscr{C}_{\dvs}$, and take the free algebra over it generated by $F$.  We define $\Borch^+(F)$ to be the quotient of $\Lie(F)$ by the Lie ideal generated by the images of the morphisms
\begin{align}
\label{BorchRel}
(F_{\dd'}^{\otimes})^{1-(\dd',\dd'')_Q} \otimes F_{\dd''}\xrightarrow{[\cdot,[\cdot\ldots[\cdot,\cdot]\ldots]} F
\end{align}
over all pairs of dimension vectors $\dd', \dd''$ satisfying either of the conditions $(\dd',\dd'')_Q=0$ or $\dd'=1_i$ for $i\in Q'_0$.  We assume that $F_{\dd}=0$ if $\dd$ is not of the form $n\cdot \dd'$ for some $\dd'$ such that there exists a simple $\dd'$-dimensional $\Pi_Q$-module: as a result, the exponent $1-(\dd',\dd'')_Q$ can always be assumed to be strictly positive (see \cite[Prop.3.6]{Dav21c}).
\begin{example}
Consider the vector space $V\in\Vect_{\dvs}$ which has basis $e_i$ for $i\in I_{\infty}$, where $e_{(i',n)}$ is given degree $n\cdotsh 1_{i'}$.  Then $\mathfrak{n}^+_Q=\Borch^+(V)$.
\end{example}
For $i\in Q_0$ and $n\geq 1$ we denote by $\Delta_{i,n}\colon \Msp^{G}_{1_i}(\ol{Q})\rightarrow \Msp^{G}_{n\cdot 1_i}(\ol{Q})$ the embedding of the small diagonal.

\begin{proposition}
Let $\SP$ be a Serre subcategory of $\CC\ol{Q}\lmod$.  Set
\begin{align*}
\Prim^{G}_{\Pi_Q,\sph}\coloneqq &\bigoplus_{i\in Q_0^{\reel}}\nnIC_{\Msp^{G}_{1_i}(\ol{Q})}\oplus \bigoplus_{\substack{i\in Q_0^{\isot}\\n\geq 1}}\Delta_{i,n,*}\nnIC_{\Msp^{G}_{1_i}(\ol{Q})}\oplus\bigoplus_{\substack{i\in Q_0^{\hype}\\n\geq 1}}\Cusp^{G}_{\Pi_Q,n\cdot 1_{i}}\\
\prim^{\SP,G}_{\Pi_Q,\sph}\coloneqq &\HO\!\varpi'_{\red,*}\varpi'^!_{\red}\Prim^{G}_{\Pi_Q,\sph}.
\end{align*}
There are morphisms of Lie algebra objects 
\[
J:\Borch^+\!\left(\Prim^{G}_{\Pi_Q,\sph}\right)\rightarrow \BPSh_{\Pi_Q}^{G};\quad\quad
L^{\SP}\colon \Borch^+\!\left(\prim^{\SP,G}_{\Pi_Q,\sph}\right)\rightarrow \fg^{\SP,G}_{\Pi_Q}.
\]
extending embeddings of $\Prim^{G}_{\Pi_Q,\sph}$ and $\prim^{\SP,G}_{\Pi_Q,\sph}$, respectively, where in the second morphism, $\prim^{\SP,G}_{\Pi_Q,\sph}$ is considered as an object of $\HG\lmod$.
\end{proposition}
\begin{proof}
The morphism $L^{\SP}$ is obtained as $\HO\!\varpi'_{\red,*}\varpi'^!_{\red}J$, so we concentrate on $J$.  Firstly, note that for $i\in Q_0$ there is an equality $\BPSh^G_{\Pi_Q,1_i}=\nnIC_{\Msp^G_{1_i}(\ol{Q})}$ so that the first summand of $\Prim^{G}_{\Pi_Q,\sph}$ naturally embeds inside $\BPSh^G_{\Pi_Q}$.  Secondly, as in Proposition \ref{IsoPure} there is an embedding (unique up to scalar) of $\Delta_{i,n,*}\nnIC_{\Msp^{G}_{1_i}(\ol{Q})}$ inside $\BPSh_{\Pi_Q,n\cdot 1_i}^G$ for each isotropic $i$.  Thirdly, for $i$ hyperbolic the morphism \eqref{PreV} provides an embedding $\Cusp^{G}_{\Pi_Q,n\cdot 1_i}\subset \BPSh_{\Pi_Q}^{G}$.  We claim that these embeddings induce the morphism $J$.  

To prove the claim, we need to check the relation \eqref{BorchRel}.  Note that if $i$ and $j$ are both real, this follows immediately from Proposition \ref{BPSvanProp}.  Otherwise, we need something a little more subtle, i.e. the decomposition theorem.

Let $i\in I^{\reel}$, let $(j,n)\in I^{\Imag}$, and set $e=1-((i,1),(j,n))$.  Set 
\[
\Msp_{i}=\Msp^{G}_{1_i}(\Pi_Q);\quad\quad
\Msp_{j,n}=\begin{cases}\Delta_{j,n}(\Msp_j)&\textrm{if }j\in Q_0^{\isot}\\ \Msp^{G}_{n\cdot 1_j}(\Pi_Q)&\textrm{if }j\in Q_0^{\hype}.\end{cases}
\]
Then we wish to show that the morphism
\[
J'\colon \underbrace{\nnIC_{\Msp_{i}}\boxtimes_{\oplus^G}\cdots\boxtimes_{\oplus^G}\nnIC_{\Msp_{i}}}_{e \textrm{ times}}\boxtimes_{\oplus^G}\nnIC_{\Msp_{j,n}}\rightarrow \BPSh_{\Pi_Q}^{G}
\]
given by the iterated Lie bracket (as in \eqref{BorchRel}) is the zero morphism.  For this, we note that the morphism provided by taking direct sums
\[
h\colon \underbrace{\Msp_{i}\times_{\B G} \cdots\times_{\B G}\Msp_i}_{e \textrm{ times}} \times_{\B G} \Msp_{j,n}\rightarrow\Msp^G_{e\cdot 1_i+\dd}(\Pi_Q)
\]
is injective, and so a closed embedding, and so since $\nnIC_{\Msp_{i}}$ and $\nnIC_{\Msp_{j,n}}$ are simple, the domain of $J'$ is a simple object.  We denote the domain of $J'$ by $\mathcal{R}$.  Since by Theorem \ref{purityThm} the target of $J'$ is semisimple we deduce that $J'$ is nonzero only if there is a direct sum decomposition $\BPSh_{\Pi_Q}^{G}\cong \mathcal{R}\oplus\mathcal{G}$ and $J'$ fits into a commutative diagram
\[
\xymatrix{
\ar[rd]_{\iota_{\mathcal{R}}}\mathcal{R}\ar[r]^-{J'}& \BPSh_{\Pi_Q}^{G}\\
&\mathcal{R}\oplus\mathcal{G}\ar[u]^{\cong}
}
\]
where $\iota_{\mathcal{R}}$ is the canonical inclusion.  

For a contradiction, we assume that this is indeed so.  Now we apply $\HO^0\!g_{*}g^!$, where 
\[
g\colon \Msp^{\SSN,G}(\ol{Q})\rightarrow \Msp^G(\ol{Q})
\]
is the inclusion of the strictly semi-nilpotent locus.  By \eqref{IsotDone} in the case of isotropic $j$, and \eqref{hHyp} in the hyperbolic case, there is an isomorphism $\HO^0\!g_{*}g^!\nnIC_{\Msp_{j,n}}\cong \QQ$ and so, since $\HO^0(\nnIC_{\Msp_i})\cong \QQ$ we deduce that $\HO^0\!g_{*}g^!\mathcal{R}\cong\QQ$, so that $\HO^0\!g_{*}g^!\iota_{\mathcal{R}}=\iota_{\QQ}\neq 0$.  On the other hand, $\HO^0\!\!g_{*}g^!J'\colon \QQ\rightarrow \fg^{\SSN}_{\Pi_Q}$ is the morphism taking $1\in\QQ$ to $[\alpha_i,\cdot]^e(\alpha_{j,n})$, with $\alpha_i$ and $\alpha_{j,n}$ defined as in \eqref{Fdef}, \eqref{Gdef}, \eqref{Hdef}.  By Corollary \ref{absIso} we have $[\alpha_i,\cdot]^e(\alpha_{j,n})=0$, giving the contradiction.

Now assume that both $i$ and $j$ are imaginary, and $(1_i,1_j)_Q=0$.  Fix $m\geq 1$.  Then, similarly to above, we wish to show that the morphism
\[
\nnIC_{\Msp_{i,m}}\boxtimes_{\oplus^G}\nnIC_{\Msp_{j,n}}\rightarrow \BPSh^G_{\Pi_Q}
\]
provided by the Lie bracket in $\BPSh^G_{\Pi_Q}$ is the zero map.  Again, applying $\HO^0\!g_*g^!$ this follows from Corollary \ref{absIso} and injectivity of the morphism $\oplus^G\colon \Msp_{i,m}\times_{\B G} \Msp_{j,n}\rightarrow \Msp_{m\cdot 1_i+n\cdot 1_j}^G(\ol{Q})$.  We thus have defined the morphism $J$.  
\end{proof}

It is of course very natural to make the following
\begin{conjecture}
The morphisms $J$ and $L^{\SP}$ are injective.
\end{conjecture}
The results of \S \ref{genSec} imply the conjecture in case there are no hyperbolic vertices.  In contrast with the Hall algebra $\HCoha_{\Pi_Q}^{\SSN}$, which is generated by the subspaces $\HCoha_{\Pi_Q,n\cdot 1_i}^{\SSN}$ \cite[Prop.5.8]{ScVa20}, the Lie algebra $\fg^{\SP}_{\Pi_Q}$ is almost never generated by the subspaces $\fg^{\SP}_{\Pi_Q,n\cdot 1_i}$, so that $L^{\Sp}$ is almost never surjective.  For instance, if $Q$ has no edge loops, then the image of $L$ lies entirely in cohomological degree zero, while unless $Q$ is of finite type, $\fg_{\Pi_Q}$ will have pieces in strictly negative cohomological degree.  (Recall that by convention $L=L^{\CC\ol{Q}\lmod}$, and in this case we have $\chi_t(\fg_{\Pi_Q})=\kac_{Q,\dd}(t^{-1})$.)

\subsection{The main conjecture for BPS Lie algebras}
\label{conjecturesSec}
We finish the paper with our main conjecture regarding the structure of BPS Lie algebras for preprojective CoHAs.  Put informally, the conjecture states that we have found \textit{all} of the generators of $\fg_{\Pi_G}^{\SP}$, which we moreover conjecture is a Borcherds algebra.  To state the conjecture fully, we make the following definitions.  First fix a Serre subcategory $\SP$, a stability condition $\zeta\in\QQ^{Q_0}$, and $\theta\in\QQ$.  Set 
\[
E=\{\dd\in\dvst \colon \Msp^{\zeta\stab}_{\dd}(\Pi_Q)\neq \emptyset\}.
\]

We set 
\[
\CStab^{G,\zeta}_{\Pi_Q,\theta}\coloneqq \bigoplus_{\dd\in E}\nnIC_{\Msp^{G,\zeta\sst}_{\dd}(\Pi_Q)};\quad\quad
\fstab^{\SP,G,\zeta}_{\Pi_Q,\theta}\coloneqq \HO\!\varpi'_{\red,*}\varpi'^!_{\red}\CStab^{G,\zeta}_{\Pi_Q,\theta}.
\]
We consider $\fstab^{\SP,G,\zeta}_{\Pi_Q,\theta}$ as a $\HG$-module below.  By \S \ref{genSec} there are inclusions $\CStab^{G,\zeta}_{\Pi_Q,\theta}\hookrightarrow \BPSh^{G,\zeta}_{\Pi_Q,\theta}$.  In the case in which the stability condition is trivial and $\theta=0$, $\HO$ of this inclusion is the inclusion of all the real simple roots, as well as cuspidal cohomology.  To cover isotropic generators, we define 
\[
U=\{n\cdot \dd\colon (\dd,\dd)_Q=0, \;\dd\in \dvst, \;n\geq 2\}.
\]
By \cite{CB01}, $U\cap E=\emptyset$.  Arguing as in \S \ref{AffineSec}, for primitive $\dd\in U$ and $n\geq 1$ there are embeddings
\[
\Delta_{n,\dd,*}\nnIC_{\Msp^{G,\zeta\sst}_{\dd}(\Pi_Q)}\hookrightarrow \BPSh^{G,\zeta}_{\Pi_Q,\theta},
\]
where $\Delta_{n,\dd}\colon \Msp^{G,\zeta\sst}_{\dd}(\Pi_Q)\hookrightarrow \Msp^{G,\zeta\sst}_{n\cdot \dd}(\Pi_Q)$ is the diagonal embedding.  Accordingly, we define
\[
\CIsot^{G,\zeta}_{\Pi_Q,\theta}\coloneqq \bigoplus_{\substack{n\cdot \dd\in U\\ \dd\textrm{ primitive}}}\Delta_{n,\dd,*}\nnIC_{\Msp^{G,\zeta\sst}_{\dd}(\Pi_Q)};\quad\quad
\fisot^{\SP,G,\zeta}_{\Pi_Q,\theta}\coloneqq \HO\!\varpi'_{\red,*}\varpi'^!_{\red}\CIsot^{G,\zeta}_{\Pi_Q,\theta}.
\]
We can now state the main conjecture\footnote{Now a theorem, see \cite{DHSM23}.}:
\begin{conjecture}
\label{mainConj}
The above inclusions extend to an isomorphism of Lie algebra objects in $\MHM^G(\Msp^{\zeta\sst}_{\theta}(\ol{Q}))$
\[
\Borch^+\!\left(\CStab^{G,\zeta}_{\Pi_Q,\theta}\oplus \CIsot^{G,\zeta}_{\Pi_Q,\theta}\right)\cong \BPSh_{\Pi_Q,\theta}^{G,\zeta}.
\]
Applying $\HO\!\varpi'_{\red,*}\varpi'^!_{\red}$, we obtain isomorphisms of $\HG$-linear Lie algebras $\Borch^+\!\left(\fstab^{\SP,G,\zeta}_{\Pi_Q,\theta}\oplus \fisot^{\SP,G,\zeta}_{\Pi_Q,\theta}\right)\cong\fg^{\SP,G,\zeta}_{\Pi_Q,\theta}$.
\end{conjecture}
Setting $\zeta=(0,\ldots,0), \theta=0,\SP=\CC\ol{Q}\lmod$ and $G=\{1\}$ this conjecture implies the Bozec--Schiffmann conjecture on the Kac polynomials for $Q$, as well as giving a precise interpretation for the cuspidal cohomology in this case.

\bibliographystyle{alpha}
\bibliography{Literatur}

\vfill

\textsc{\small B. Davison: School of Mathematics, University of Edinburgh, James Clerk Maxwell Building, Peter Guthrie Tait Road, King's Buildings, Edinburgh EH9 3FD, United Kingdom}\\
\textit{\small E-mail address:} \texttt{\small ben.davison@ed.ac.uk}\\

\end{document}